\documentclass[10pt]{amsart}
\usepackage{egastyle}

\usepackage{amssymb,amsmath,amscd,amsthm,amsfonts,wasysym,mathrsfs,amsxtra,stmaryrd,array}

\usepackage{graphicx,array,url}
\usepackage[vcentermath]{genyoungtabtikz}
\usepackage{tikz}
\usepackage[section]{placeins}
\usepackage{afterpage}
\usepackage{verbatim}

\input xy
\xyoption{all}

\newcommand{\RR}{\mathbb{R}}
\newcommand{\OO}{\mathcal{O}}
\newcommand{\ZZ}{\mathbb{Z}}

\newcommand{\HH}{\mathbb{H}}

\newcommand{\image}{\textnormal{im}\,}

\newcommand{\degree}{\textnormal{deg}\,}
\newcommand{\Hom}{\textnormal{Hom}}
\newcommand{\dimension}{\textnormal{dim}\,}

\newcommand{\rank}{\textnormal{rk}\,}

\newcommand{\Ext}{\textnormal{Ext}}

\newcommand{\Ac}{\mathcal{A}}
\newcommand{\Bc}{\mathcal{B}}
\newcommand{\Cc}{\mathcal{C}}
\newcommand{\Fc}{\mathcal{F}}
\newcommand{\Tc}{\mathcal{T}}
\newcommand{\Dc}{\mathcal{D}}
\newcommand{\Hc}{\mathcal{H}}
\newcommand{\Pc}{\mathcal{P}}
\newcommand{\Uc}{\mathcal{U}}
\newcommand{\Coh}{\mathrm{Coh}}

\newcommand{\arinj}{\ar@{^{(}->}}
\newcommand{\arsurj}{\ar@{->>}}
\newcommand{\areq}{\ar@{=}}

\newcommand{\muob}{\mu_{\omega,B}}
\newcommand{\Bob}{\mathcal B_{\omega, B}}
\newcommand{\Tob}{\mathcal T_{\omega, B}}
\newcommand{\Fob}{\mathcal F_{\omega, B}}

\newcommand{\wh}{\widehat}

\newcommand{\Bl}{\mathcal{B}^l}

\newcommand{\ch}{\mathrm{ch}}

\newcommand{\Stab}{\mathrm{Stab}}
\newcommand{\Stabd}{\mathrm{Stab}^\dagger}
\newcommand{\Stabpol}{\mathrm{Stab}_{\text{Pol}}}
\newcommand{\Aut}{\mathrm{Aut}}

\newcommand{\whPhi}{{\wh{\Phi}}}

\newcommand{\wt}{\widetilde}
\newcommand{\ol}{\overline}

\newcommand{\olw}{{\overline{\omega}}}
\newcommand{\olB}{{\overline{B}}}
\newcommand{\bsm}{\begin{smallmatrix}}
\newcommand{\esm}{\end{smallmatrix}}

\newcommand{\Loov}{{(\!(\tfrac{1}{v})\!)}}

\newcommand{\CLoovc}{{\mathbb{C}(\!(\tfrac{1}{v})\!)^c}}
\newcommand{\RLoovc}{{\mathbb{R}(\!(\tfrac{1}{v})\!)^c}}
\newcommand{\Loovc}{{(\!(\tfrac{1}{v})\!)^c}}

\newcommand{\RPwc}{{\mathbb{R}\llbracket w \rrbracket^c}}

\newcommand{\CPwc}{{\mathbb{C}\llbracket w \rrbracket^c}}

\newcommand{\GLlp}{{\mathrm{GL}^{l,+}\!(2,\RLoovc)}}
\newcommand{\GLtr}{{\mathrm{GL}^+(2,\mathbb{R})}}
\newcommand{\sgn}{{\mathrm{sgn}\,}}
\newcommand{\Patilde}{\wt{\mathrm{Pa}}}

\makeatletter
\newtheorem*{rep@theorem}{\rep@title}
\newcommand{\newreptheorem}[2]{%
\newenvironment{rep#1}[1]{%
 \def\rep@title{#2 \ref{##1}}%
 \begin{rep@theorem}}%
 {\end{rep@theorem}}}
\makeatother

\newenvironment{que}[1][]%
  {\begin{genthm}{Question}{true}{#1}{paragraph}}%
  {\end{genthm}}

  {\begin{genthm}{Question}{true}{#1}{equation}}%
  {\end{genthm}}

\newreptheorem{theorem}{Theorem}

\usepackage{scalerel,stackengine}
\stackMath
\newcommand\reallywidehat[1]{%
\savestack{\tmpbox}{\stretchto{%
  \scaleto{%
    \scalerel*[\widthof{\ensuremath{#1}}]{\kern-.6pt\bigwedge\kern-.6pt}%
    {\rule[-\textheight/2]{1ex}{\textheight}}
  }{\textheight}%
}{0.5ex}}%
\stackon[1pt]{#1}{\tmpbox}%
}
\begin{document}

\title[]{Weight functions, tilts, and stability conditions}

\author[Jason Lo]{Jason Lo}
\address{Department of Mathematics \\
California State University, Northridge\\
18111 Nordhoff Street\\
Northridge CA 91330 \\
USA}
\email{jason.lo@csun.edu}


\begin{abstract}
In this article, we treat stability conditions in the sense of King, Bridgeland and Bayer in a single framework.  Following King, we begin with   weight functions on a triangulated category, and  consider increasingly specialised configurations of triangulated categories, t-structures and stability functions that give equivalent categories of stable objects.  Along the way, we recover existing results in representation theory and algebraic geometry, and prove a series of new results on elliptic surfaces, including correspondence theorems for Bridgeland stability conditions and polynomial stability conditions, local finiteness and boundedness for mini-walls for Bridgeland stability conditions, isomorphisms between moduli of 1-dimensional twisted Gieseker semistable sheaves and 2-dimensional Bridgeland semistable objects, the preservation of geometric Bridgeland stability by autoequivalences on elliptic surfaces of nonzero Kodaira dimension, and solutions to Gepner equations on elliptic surfaces.
\end{abstract}

\maketitle
\tableofcontents

\section{Introduction}

\subsection{Overview} In representation theory, moduli spaces of modules are defined using  stability in the sense of King \cite{king1994moduli}, i.e.\ using a weight function.  In its original form, this concept of stability is closely related to Mumford's stability for vector bundles on a curve.  In algebraic geometry, $\mu$-stability and other types of stability defined via slope functions have  been used to study moduli spaces of coherent sheaves on varieties of any dimension.  Out of a need to understand stability arising from mathematical physics, Bridgeland coined a concept of  stability condition on triangulated categories \cite{StabTC}, the definition of which was  relaxed by Bayer to give more flexibility in understanding objects such as Pandharipande-Thomas stable pairs \cite{BayerPBSC} (see also Toda's limit stability, which can be interpreted to be  under Bayer's framework \cite{Toda1}).

In this article, we present configurations of t-structures and stability conditions that   give rise to equivalent categories of semistable objects and hence isomorphic moduli spaces of semistable objects.     Instead of understanding relations among \emph{moduli spaces} alone, which require fixing numerical classes of objects, we consider relations among notions of \emph{stability} themselves.  We also  consider stability conditions in the sense of King, Bridgeland and Bayer  as part of a single, connected framework.

In the first thread  within this framework, following King, we define a weight function on a triangulated category $\Dc$ to be a group homomorphism $S : K(\Dc) \to R$ from the Grothendieck group of $\Dc$ to a totally ordered abelian group.  Then, for any abelian subcategory $\Ac$ of $\Dc$, the weight function defines a notion of stability for objects in $\Ac$.  In a separate thread, we extend the definition of Bayer's polynomial stability by allowing central charges to take values in convergent Laurent series that are  expansions of analytic functions.  Using this extended definition of polynomial stability, a Bridgeland stability condition is a polynomial stability condition where the central charge takes values in constant functions.  In addition, slope functions for sheaf stability such as $\mu$-stability can be defined as the slope functions coming from `weak' polynomial stability functions.

These two threads  are then connected by defining a suitable weight function anytime we have a weak polynomial stability function and a fixed nonzero object - this generalises a construction that  appeared in Rudakov's work \cite{rudakov1997stability}.

We note that constructions passing between  Bridgeland stability  and King's stability, or hints of such constructions, already exist in the literature  \cite{Bayer2017,bridgeland2017scattering,dimitrov2013dynamical}.  Also, our technical results such as Theorems \ref{thm:fundamental}, \ref{thm:main3}, \ref{thm:configIII-w}(iii) and \ref{prop:paper30prop11-11ext}(iv)  can be interpreted as formal results on how stability can be "preserved" under tilting, which is also studied in \cite[Sections 14, 19]{bayer2019stability} for weak stability conditions where the central charges take values in $\mathbb{C}$.  One of the differences between our approach and that in \cite{bayer2019stability} is that we do not assume the Harder-Narasimhan property in proving the preservation of stability.

Although we will give a proper outline of the article  in Section \ref{para:intro-sum}, we include here a quick list of our main results that is more descriptive than the abstract:

Starting with weight functions alone, we  prove results describing how weight functions in certain compatible configurations give rise to the same semistable objects.    Then, via the connection between weight functions and polynomial stability functions, we prove a similar result for  compatible polynomial stability conditions.  Only homological algebra is needed up to this point.   With different inputs,  these homological statements lead to a series of results - both old and new - including:
\begin{itemize}
\item[(1)] We recover a result  in the classification of finite-dimensional algebras, which uses a tilting equivalence to establish  isomorphisms between moduli spaces of semistable modules for different weight functions (Theorem \ref{thm:chindrisA}).
\item[(2)] We recover a fundamental result on elliptic curves, that the Fourier-Mukai transform from the Poincar\'{e} line bundle preserves stability of coherent sheaves (Theorem \ref{thm:main0}).
\end{itemize}
Then, using a relative Fourier-Mukai transform $\Phi$ on Weierstra{\ss} elliptic surfaces, we prove:
\begin{itemize}
    \item[(3)] There are two polynomial stability conditions that correspond to each other via $\Phi$, one being the large volume limit (Theorem \ref{thm:main5}).
    \item[(4)] There are two families of Bridgeland stability conditions whose members  correspond to each other via $\Phi$ (Theorem \ref{thm:Bristabcorr}).
    \item[(5)] For fixed Chern classes, $\Phi$ induces a correspondence of the wall and chamber structures in the  families in (4); in each family, the walls are locally finite and bounded (Theorem \ref{cor:AG48-58-1} and proof).
    \item[(6)] In terms of the moduli spaces they define, the outermost chambers for the two families of Bridgeland stability conditions in (4) coincide with the   two polynomial stability conditions in (3) (Theorem \ref{thm:stabBriNlim}).
\end{itemize}
Applications of the results on Weierstra{\ss} elliptic surfaces include:
\begin{itemize}
\item[(7)] We construct  isomorphisms from moduli spaces of 1-dimensional twisted Gieseker semistable sheaves to  moduli spaces of Bridgeland semistable objects of nonzero rank (Corollary \ref{cor:AG48-55-1}).
\item[(8)] We prove that the autoequivalence group preserves the geometric component of the Bridgeland stability manifold when the elliptic surface has nonzero Kodaira dimension (Theorem \ref{thm:AG48-67-1}).
\item[(9)] We construct Gepner-type Bridgeland stability conditions (Example  \ref{eg:Gepnerpoints}).
\end{itemize}


\paragraph[Outline and summary] We \label{para:intro-sum} now describe the outline of this article and give a more detailed summary of our main results.

After fixing some notations in Section \ref{sec:prelim}, we begin Section \ref{sec:weightfnc} with the definition of a weight function on a triangulated category in \ref{para:weightfuncdef}.  We describe Configuration I in \ref{para:config3}, a setting where weight functions satisfy sign compatibility, and introduce the refinement property in \ref{para:WPdef3}, which generalises the concept of `well-positioned' functions by Chindris.  Our first theorem, Theorem \ref{thm:fundamental},  states that given two weight functions $S, S'$ on a triangulated category in Configuration I where $S$ satisfies the refinement property, an object is  $S$-semistable if and only if it is  $(-1)^iS'$-semistable after a shift by $i$.  Then, by bringing exact equivalences of triangulated categories into the picture, we describe Configuration II in \ref{para:config2} and prove Theorem \ref{thm:main3} as a counterpart of Theorem \ref{thm:fundamental}.  We end Section 3 by using Theorem \ref{thm:main3} to recover a theorem of Chindris' \cite[Theorem 1.3(a)]{chindris2013} (Theorem \ref{thm:chindrisA}), which identifies two moduli spaces of semistable modules from different weights, a result integral to the classification of  wild tilted algebras.

In Section \ref{sec:polystab}, we extend Bayer's notion of polynomial stability using   convergent Laurent series, and  consider Bridgeland stability conditions as polynomial stability conditions.  We also describe a process for producing weight functions using a weak polynomial stability function and a nonzero object (Lemma \ref{lem:S-stabfuncdef}), which generalises a result of Rudakov's \cite[Proposition 3.4]{rudakov1997stability}.

In Section \ref{sec:configIII}, we modify Configuration II to Configuration III by replacing weight functions with polynomial stability functions, the latter of which are more suitable for applications in algebraic geometry than in representation theory. Configuration III involves an exact equivalence $\Phi : \Dc \to \Uc$ of triangulated categories, hearts of t-structures $\Ac, \Bc$ on $\Dc, \Uc$, respectively, with  polynomial stability functions $Z_\Ac, Z_\Bc$ on $\Ac, \Bc$, and a linear operation $T$ that accounts for the difference between $Z_\Bc$ and the image of $Z_\Ac$ under the action of $\Phi$.  In  Theorem \ref{prop:paper30prop11-11ext}, we prove that under Configuration III, $Z_\Ac$-semistable objects  correspond to $Z_\Bc$-semistable objects up to a shift - this is without assuming the Harder-Narasimhan property for the stability functions involved.  When the polynomial stability functions $Z_\Ac, Z_\Bc$ do have the Harder-Narasimhan property,  we obtain the equation
\begin{equation}\label{eq:intro2}
  \Phi \cdot (Z_\Ac, \Pc_\Ac) = (Z_\Bc, \Pc_\Bc)\cdot \Patilde (T)
\end{equation}
where $\Pc_\Ac, \Pc_\Bc$ denote the slicings of the respective polynomial stability conditions and $\Patilde (T)$ is a path in   $\wt{\mathrm{GL}}^+\!(2,\mathbb{R})$, the universal cover of $\mathrm{GL}^+\!(2,\mathbb{R})$.

In Configurations I, II and III, the idea is that once all the required inputs (of exact equivalences, stability functions, and t-structures) are in place, the output is a correspondence theorem of semistable objects or an equation for  stability conditions analogous to \eqref{eq:intro2}.  The remainder of this article concerns   applications of Theorem \ref{prop:paper30prop11-11ext} on elliptic curves and elliptic surfaces.

Section \ref{sec:prelim-ellfib} sets up the necessary terminology for Weierstra{\ss} elliptic fibrations $X$ and introduces a fundamental autoequivalence $\Phi$ of $D^b(X)$, which is a relative Fourier-Mukai transform.

Section \ref{sec:ellcur} contains our first application of  Theorem \ref{prop:paper30prop11-11ext}:  We recover, in the form of Theorem \ref{thm:main0},  the well-known result  that under the Fourier-Mukai transform on an elliptic curve whose kernel is the Poincar\'{e} line bundle, a semistable sheaf is taken to a semistable sheaf.  Previous expositions on this fundamental result on elliptic curves can be found in  Bridgeland \cite[3.2]{FMTes}, Hein-Ploog \cite{hein2005fourier}, Polishchuk \cite[Lemma 14.6]{polishchuk2003abelian}, and  \cite[Corollary 3.29]{FMNT}.

Section \ref{sec:ellsurfGLact} considers the cohomological formula for the autoequivalence $\Phi$, and rewrites it in a simpler form using  $\mathrm{GL}^+\!(2,\mathbb{R})$-actions.  In Section \ref{sec:polytwststab}, we recall the large volume limit on surfaces - a polynomial stability condition we denote as $Z_l$-stability -  and check that it coincides with twisted Gieseker stability (a stability defined using a slope function) in the case of 1-dimensional sheaves.

In Section \ref{sec:ellsurfpolystab}, we solve a system of polynomial equations, which in turn  allows us to "patch"  central charges together up to the autoequivalence $\Phi$.  The solutions can be expressed as convergent Laurent series, which lead to the use of polynomial stability conditions.  We also recall $Z^l$-stability, a polynomial stability condition first defined in the author's joint work with Liu and Martinez \cite{LLM} by deforming the ample class along a hyperbola in the ample cone. Then we prove a polynomial-stability generalisation of Theorem \ref{thm:main0} on elliptic surfaces $X$:  an object $E$ in the derived category $D^b(X)$ is  $Z_l$-semistable if and only if $\Phi E$ is $Z^l$-semistable; moreover, the equation \eqref{eq:intro2} specialises to a correspondence between $Z_l$-stability and $Z^l$-stability via $\Phi$ (Theorem \ref{thm:main5}).  This is the second application of  Theorem \ref{prop:paper30prop11-11ext}.

On the other hand, whenever there is a solution to  the patching equations in  Section \ref{sec:ellsurfpolystab}, we obtain a generalisation of Theorem \ref{thm:main0} to Bridgeland stability on elliptic surfaces: for  ample divisors $\olw, \omega$ and $\mathbb{R}$-divisors $\olB, B$ that satisfy the patching equations, an object $E$ in $D^b(X)$ is $Z_{\olw,\olB}$-semistable if and only if $\Phi E$ is $Z_{\omega, B}$-semistable (Theorem \ref{thm:Bristabcorr}).  We also prove  the stronger statement
\begin{equation}\label{eq:intro1}
  ([1]\circ \Phi)\cdot \sigma_{\olw,\olB}= \sigma_{\omega,B} \cdot (T,g)
\end{equation}
where $\sigma_{\olw,\olB}, \sigma_{\omega,B}$ are Bridgeland stability conditions with central charges $Z_{\olw,\olB}, Z_{\omega, B}$, respectively.  In this equation, $[1]$ denotes the shift functor in $D^b(X)$, and $(T,g)$ is an element of $\wt{\mathrm{GL}}^+\!(2,\mathbb{R})$.  Our approach gives explicit coordinates for $\olw, \omega, \olB$ and $B$.

Another way to interpret equation \ref{para:gpactionBri} is: for suitable $\olw,\olB$, the image of the Bridgeland stability condition $\sigma_{\olw,\olB}$ under the autoequivalence $[1]\circ \Phi$ is, up to a $\wt{\mathrm{GL}}^+\!(2,\mathbb{R})$-action (which simply relabels the phases of semistable objects),  again of the  form $\sigma_{\omega,B}$.  Moreover,  we can compute $\omega,B$ explicitly.

At this point, we have constructed  a clear picture of the behaviours of moduli spaces of Bridgeland stability conditions under the autoequivalence $\Phi$ on a Weierstra{\ss} elliptic surface $X$ - this is summarised in the form of Theorem \ref{thm:stabBriNlim} and diagram \eqref{eq:stabequivs}.  Roughly speaking, there is a countable number of rays in the space $\mathrm{Stab}(X)$ of Bridgeland stability conditions on $D^b(X)$, each of which is partnered with a  hyperbola in $\mathrm{Stab}(X)$.  Sufficiently away from the origin, stability conditions on each ray corresponds to stability conditions on its partner hyperbola.  As a result, the local finiteness and boundedness of mini-walls on the rays in the author's joint work with Qin \cite{LQ} give corresponding local finiteness and boundedness results of mini-walls on the hyperbolas (Theorem \ref {cor:AG48-58-1}).

Moreover, when a Chern character is fixed, the Bridgeland moduli space in the outer-most mini-chamber on a ray, which corresponds to $Z_l$-moduli space (proved in \cite{LQ}), now corresponds via $\Phi$ to the Bridgeland moduli space in the outer-most mini-chamber on its partner hyperbola, which corresponds to $Z^l$-moduli space.  A consequence of this picture is,  that for  Chern classes of 1-dimensional sheaves, the moduli space of twisted Gieseker semistable 1-dimensional sheaves is isomorphic to a moduli space of Bridgeland semistable objects of nonzero rank (Corollary \ref{prop:LLM1}).  This improves a previous result in the author's joint work with Liu and Martinez, in which an embedding of moduli spaces was obtained via a different, more direct wall-crossing result  \cite[Corollary 6.14]{LLM}.

We conclude the paper with two applications of Theorem \ref{thm:Bristabcorr} that are concerned with  Bridgeland stability manifolds for surfaces.

The first application is related to a conjecture \cite[Conjecture 1.2]{SCK3} of Bridgeland's on the structure of the autoequivalence group $\mathrm{Aut}(D^b(X))$ when $X$ is a K3 surface.  The first part of the conjecture asserts that the autoequivalence group preserves the connected component of $\mathrm{Stab}(X)$ containing geometric stability conditions.  We show that the corresponding statement holds for Weierstra{\ss} elliptic surfaces of non-zero Kodaira dimensions.

The second application concerns Gepner-type Bridgeland stability conditions in the sense of Toda \cite{toda2013gepner}.  A Gepner-type Bridgeland stability condition on a smooth projective variety is an orbifold point of $\mathrm{Aut}(D^b(X)) \backslash \mathrm{Stab}(X)/\mathbb{C}$.  When $X$ is a Calabi-Yau threefold, such stability conditions mimic stability conditions on Orlov's triangulated category of graded matrix factorisations, and can impose constraints and reveal symmetries in Donaldson-Thomas invariants on $X$.  Using Theorem \ref{thm:Bristabcorr},  we construct  Gepner points in $\mathrm{Stab}(X)$ for Weierstra{\ss} elliptic surfaces $X$ (Theorem \ref{thm:Gepnersol} and Example \ref{eg:Gepnerpoints}).

\paragraph[Acknowledgements]  The author thanks Conan Leung, Yu-Shen Lin, Wu-Yen Chuang, Cristian Martinez,  Wanmin Liu, Zhenbo Qin, and Ziyu Zhang for helpful discussions during various stages of this project.   He would also like to thank Arend Bayer for pointing out a gap in an earlier version of this work and suggesting a fix.

\section{Preliminaries}\label{sec:prelim}

\paragraph[Notation] For any abelian category $\Ac$, we will write $D^b(\Ac)$ for the bounded derived category of $\Ac$.  When $X$ is a smooth projective variety, we will write $\Coh (X)$ to denote the category of coherent sheaves on $X$, and  write $D^b(X)$ to denote $D^b(\Coh (X))$.  For any divisor class $B$ on $X$ and any $E \in D^b(X)$, we will write $\ch^B(E)$ to denote the twisted Chern character $e^{-B}\ch(E)$ of $E$.  That is, we have $\ch_0^B(E)=\ch_0(E), \ch_1^B(E)=\ch_1(E)-B\ch_0(E)$, etc.

\subparagraph Unless otherwise stated, all the t-structures on triangulated categories that appear in this article will be bounded.

\subparagraph Suppose $\Tc$ is a triangulated category and $\Ac$ is the heart of a t-structure on $\Tc$; then we  write $\Hc_\Ac^i : \Tc \to \Ac$ to denote the $i$-th cohomology functor with respect to the t-structure.  When $\Tc$ is  the bounded derived category $D^b(\Bc)$ of an abelian category $\Bc$, we write $H^i$ instead of $\Hc_\Bc^i$, i.e.\  $H^i$ denotes the $i$-th cohomology functor on $D^b(\Bc)$ with respect to the standard t-structure.  When the triangulated category $\Tc$ is understood, for any integers $a \leq b$ we  define the full subcategory of $\Tc$
\[
  D^{[a,b]}_{\Bc} := \{ E \in \Tc : \Hc^i_\Bc (E) =0 \text{ for all } i \notin [a,b]\}.
\]

\subparagraph[WIT$_i$] Let \label{para:def-WIT} $\Phi : \Tc \to \Uc$ be an exact equivalence of triangulated categories.
\begin{itemize}
\item  Suppose $\Bc$ is the heart of a t-structure on $\Uc$.  We will  write $\Phi_\Bc^i$ to denote the composite functor $\Hc_\Bc^i \circ \Phi : \Tc \to \Bc$.  We will say an object $E \in \Tc$ is $\Phi_\Bc$-WIT$_i$ if $\Phi E$ is isomorphic  to an object in $\Bc [-i]$, i.e.\ if $\Phi_\Bc^j (E)=0$ for all $j \neq i$.
\item Suppose $\Ac, \Bc$ are hearts of bounded t-structures in $\Tc, \Uc$, respectively.    For any integer $i$, we   define the  full subcategory of $\Ac$
\begin{equation*}
  W_{i,\Phi,  \Ac, \Bc} := \{ E \in \Ac : E \text{ is $\Phi_\Bc$-WIT$_i$}\}.
\end{equation*}
\item When $\Uc=D^b(\Bc)$ is the bounded derived category of an abelian category $\Bc$, we simply write $\Phi$-WIT$_i$ instead of $\Phi_\Bc$-WIT$_i$.
\item When $\Tc = D^b(\Ac), \Uc = D^b(\Bc)$ are bounded derived categories of abelian categories $\Ac, \Bc$, respectively, we simply write $W_{i,\Phi}$ for $W_{i,\Phi,  \Ac, \Bc}$, and we  write $\wh{E}$ to denote $\Phi E [i]$ when $E$ is a $\Phi$-WIT$_i$ object in $\Ac$.
\end{itemize}

\paragraph[Torsion pair and tilting] Given an  abelian category $\Ac$, a pair of full subcategories $(\Tc, \Fc)$ of $\Ac$ will be called a torsion pair in $\Ac$ if they satisfy the following two conditions:
\begin{itemize}
\item[(a)] For any $T \in \Tc$ and $F \in \Fc$, we have $\Hom_\Ac (T,F)=0$.
\item[(b)] Every $E \in \Ac$ fits in some short exact sequence in $\Ac$
\[
 0 \to E' \to E \to E'' \to 0
\]
where $E' \in\Tc$ and $E'' \in \Fc$.
\end{itemize}
Whenever $(\Tc, \Fc)$ is a torsion pair in $\Ac$, we will refer to $\Tc$ (resp.\ $\Fc$) as a \emph{torsion class} (resp.\ \emph{torsion-free class}) in $\Ac$.  When $\Ac$ is a noetherian abelian category, a full subcategory of $\Ac$ is a torsion class if and only if it is closed under extension and quotient in $\Ac$ \cite[Lemma 1.1.3]{Pol}.  Given an abelian category $\Ac$ and a torsion pair $(\Tc, \Fc)$ in $\Ac$, the extension closure in $D^b(\Ac)$
\[
  \Ac^\dagger = \langle \Fc [1], \Tc\rangle
\]
is again the heart of a t-structure on $D^b(\Ac)$, which we refer to as the \emph{tilt} of the heart $\Ac$ at the torsion pair $(\Tc, \Fc)$.

\paragraph Given an abelian category $\Ac$ and a collection of full subcategories $\Cc_1, \cdots, \Cc_n$ of $\Ac$, we will write $[\Cc_1, \cdots, \Cc_n]$ to denote the full subcategory of all objects $E$ in $\Ac$ that admit a filtration of the form
\begin{equation}\label{eq:ntuplefiltration}
  0=: E_0 \subseteq E_1 \subseteq E_2 \subseteq \cdots \subseteq E_n = E
\end{equation}
where $E_i/E_{i-1} \in \Cc_i$ for each $1 \leq i \leq n$.

\paragraph[Torsion $n$-tuple] Following \cite[Section 2.2]{Pol2} (see also  \cite[Definition 3.5]{Toda2}), given an abelian category $\Ac$, we will say an $n$-tuple  $(\Cc_1, \cdots, \Cc_n)$ of full subcategories of  $\Ac$ is a \emph{torsion $n$-tuple in $\Ac$} if they satisfy the following two conditions:
\begin{itemize}
\item[(a)] For any $E' \in \Cc_i$ and $E'' \in \Cc_j$ where $i<j$, we have $\Hom_\Ac (E',E'')=0$.
\item[(b)] $[\Cc_1, \cdots, \Cc_n] = \Ac$.
\end{itemize}
 When $(\Cc_1,\cdots, \Cc_n)$ is a torsion $n$-tuple in $\Ac$, it is clear that for each $1 \leq i \leq n-1$, the pair $([ \Cc_1, \cdots, \Cc_i],[ \Cc_{i+1}, \cdots, \Cc_n ])$ is a torsion pair in $\Ac$.  As a result, the filtration \eqref{eq:ntuplefiltration} for the object $E$ is canonical, and we will refer to the subfactor $E_i/E_{i-1}$  as the $\Cc_i$-component of $E$, and refer to \eqref{eq:ntuplefiltration} as the $(\Cc_1,\cdots,\Cc_n)$-decomposition of $E$.  When $n=2$, the definition of a torsion $n$-tuple reduces to that of a torsion pair.

\paragraph[Subcategories of $\Coh (X)$]  Let $X$ be a smooth projective variety.  For any integers $0 \leq e \leq d$, we define the full subcategories of $\Coh (X)$
\begin{align*}
\Coh^{\leq d}(X) &= \{ E \in \Coh (X): \dimension \mathrm{supp}(E) \leq d\} \\
\Coh^{\geq d}(X) &= \{ E \in \Coh (X): \Hom_{\Coh (X)} (F,E)=0 \text{ for all }F \in \Coh^{\leq d-1}(X)\}\\
\Coh^{=d}(X) &= \Coh^{\leq d}(X) \cap \Coh^{\geq d}(X).
\end{align*}
Given a morphism of smooth projective varieties $p : X \to B$, we also define the full subcategories of $\Coh (X)$
\begin{align*}
\Coh (p)_{\leq d} &= \{ E \in \Coh (X): \dimension p (\mathrm{supp}(E)) \leq d\} \\
\Coh (p)_0 &= \{ E \in \Coh (X) : \dimension p (\mathrm{supp}(E)) = 0\} \\
\{ \Coh^{\leq 0} \}^\uparrow &= \{ E \in \Coh (X): E|_b \in \Coh^{\leq 0} (X_b) \text{ for all closed points $b \in B$} \}
\end{align*}
where  $\Coh^{\leq 0}(X_b)$ denotes the category of coherent sheaves $F$ supported  on the fiber $p^{-1}(b)=X_b$, with $b$ being  a closed point of $B$, such that $F$ is supported in dimension 0.  We will also refer to coherent sheaves on $X$ that are supported on a finite number of fibers of $p$  as \emph{fiber sheaves} on $X$, i.e.\   $\Coh (p)_0$ is  the category of fiber sheaves on $X$.

\section{Weight functions, stability,  Configurations I and II}\label{sec:weightfnc}

In this section, we focus on stability arising from weight functions.  We give two configurations of weight functions, and show that together with the `refinement property', weight functions in these configurations define equivalent categories of semistable objects.

\paragraph[Weight functions and King's stability] Suppose \label{para:weightfuncdef} $(R,\preceq)$ is a totally ordered abelian group under addition $+$, with $0$ as the identity element.  That is, $R$ is an abelian group  such that $(R,\preceq)$ is a totally ordered set satisfying the translation invariance property: for any $a, b, c \in R$, we have $a \preceq b$ implies $a+c \preceq b+c$.  Then for any triangulated category $\Dc$, any bounded t-structure on $\Dc$ with heart  $\Ac$,  and any group homomorphism $S : K(\Dc) \to R$, we  define a notion of semistability on $\Ac$ by declaring a nonzero object $E \in \Dc$ to be an \emph{$S$-semistable} (resp.\ \emph{$S$-stable}) \emph{object in $\Ac$}, or simply \emph{$S$-semistable in $\Ac$} (resp.\ \emph{$S$-stable in $\Ac$}), if $E \in \Ac$ and
\begin{itemize}
\item[(i)] $S(E)=0$;
\item[(ii)] $S(E') \preceq 0$ (resp.\ $S(E') \prec 0$) for all nonzero proper subobjects $E'$ of $E$ in $\Ac$.
\end{itemize}
 We will  refer to $S$ as a \emph{weight function} on the triangulated category $\Dc$.  Note that the zero object is always $S$-semistable in $\Ac$.  Also, given condition (i), condition (ii) for $S$-semistable (resp.\ $S$-stable) objects is equivalent to the following  by the translation invariance of $\preceq$:
\begin{itemize}
\item[(ii')] $0 \preceq S(E'')$ (resp.\ $0 \prec S(E'')$) for all nonzero proper quotients $E''$ of $E$ in $\Ac$.
\end{itemize}

\begin{egsub}
For any $\Dc, \Ac, (R,\preceq)$ as above,  if we take the weight function $S : K(\Dc) \to R$  to be the zero function, then every object in the abelian category $\Ac$ is $S$-semistable.
\end{egsub}

\begin{egsub}
Given any abelian category $\Ac$, we can take $\Dc=D^b(\Ac)$ and take $(R,\preceq)=(\mathbb{R},\leq)$.  Then $S$-semistability is King's  semistability on an abelian category as originally defined in \cite{king1994moduli}, with the opposite sign.
\end{egsub}

\begin{egsub}\label{eg:smprojcurv}
Let $X$ be a smooth projective curve, $\Ac = \Coh (X)$ and $\Dc = D^b(X)$.  Fix any nonzero object $M$ in $\Ac$ and consider the weight function $S : K(X) \to \mathbb{Z}$ given by
\[
  S(E)= (\rank M)(\degree E) - (\degree M)(\rank E).
\]
On the other hand, we have the usual slope function on $\Coh (X)$ defined by $\mu (E)=\degree E/\rank (E)$ (taken to be $\infty$ if $\rank E=0$).  One can check that for any nonzero coherent sheaf $E$, the following are equivalent:
\begin{itemize}
\item   $E$  is $S$-semistable in $\Ac$;
\item  $\mu (E)=\mu (M)$ and $E$ is a $\mu$-semistable sheaf.
\end{itemize}
This is an example of  Lemma \ref{lem:S-stabfuncdef} where we take $Z : K(\Dc) \to \mathbb{C}$ to be $Z(E)= -\degree E + i \rank E$.
\end{egsub}

\subparagraph[Signs] Given a totally ordered abelian group $(R,\preceq)$, we define the sign function $\sgn : R \to \{0, \pm 1\}$ in the obvious manner, i.e.\ for $r \in R$ we set $\sgn (r)$ to be $1, 0$ or $-1$ if $r \succ 0$, $r=0$ or $r \prec 0$, respectively.

\paragraph[Configuration I]  We   \label{para:config3} say we are in Configuration I when we have the following setting:  $\Dc$ is a triangulated category, $\Ac$ is the heart of a bounded t-structure on $\Dc$, and $(\Tc, \Fc)$ is a torsion pair in $\Ac$.  Additionally, $(R,\preceq)$ is a totally ordered abelian group and $S, S' : K(\Dc) \to R$ are group homomorphisms satisfying the sign compatibility
\begin{itemize}
\item[(S)] For any $G$ lying in either $\Tc$ or $\Fc$, we have $\sgn S(G) = \sgn S'(G)$.
\end{itemize}

Note that when $\Dc, \Ac$ and $(\Tc, \Fc)$ are fixed, the statement of Configuration I is symmetric in $S$ and $S'$.

\paragraph[Refinement property] Let  \label{para:WPdef3} $\Dc$ be a triangulated category,  $\Ac$ the heart of  a bounded  t-structure on $\Dc$,   and $(\Tc, \Fc)$ a torsion pair in   $\Ac$.  Suppose  $\Ac^\dagger = \langle \Fc [1], \Tc\rangle$ is the corresponding tilted heart.  Given a totally ordered abelian group $(R,\preceq)$ and for $i=0, 1$, we say a weight function $S : K(\Dc) \to R$ satisfies \emph{refinement-$i$ with respect to $(\Tc, \Fc)$}  if:
\begin{itemize}
\item[] Every nonzero $S$-semistable object $E$ in $\Ac$ lies in $\Ac \cap (\Ac^\dagger [-i])$, and for every nonzero object $G \in \Ac \cap (\Ac^\dagger [i-1])$ we have $(-1)^i \sgn (S(G)) =-1$.
\end{itemize}
More explicitly, a weight function $S$ satisfies   refinement-$0$ or refinement-$1$ when:
\begin{itemize}
\item (Refinement-$0$) Every nonzero $S$-semistable object $E$ in $\Ac$ lies in $\Tc$ (so that $E[0] \in \Ac^\dagger$) , and for every nonzero object $G \in \Fc$ we have $S(G)  \prec 0$.
\item (Refinement-$1$) Every nonzero $S$-semistable object $E$ in $\Ac$ lies in $\Fc$ (so that $E[1] \in \Ac^\dagger$), and for every nonzero object $G \in \Tc$ we have  $S(G) \succ 0$.
\end{itemize}
We also say a weight function satisfies the refinement property  if it satisfies either refinement-$0$ or refinement-$1$.

\begin{lem}\label{lem:refinement-sym}
Assume Configuration I.  For $i=0$ or $1$, if $S$ satisfies refinement-$i$ with respect to $(\Tc, \Fc)$, then $(-1)^iS'$ satisfies refinement-$(1-i)$ with respect to $(\Fc [1], \Tc)$, considered as a torsion pair in the heart $\Ac^\dagger =\langle \Fc[1], \Tc\rangle$.
\end{lem}

\begin{proof}
(a) The case $i=0$: Suppose $S$ satisfies refinement-$0$ with respect to $(\Tc, \Fc)$.  To show that $S'$ satisfies refinement-$1$ with respect to $(\Fc [1], \Tc)$, take any nonzero $S'$-semistable object $E$ in $\Ac^\dagger$ and consider the $(\Fc [1], \Tc)$-decomposition of $E$ in $\Ac^\dagger$
\[
0 \to E' \to E \to E'' \to 0
\]
where $E' \in \Fc [1]$ and $E'' \in \Tc$.  By the $S'$-semistability of $E$ in $\Ac^\dagger$, we have $-S'(E'[-1])=S'(E') \preceq S'(E)=0$.  On the other hand, since $E'[-1] \in \Fc$, the refinement-$0$ property of $S$ implies that $S(E'[-1]) \preceq 0$ with equality if and only if $E'=0$.  By condition \ref{para:config3}(S), we have $S'(E'[-1])\preceq 0$, i.e.\ $-S'(E'[-1]) \succeq 0$.  Hence $S'(E'[-1])=0$ overall, forcing $E'=0$, i.e.\ $E \cong E''\in \Tc$.  Then we have $S(E)=0$ by \ref{para:config3}(S) again.

It remains to show that $S'(G) \succ 0$ for every nonzero object $G \in \Fc [1]$.  Take any nonzero object $G \in \Fc [1]$.  Then $G[-1] \in \Fc$ and so $S(G[-1]) \prec 0$ by the refinement-$0$ property of $S$.  Then $S'(G[-1]) \prec 0$ by \ref{para:config3}(S) and so $S'(G) \succ 0$ as wanted.

(b) The case $i=1$: Even though the argument is analogous to that for $i=0$, we give the full argument to show how the sign changes are compatible with the tilt.  Suppose $S$ satisfies refinement-$1$ with respect to $(\Tc,\Fc)$.  To show that $(-S')$ satisfies refinement-$0$ with respect to $(\Fc [1], \Tc)$, take any nonzero $(-S')$-semistable object $E$ in $\Ac^\dagger$ and consider its $(\Fc [1], \Tc)$-decomposition in  $\Ac^\dagger$
\[
0 \to E' \to E \to E'' \to 0
\]
where $E' \in \Fc [1]$ and $E'' \in \Tc$.  Since $S$ satisfies refinement-$1$, we have $S(E'') \succeq 0$, with equality if and only if $E''=0$; this implies $S'(E'')\succeq 0$ by \ref{para:config3}(S).  On the other hand, by the $(-S')$-semistability of $E$, we have $(-S')(E'')\succeq 0$.  Hence $S'(E'')=0$ overall and $E''=0$, i.e.\ $E\cong E' \in \Fc [1]$.

It remains to show that $(-S')(G) \prec 0$ for every nonzero object $G \in \Tc$.  Take any nonzero object $G \in \Tc$.  Then $S(G) \succ 0$ since $S$ satisfies refinement-$1$.  Hence $S'(G) \succ 0$ by \ref{para:config3}(S) and so $(-S')(G) \prec 0$ as wanted.
\end{proof}

Lemma \ref{lem:AG48-p7-D2} shows, that if $S$ and $S'$ are weight functions with sign compatibility as in Configuration I, then the refinement property on $S$ ensures that $S$-semistability implies $S'$-semistability under a suitable shift and sign change.

\begin{lem}\label{lem:AG48-p7-D2}
Assume Configuration I.  For $i=0$ or $1$, suppose $S$ satisfies refinement-$i$ with respect to $(\Tc, \Fc)$.  If $E$ is $S$-semistable  in $\Ac$, then $E[i]$ is $(-1)^iS'$-semistable  in the heart $\Ac^\dagger =\langle \Fc[1], \Tc\rangle$. The same statement holds if we replace `semistable' by `stable.'
\end{lem}

\begin{proof}
(a) The case $i=0$: Suppose $S$ satisfies refinement-$0$ and $E$ is a nonzero $S$-semistable object in $\Ac$.   Then  $E \in \Tc \subset {\Ac^{\dagger}}$ and $S(E)=0$.  By \ref{para:config3}(S) we have $S'(E)=0$.  Now take any ${\Ac^{\dagger}}$-short exact sequence $0 \to M \to E \to N \to 0$ where $M, N \neq 0$; considering it as an exact triangle in $\Dc$, we can take its $\Ac$-long exact sequence
\[
 0 \to \Hc^{-1}_\Ac (N) \to M \overset{\alpha}{\to} E \to \Hc^0_\Ac (N) \to 0.
\]
Note that $\Hc^{-1}_\Ac (M)=0$ and so $M$ lies in $\Tc$.  By the $S$-semistability of $E$, we have $S(\image\alpha ) \preceq 0$.  Since $\Hc^{-1}_\Ac (N) \in \Fc$, the refinement-$0$ property of $S$ implies  $S (\Hc^{-1}_\Ac (N)) \prec 0$ if $\Hc^{-1}_\Ac (N)$ is nonzero, giving us  $S (M) \preceq 0$.  Since $M \in \Tc$, condition \ref{para:config3}(S) applies and  gives $S' (M) \preceq 0$, showing that $E$ is $S'$-semistable in ${\Ac^{\dagger}}$.

Suppose $E$ is $S$-stable in $\Ac$.  If $\image \alpha = E$, then $\Hc^0_\Ac (N)=0$ and $N=\Hc^{-1}_\Ac (N)[1]$ is nonzero, which implies $S(N) \succ 0$ from above.  On the other hand, if $\image \alpha \subsetneq E$ then either $\Hc^{-1}_\Ac (N) \neq 0$ or $0 \neq \image \alpha \subsetneq E$; in either case, we have    $S(M)\prec 0$, which implies $S'(M)\prec 0$ by \ref{para:config3}(S).  Hence  $E$ is $S'$-stable in $\Ac^\dagger$.

(b) The case $i=1$.  Suppose $S$ satisfies refinement-$1$ and $E$ is a nonzero $S$-semistable object in $\Ac$.  Then $E \in \Fc$ and $E[1]\in \Ac^\dagger$.  We need to show that $E[1]$ is $(-S')$-semistable in $\Ac^\dagger$.

Take any ${\Ac^{\dagger}}$-short exact sequence $0 \to M \to E[1] \to N \to 0$ where $M, N \neq 0$; the associated $\Ac$-long exact sequence is
\[
  0 \to \Hc_{\Ac}^{-1}(M) \to E \overset{\alpha}{\to} \Hc_{\Ac}^{-1}(N) \to \Hc_{\Ac}^0(M) \to 0
\]
and we see $\Hc_{\Ac}^0(N)=0$.  That $\Hc^{-1}_{\Ac}(N)\in \Fc$ implies $\image \alpha \in \Fc$, and so by the $S$-semistability of $E$ we have $S(\image \alpha)\succeq 0$.  On the other hand, we have $\Hc^0_{\Ac} (M) \in \Tc$; since $S$ satisfies refinement-$1$, we have $S(\Hc^0_{\Ac}(M)) \succeq 0$.  Overall, we have $S(\Hc^{-1}_{\Ac}(N)\succeq 0$.  By condition \ref{para:config3}(S), this means $S'(\Hc^{-1}_{\Ac}(N))\succeq 0$, and so $S'(N) = S'(\Hc^{-1}_{\Ac}(N)[1]) \preceq 0$, i.e.\ $(-S')(N) \succeq 0$.  This shows the $(-S')$-semistability of $E[1]$ in $\Ac^\dagger$.

Suppose $E$ is $S$-stable in $\Ac$.  We cannot have $\image \alpha =0$ as that would imply $\Hc_{\Ac}^{-1}(N) \cong \Hc_{\Ac}^0(M)$ lies in both $\Tc$ and $\Fc$, forcing both to vanish and thus $N=0$, a contradiction.  Hence $\image \alpha \neq 0$.  If $E=\image \alpha$, then $\Hc^{-1}_\Ac (M)=0$ and $M=\Hc^0_\Ac (M)$ is nonzero, and by refinement-1 we have $S(M)\succ 0$, i.e.\ $(-S')(M)\prec 0$; if $0 \neq \image \alpha \subsetneq E$, then $S(\image \alpha)\succ 0$ since $E$ is  $S$-stable, giving us $S(\Hc^{-1}_\Ac (N)) \succ 0$, and subsequently $(-S')(N)\succ 0$ by the same argument as above. Hence $E[1]$ is $(-S')$-stable in $\Ac^\dagger$.
\end{proof}

Lemma \ref{lem:refinement-sym} allows us to prove the converse of Lemma \ref{lem:AG48-p7-D2}, giving us:

\begin{thm}\label{thm:fundamental}
Assume Configuration I, and let $\Ac^\dagger = \langle \Fc [1], \Tc \rangle$.   For $i=0$ or $1$, suppose $S$ satisfies refinement-$i$ with respect to $(\Tc, \Fc)$.  Then   an object $E\in \Dc$ is $S$-semistable in $\Ac$ if and only if $E[i]$ is $(-1)^iS'$-semistable in $\Ac^\dagger$.  The same statement holds if we replace `semistable' with `stable.'
\end{thm}

\begin{proof}
Since Lemma \ref{lem:AG48-p7-D2} gives the `only if' statements, we only need to prove the `if' statements for $i=0,1$; furthermore, we will prove only the `semistable' case, since the adjustments needed for the `stable' case are obvious.

(a) When $i=0$: Assume that $E$ is an $S'$-semistable object in $\Ac^\dagger$.  Since $S$ satisfies refinement-$0$ with respect to $(\Tc, \Fc)$, from Lemma \ref{lem:refinement-sym} we know $S'$ satisfies refinement-$1$ with respect to $(\Fc [1], \Tc)$, and so $E$ must lie in $\Tc$.  Now that $S', S$ are  in Configuration I with $S'$ satisfying  refinement-$1$ with  respect to $(\Fc [1], \Tc)$,  applying Lemma \ref{lem:AG48-p7-D2}  gives that $E[1]$ is $(-S)$-semistable in the heart $\langle \Tc [1], \Fc [1]\rangle = \Ac [1]$.

Now take any $\Ac$-short exact sequence $0 \to E' \to E \to E'' \to 0$, which induces the $\Ac [1]$-short exact sequence $0 \to E'[1] \to E[1] \to E'' [1] \to 0$.  The $(-S)$-semistability of $E[1]$ in $\Ac [1]$ then implies $(-S)(E'[1])\preceq 0$, i.e.\ $S(E')\preceq 0$, showing the $S$-semistability of $E$ in $\Ac$.

(b) When $i=1$: Assume that $E[1]$ is $(-S')$-semistable in $\Ac^\dagger$.  Since $S$ satisfies refinement-$1$ with respect to $(\Tc, \Fc)$, from Lemma \ref{lem:refinement-sym} we know $(-S')$ satisfies refinement-$0$ with respect to $(\Fc [1], \Tc)$.  Now that $(-S'), (-S)$ are in Configuration I with $(-S')$ satisfying refinement-$0$ with  respect to $(\Fc [1], \Tc)$,  applying Lemma \ref{lem:AG48-p7-D2}  gives that  $E[1]$ is $(-S)$-semistable in $\Ac[1]$.  The same argument as in (a) then shows that $E$ is $S$-semistable in $\Ac$.
\end{proof}

Stability defined in terms of weight functions can also be pulled back via an exact equivalence in the obvious manner:

\begin{lem}\label{lem:AG48-10-1}
Let $\Phi : \Dc \to \mathcal{U}$ be an exact equivalence of triangulated categories, and let $\Phi^K$ denote the induced isomorphism $K(\Dc)\to K(\mathcal{U})$ on the Grothendieck groups.  Let $\Cc$ denote the heart of a t-structure on $\Dc$ and $S' : K(\Uc) \to R$ be a weight function for some totally ordered abelian group $(R,\preceq)$.  Then an object $E\in \Dc$ is $(S'\Phi^K)$-semistable in $\Cc$ if and only if $\Phi E$ is $S'$-semistable in $\Phi \Cc$.  The same statement holds if we replace `semistable' by `stable.'
\end{lem}

Under the setup of Lemma \ref{lem:AG48-10-1}, we have the following commutative diagram
\begin{equation}
\xymatrix{
  \Dc \ar[r]^\Phi \ar[d]_(.4){[ \,\, ]} & \mathcal{U} \ar[d]^(.4){[ \,\, ]}  \\
  K(\Dc) \ar[r]^{\Phi^K} \ar@{.>}[dr]_{S'\Phi^K} & K(\Uc) \ar[d]^{S'} \\
  & R
}.
\end{equation}
We often suppress the notation for the maps $[ \, \, ]$.

\begin{proof}[Proof of Lemma \ref{lem:AG48-10-1}]
By the symmetry in the statement of Lemma \ref{lem:AG48-10-1}, it suffices to prove the `only if' direction.  Suppose $E$ is an $(S'\Phi^K)$-semistable object in $\Cc$.  Then $\Phi E \in \Phi \Cc$ and $0=(S' \Phi^K)(E)=S'(\Phi E)$.  Any $\Phi \Cc$-short exact sequence of the form $0 \to M \to \Phi E \to N \to 0$ then induces a $\Cc$-short exact sequence $0 \to \Psi M \to E \to \Psi N \to 0$ where $\Psi : \mathcal{U}\to \Dc$ is a quasi-inverse of $\Phi$.  The $(S'\Phi^K)$-semistability of $E$ implies that $0 \succeq (S'\Phi^K)(\Psi M) = S'(\Phi \Psi M)=S'(M)$, showing that $\Phi E$ is $S'$-semistable.  The above argument can be easily adjusted to show the `stable' case.
\end{proof}

\paragraph[Configuration II]  We  \label{para:config2} say we are in Configuration II when we have the following setting:    $(R,\preceq)$ is a totally ordered abelian group, and
\begin{itemize}
\item[(a)] $\Phi : \Dc \to \Uc$ and $\Psi : \Uc \to \Dc$ are exact equivalences between triangulated categories $\Dc, \Uc$ satisfying $\Psi \Phi \cong \mathrm{id}_\Dc [-1]$ and $\Phi \Psi \cong \mathrm{id}_\Uc [-1]$.
\item[(b)] $\Ac, \Bc$ are hearts of bounded t-structures on $\Dc, \Uc$, respectively, such that $\Phi \Ac \subset D^{[0,1]}_\Bc$.  (Given (a),  this implies   $\Psi \Bc \subset D^{[0,1]}_\Ac$ by Lemma \ref{para:AcBcPhitilt} below.)
\item[(c)] $S_\Ac : K(\Dc) \to R$ and $S_\Bc : K(\mathcal{U}) \to R$ are group homomorphisms such that, for any $E \in \Ac$ that is either $\Phi_\Bc$-WIT$_0$ or $\Phi_\Bc$-WIT$_1$, we have     $\sgn S_\Ac (E) = \sgn S_\Bc (\Phi E)$.
\end{itemize}

Configuration II above  involves  equivalences of triangulated categories, and will be more readily applied to  situations in representation theory and algebraic geometry involving tilting equivalences and Fourier-Mukai transforms.

Let us consider some consequences of conditions (a) and (b) in \ref{para:config2} before moving on.

\begin{lem}\label{para:AcBcPhitilt}
Suppose $\Dc, \Uc$ are triangulated categories with exact equivalences  $\Phi : \Dc \to \Uc$ and $\Psi : \Uc \to \Dc$  satisfying $\Psi \Phi \cong \mathrm{id}_{\Dc}[-1], \Phi \Psi \cong \mathrm{id}_{\Uc}[-1]$.  If $\Ac, \Bc$ are hearts of bounded t-structures on $\Dc, \Uc$, respectively, such that $\Phi \Ac \subset D^{[0,1]}_{\Bc}$, then we also have $\Psi \Bc \subset D^{[0,1]}_{\Ac}$.
\end{lem}

\begin{proof}
The assumption $\Phi \Ac \in D^{[0,1]}_{\Bc}$ implies $\Phi \Ac \subset \langle \Bc , \Bc [-1]\rangle$, and so $\Phi \Ac$ is a tilt of $\Bc [-1]$ at some torsion pair $(\Tc, \Fc)$ in $\Bc [-1]$  \cite[Proposition 2.3.2]{BMT1}.  Then $\Phi \Ac = \langle \Fc [1], \Tc \rangle$ while $\Bc [-1] = \langle \Tc, \Fc\rangle$, from which we see $\Bc [-1]$ is a tilt of $\Phi \Ac [-1]$, i.e.\ $\Bc [-1] \subset \langle \Phi \Ac, \Phi \Ac [-1]\rangle$, which gives $\Psi \Bc \subset \langle \Ac, \Ac [-1]\rangle$, i.e.\ $\Psi \Bc \subset D^{[0,1]}_{\Ac}$.
\end{proof}

Recall the notation defined in \ref{para:def-WIT}.

\begin{lem}\label{lem:AG46-80-1}
Suppose  conditions (a) and (b) of Configuration II are satisfied.  Then
\begin{itemize}
\item[(i)] For $i=0,1$, if $E$ is a $\Phi_\Bc$-WIT$_i$ object in $\Ac$  then $\Phi E [i]$ is a $\Psi_\Ac$-WIT$_{1-i}$ object in $\Bc$.
\item[(ii)] $( W_{0,\Phi, \Ac, \Bc},  W_{1,\Phi, \Ac, \Bc})$ is a torsion pair in $\Ac$.
\end{itemize}
The same statements hold if we switch the roles of $\Phi$ and $\Psi$ and the roles  of $\Ac$ and $\Bc$.
\end{lem}

\begin{proof}
By symmetry in the setup of Lemma \ref{para:AcBcPhitilt}, it suffices to prove statements (i) and (ii).

(i) If $E \in \Ac$ is $\Phi_\Bc$-WIT$_0$, then $\Phi E \in \Bc$ and $\Psi (\Phi E) [1] \cong E \in \Ac$, showing that $\Phi E$ is a $\Psi_\Ac$-WIT$_1$ object in $\Bc$.  Similarly, if $E \in \Ac$ is $\Phi_\Bc$-WIT$_1$, then $\Phi E \in \Bc [-1]$ and $\Psi (\Phi E [1])\cong E \in \Ac$, showing $\Phi E [1]$ is a $\Psi_\Ac$-WIT$_0$ object in $\Bc$.

(ii) Take any $E \in \Ac$.  Since $\Phi E \in D^{[0,1]}_{\Bc}$, we have an exact triangle
\[
\Phi^0_{\Bc}E \to \Phi E \to \Phi^1_{\Bc} E [-1] \to \Phi^0_{\Bc} E [1]
\]
which in turn gives the exact triangle
\begin{equation}\label{eq:AG46-80-1}
\Psi(\Phi^0_{\Bc}E)[1] \to E \to \Psi (\Phi^1_{\Bc}E  )  \to \Psi(\Phi^0_{\Bc}E)[2].
\end{equation}
Since $\Phi^0_{\Bc} E \in \Bc$, we have $\Psi (\Phi^0_{\Bc} E) [1] \in D^{[-1,0]}_{\Ac}$ while $\Psi (\Phi^1_{\Bc} E) \in D^{[0,1]}_{\Ac}$ by Lemma \ref{para:AcBcPhitilt}.  Taking the long exact sequence of cohomology of \eqref{eq:AG46-80-1} with respect to the heart $\Ac$  gives:
\begin{itemize}
\item $\Hc^{-1}_{\Ac} (\Psi (\Phi^0_{\Bc} E)[1])=0$, which means that $\Phi^0_{\Bc} E$ is $\Psi_{\Ac}$-WIT$_1$; by (i), this means  $\Psi (\Phi^0_{\Bc} E)[1]$ is a $\Phi_{\Bc}$-WIT$_0$ object in $\Ac$.
\item $\Hc^1_{\Ac} (\Psi (\Phi^1_\Bc E))=0$, which means $\Phi^1_\Bc E$ is $\Psi_{\Ac}$-WIT$_0$; by (i), this means  $\Psi (\Phi^1_\Bc E)$ is a $\Phi_\Bc$-WIT$_1$ object in $\Ac$.
 \end{itemize}
   Hence \eqref{eq:AG46-80-1} gives an $\Ac$-short exact sequence.  That $\Hom_{\Ac}( W_{0,\Phi, \Ac, \Bc},  W_{1,\Phi, \Ac, \Bc})=0$ is clear, and so  (ii) follows.
\end{proof}

\subparagraph The proof of Lemma \ref{lem:AG46-80-1}(ii) shows that, when  conditions (a) and (b) of Configuration II are in place,  given any $E \in \Ac$, the object $\Phi^i_\Bc E$ is $\Psi_\Ac$-WIT$_{1-i}$ for $i=0,1$.

\paragraph[Passing between Configuration I and Configuration II] We can pass between Configuration I and Configuration II as follows.

\subparagraph To pass from \label{para:convItoII} Configuration I to Configuration II, assume the setting and notation of Configuration I in  \ref{para:config3}.  By  setting $\Uc = \Dc$, $\Phi = \mathrm{id}_{\Dc}$ and $\Psi = \mathrm{id}_{\Dc}[-1]$, we obtain condition \ref{para:config2}(a).  Putting $\Bc = \langle \Fc [1], \Tc\rangle$, we obtain condition \ref{para:config2}(b).  Finally, choose $S_\Ac=S$ and $S_\Bc= S'$; then an object $E \in \Ac$ is $\Phi_\Bc$-WIT$_0$ (resp.\ $\Phi_\Bc$-WIT$_1$) if and only if $E$ lies in $\Tc$ (resp.\ $\Fc$); condition \ref{para:config3}(S) of Configuration I then implies condition \ref{para:config2}(c).  So we are now in Configuration II.

\subparagraph To pass from \label{para:convIItoI} Configuration II to Configuration I, assume the setting and notation of Configuration II in \ref{para:config2}.   Let  $\wt{\Bc}=\Psi \Bc [1]$.  Then  $\wt{\Bc} \subset D^{[-1,0]}_{\Ac}$ by Lemma \ref{para:AcBcPhitilt}, and so $\wt{\Bc}$ is a tilt of $\Ac$  with respect to some torsion pair $(\Tc,\Fc)$ in $\Ac$ by \cite[Proposition 2.3.2(b)]{BMT1}.  Then $\Ac = \langle \Tc, \Fc\rangle$ and $\wt{\Bc} = \langle \Fc [1], \Tc \rangle$ while
\begin{align*}
  W_{0,\Phi,\Ac,\Bc} &= \{E\in \Ac : \Phi E \in \Bc\} =  (\Psi [1])(\Phi \Ac \cap \Bc) = \Ac \cap \wt{\Bc}= \Tc, \\
  W_{1,\Phi,\Ac,\Bc} &=  \{ E \in \Ac : \Phi E \in \Bc [-1]\}= (\Psi [1])(\Phi \Ac \cap (\Bc[-1])) = \Ac \cap (\wt{\Bc}[-1]) = \Fc.
\end{align*}
Let $\Phi^K$ denote the $K$-theoretic isomorphism $K(\Dc) \to K(\mathcal{U})$  induced by the exact equivalence $\Phi$ \cite[Remark 5.25 i)]{huybrechts2006fourier}.  Take $S=S_\Ac$ and $S' = S_\Bc \circ \Phi^K$.  Condition \ref{para:config2}(c) in Configuration II now implies the sign condition  in Configuration I for $S, S'$, and we are now in Configuration I.

\subparagraph It is easy to check, that if one starts  with the data in Configuration I as in \ref{para:config3}, and apply the construction in \ref{para:convItoII} followed by the construction in \ref{para:convIItoI}, then one obtains the same data that one starts with.

\begin{thm}\label{thm:main3}
Assume Configuration II.  For $i=0$ or $1$,  suppose  the weight function $S_\Ac$ satisfies  refinement-$i$ with respect to the torsion pair $(W_{0,\Phi,\Ac,\Bc}, W_{1,\Phi,\Ac,\Bc})$ in $\Ac$.  Then the functor $\Phi[i]$  induces an equivalence from the category of $S_\Ac$-semistable objects in $\Ac$ to the category of $(-1)^iS_\Bc$-semistable objects in $\Bc$.  The same statement holds if we replace `semistable' by `stable.'
\end{thm}

\begin{proof}
Assume Configuration II.  Using the argument and notation in \ref{para:convIItoI}, we can pass to Configuration I.  In particular, we have $S=S_\Ac$, and the torsion pair $(\Tc, \Fc)$ is the same as $(W_{0,\Phi,\Ac,\Bc}, W_{1,\Phi,\Ac,\Bc})$.  Then  the weight function $S$ satisfies the refinement property with respect to the torsion pair $(\Tc, \Fc)$ in $\Ac$.  By Theorem \ref{thm:fundamental}, an object $E\in \Dc$ is $S$-semistable in $\Ac$ if and only if $E[i]$ is $(-1)^iS'$-semistable in the tilted heart $\langle \Fc [1], \Tc\rangle$.  Since $(-1)^i S'= (-1)^i S_\Bc \Phi^K$ and $\Phi \langle \Fc [1], \Tc\rangle = \Bc$,  Lemma \ref{lem:AG48-10-1} in turn implies that $E[i]$ is $(-1)^iS'$-semistable in  $\langle \Fc [1], \Tc\rangle$ if and only if $(\Phi E)[i]$ is $(-1)^i S_\Bc$-semistable in $\Bc$.  This shows the `semistable' case of the theorem.  The `stable' case of the theorem follows from the corresponding cases of Theorem \ref{thm:fundamental} and  Lemma \ref{lem:AG48-10-1}.
\end{proof}

We can now recover the following result on representations of algebras due to Chindris. Chindris used Theorem \ref{thm:chindrisA} to prove the existence of singular moduli spaces of modules for wild tilted algebras, which in turn showed one implication of a conjecture of Weyman's on strongly simply connected algebras \cite[Section 1]{chindris2013}.

\begin{thm}\cite[Theorem 1.3(a)]{chindris2013}\label{thm:chindrisA}
Let $A$ be a bound quiver algebra, $T$ a basic tilting $A$-module, and $\theta$ an integral weight of $A$ which is well-positioned with respect to $T$.  Let $F$ be either the functor $\Hom_A(T,-)$ in case there are nonzero $\theta$-semistable torsion $A$-modules or the functor $\Ext^1_\Ac (T,-)$ in case there are nonzero $\theta$-semistable torsion-free $A$-modules.  Denote $\mathrm{End}_A(T)^{op}$ by $B$ and let $u : K(\text{$A$-mod}) \to K(\text{$B$-mod})$ be the isometry induced by the tilting module $T$.  Then the functor $F$ defines an equivalence from the category of $\theta$-semistable $A$-modules, to the category of $|\theta \circ u^{-1}|$-semistable $B$-modules.
\end{thm}

\begin{proof}
That $T$ is a tilting module means  we have the following derived equivalences that are quasi-inverse to each other
\begin{align*}
  R\Hom_A(T,-) : D^b(\text{$A$-mod}) &\to D^b(\text{$B$-mod}), \\
  T \overset{L}{\otimes_B} - : D^b(\text{$B$-mod}) &\to D^b(\text{$A$-mod}).
\end{align*}
Moreover, we know that $R\Hom_A(T,E) \in D^{[0,1]}_{\text{$B$-mod}}$ for any $E \in \text{$A$-mod}$, while $T  \overset{L}{\otimes_B} F \in D^{[-1,0]}_{\text{$A$-mod}}$ for any $F \in \text{$B$-mod}$.

Let us  set $\Ac = \text{$A$-mod}$ and $\Bc = \text{$B$-mod}$, take $\Dc = D^b(\Ac)$ and $\Uc = D^b(\Bc)$,    take $\Phi$ to be the functor $R\Hom_A (T,-)$ and $\Psi$  the functor $(T \overset{L}{\otimes_B} -)[-1]$.  Then \ref{para:config2}(a),(b) are satisfied.

In addition, let us identify $K(\Ac), K(\Bc)$ with $K(\Dc), K(\mathcal{U})$, respectively. Let $(R,\preceq)=(\mathbb{Z},\leq)$, let $S_\Ac$ be an integral weight $\theta \in \Hom_{\mathbb{Z}} (K(\Ac),\mathbb{Z})$ and   let $S_\Bc = \theta \circ u^{-1}$ where $u : K(\Ac) \to K(\Bc)$ satisfies $u(\dimension M)= \dimension (\Phi M)$, with $\dimension (-)$ denoting the dimension vector of a module.  Then we obtain the commutative diagram
\begin{equation}\label{lem:AG48-13-1}
\xymatrix{
  \Dc \ar[r]^\Phi \ar[d]_(.4){\text{dim}} & \mathcal{U} \ar[d]^(.4){\text{dim}}  \\
  K(\Dc) \ar[r]^{u} \ar[dr]_{S_\Ac} & K(\Uc) \ar@{.>}[d]^{S_\Bc} \\
  & \mathbb{Z}
}.
\end{equation}
Then for any $E\in \Ac$, we have
\[
S_\Ac (\dimension E)=\theta (\dimension E) = (\theta \circ u^{-1})(\dimension \Phi E) = S_\Bc (\dimension \Phi E),
\]
which implies that  \ref{para:config2}(c) is also satisfied. Hence we are in Configuration II. Moreover, Chindris' definition of  $\theta$ being well-positioned with respect to $T$ implies the refinement property for $S_\Ac$ with respect to the torsion pair $(W_{0,\Phi,\Ac,\Bc}, W_{1,\Phi,\Ac,\Bc})$ (see \cite[Definition 3.2]{chindris2013}).    Depending on whether $S_\Ac$ satisfies refinement-$0$ or refinement-$1$, $|S_\Bc|$-semistability in the sense of \cite[Section 3B]{chindris2013} coincides with $(-1)^iS_\Bc$-semistability.  Theorem \ref{thm:main3} now specialises to  Theorem \ref{thm:chindrisA}.
\end{proof}

\section{Polynomial stability conditions}\label{sec:polystab}

We slightly extend the definition of polynomial stability conditions as originally given by Bayer  \cite{BayerPBSC}.  In particular, we allow our central charges to take values in the ring $\CLoovc$ of convergent Laurent series in $\tfrac{1}{v}$, instead of only the polynomial ring $\mathbb{C}[v]$.  

\paragraph[Big $O$, big $\Theta$] Given \label{para:bigOTheta} functions $f, g : (l,\infty) \to \mathbb{R}$ for some $l \in \mathbb{R}$, recall that we write $f=O(g)$ if there exist constants $C \in \mathbb{R}_{>0}$ and $k \in (l,\infty)$ such that $|f(v)| \leq C |g(v)|$ for all $v > k$, and we write $f=\Theta (g)$ if $f=O(g)$ as well as $g=O(f)$.

\subparagraph[Power series] For  $K=\mathbb{R}$ or $\mathbb{C}$, we will write $K\llbracket w \rrbracket$ to denote the ring of power series in the indeterminate $w$ over $K$.  We will write $K \llbracket w \rrbracket^c$ to denote the $K$-subalgebra of power series with positive radii of convergence.  We can also regard $\RPwc$ as a subring of  $\CPwc$ by \cite[Theorem 2, Chap. 2]{ahlfors}.

\subparagraph[Laurent series] For  $K=\mathbb{R}$ or $\mathbb{C}$, we will write $K (\!(w)\!)$ to denote the field of formal Laurent series in the indeterminate $w$ over $K$, which we take to be series of the form
\[
 \sum_{k \geq m}  a_k w^k  \text{\quad where $m \in \mathbb{Z}, a_k \in K$ for all $k \geq m$}.
\]
We only allow finitely many negative-degree terms in a formal Laurent series.  We will write $K(\!(w)\!)^c$ to denote the subfield  of $K(\!(w)\!)$ consisting of series that are convergent on  punctured disks  of positive radii centered at $0$, i.e.\
\begin{align*}
  K(\!(w)\!)^c &= \{ f(w) \in K(\!(w)\!) : f \text{ converges on } 0 < |w| < R \text{ for some $R>0$}\} \\
  &= \bigcup_{k \in \mathbb{Z}^+} \tfrac{1}{w^k} K\llbracket w \rrbracket^c.
\end{align*}
 In particular, every element of $K(\!(w)\!)^c$ is uniformly convergent on some punctured disk of positive radius centered at $0$.   For any $f \in \mathbb{R}\Loov^c$, if we write $f = \sum_{k \geq m} a_k (\tfrac{1}{v})^k$ where $a_k \in \mathbb{R}$ for $k \geq m$ and $a_m \neq 0$, then we have $\Theta (f)=\tfrac{1}{v^m}$ as $v \to \infty$.

\subparagraph We have a \label{para:R1vordering} partial order $\preceq$ on  $\mathbb{R}\Loovc$  determined by
\[
f \preceq g \text{ if and only if } f(v) \leq g(v) \text{ for $v \gg 0$}.
\]
It is clear that $\mathbb{R}\Loovc$ is totally ordered with respect to $\preceq$, and so $(\mathbb{R}\Loovc, \preceq)$ is a totally ordered abelian group with respect to addition.

\paragraph[Polynomial phase function] Given any \label{para:defpolyphasefunc} nonzero $f(v) \in \CLoovc$, we can write $f(v)=\sum_{k\leq m} a_kv^k$ where $a_k \in \mathbb{C}$ for $k \leq m$, for some $m \in\mathbb{Z}$ where $a_m \neq 0$.  Then $f(v)=v^mg(v)$ for some  $g(v) \in \mathbb{C}\llbracket \tfrac{1}{v}\rrbracket^c$.  Thus $f(v)$ is an analytic function in the region  $|v| >r$ for some $r\in \mathbb{R}_{>0}$ and,  in particular, if we let $v \to \infty$ along the positive real axis, then there is a continuous function germ $\phi : \mathbb{R}_{>0} \to \mathbb{R}$ such that
\[
f (v) \in \mathbb{R}_{>0}e^{i\pi \phi (v)} \text{\quad  for $v \gg 0$}.
\]
Note that the choice of the function germ $\phi$ is only unique up to translation by integer multiples of $2$.
Any function germ $\phi$ arising in this manner will be called a \emph{polynomial phase function}.  Since $a_m \neq 0$ by assumption,  $g(v)$ has nonzero constant term $a_m$, which means
the limit $\phi (\infty) := \lim_{v \to \infty} \phi (v)$ exists and  $a_m \in \mathbb{R}_{>0}e^{i\pi \phi (\infty)}$.

Given polynomial phase functions $\phi_1, \phi_2 $, we  write $\phi_1 \prec \phi_2$  (resp.\ $\phi_1 \preceq \phi_2$) if $\phi_1 (v) < \phi_2 (v)$ (resp.\ $\phi_1(v) \leq \phi_2(v)$) for $v \gg 0$.

\begin{lem}\label{lem:seriescomparability}
The set of all polynomial phase functions is totally ordered with respect to $\preceq$.
\end{lem}

\begin{proof}
Suppose $f_1, f_2$ are nonzero elements of $\CLoovc$  that give rise to  polynomial phase functions $\phi_1, \phi_2$, respectively.  We need to show that $\phi_1, \phi_2$ are comparable with respect to $\preceq$.

From above, we know that the limits of $\phi_1 (v), \phi_2 (v)$ as $v \to \infty$ along the real axis exist.  If these two limits are not equal, then  either $\phi_1 \prec \phi_2$ or $\phi_1 \succ \phi_2$.  Therefore, it remains to consider the case where the limits are equal.

Suppose now $\lim_{v \to \infty} \phi_1 (v) = \lim_{v \to \infty} \phi_2 (v)$, and let $\phi$ denote the particular branch of   phase functions of the quotient $f_1/f_2$ satisfying $\lim_{v \to\infty} \phi (v)=0$, and that $\phi (v) \in (-1, 1)$ for $v \gg 0$. Then determining the order relation between $\phi_1$ and $\phi_2$ is equivalent to determining the order relation between $\phi$ and the constant phase function $0$.  Since $\CLoovc$ is a field, we can write
\[
  \frac{f_1}{f_2} = v^m \sum_{j \geq 0} a_j \left( \tfrac{1}{v}\right)^j
\]
for some $m \in \mathbb{Z}$ and  $a_0 \neq 0$.  Note that $a_0 \in \mathbb{R}$ since $\lim_{v \to \infty} \phi (v)=0$.  If all the coefficients $a_k$ lie in $\mathbb{R}$, then $\phi=0$; otherwise, let $j'$ be the least positive integer such that $a_{j'}\notin\mathbb{R}$.  Then we have $\phi \prec 0$ (resp.\ $\phi \succ 0$) if the imaginary part of $a_{j'}$ is negative (resp.\ positive), and the lemma follows.
\end{proof}

\paragraph[(Weak) polynomial stability function] Let  \label{para:weakpgstabfunc} $\Dc$ be a triangulated category, and $\Ac$ the heart of a bounded t-structure on $\Dc$.   We say  a group homomorphism $Z : K(\Dc) \to \CLoovc$  is a \emph{weak polynomial stability function}  on $\Ac$ with respect to  $(\phi_0, \phi_0+1]$  if there exists a polynomial phase function $\phi_0$ such that, for all nonzero objects $E$ in $\Ac$,  we have
\[
  Z(E)(v) \in \mathbb{R}_{\geq 0}\cdot e^{i\pi (\phi_0(v), \phi_0(v)+1]} \text{\qquad for $v \gg 0$}.
\]
Note that we allow the possibility of $Z(E)=0$ in $\CLoovc$ for a  nonzero object $E$ in $\Ac$.  Whenever $Z$ is a weak polynomial stability function on $\Ac$, we will set
\[
  \Ac_{\mathrm{ker} Z} = \{ E \in\Ac : Z(E)=0\}.
\]
It is easy to check that $\Ac_{\mathrm{ker} Z}$ is a Serre subcategory of $\Ac$ (i.e.\ closed under subobject and quotient).

A weak polynomial stability function $Z$  will be called a \emph{polynomial stability function}  if $\Ac_{\mathrm{ker} Z}=\{0\}$, i.e.\  for all nonzero objects $E$ in $\Ac$ we have
\[
  Z(E)(v) \in \mathbb{R}_{>0}\cdot e^{i \pi (\phi_0(v), \phi_0(v)+1]} \text{\quad for $v \gg 0$}.
\]

\begin{remsub}\label{rem:constphase}
Every real number $r$ is a (constant) polynomial phase function, since it arises from the constant polynomial  $e^{ir\pi}$ in $\CLoovc$ (see \ref{para:defpolyphasefunc}).
\end{remsub}

\subparagraph Given a weak \label{para:weakpgstabss} polynomial stability function $Z$ as in \ref{para:weakpgstabfunc},  for every  $0 \neq E \in \Ac$ we define the \emph{(polynomial) phase function} $\phi_Z (E)$ as a function germ $\mathbb{R}_{>0} \to \mathbb{R}$ via the relations
\[
  \begin{cases} Z(E)(v) \in \mathbb{R}_{>0}    e^{i \pi \phi_{Z} (E)(v)} \text{ for $v \gg 0$, and } \phi_0 \prec  \phi_Z(E) \preceq \phi_0+1  &\text{\quad if $Z(E)\neq 0$}, \\
  \phi_{Z} (E) = \phi_0+1 &\text{\quad if $Z(E)=0$}.
  \end{cases}
\]
Note that we always have $\phi_0 \prec \phi_Z(E) \preceq \phi_0+1$ for $0 \neq E \in \Ac$.  A nonzero object $E \in \Ac$ is said to be \emph{$Z$-semistable} (resp.\ \emph{$Z$-stable}) if, for every $\Ac$-short exact sequence
\[
  0 \to M \to E \to N \to 0
\]
where $M, N\neq 0$, we have
\[
  \phi_Z (M) \preceq \phi_Z (N) \text{\quad (resp.\ $\phi_Z (M) \prec \phi_Z (N)$)}.
\]
More generally, we say that an object $E \in \Dc$ is $Z$-semistable (resp.\ $Z$-stable) if $E[j]$   is a $Z$-semistable (resp.\ $Z$-stable) object in $\Ac$  for some $j \in \mathbb{Z}$, in which case we define $\phi_Z(E)$ via the relation $\phi_Z (E) + j = \phi_Z (E[j])$.  We say $Z$ satisfies the \emph{Harder-Narasimhan (HN) property} on $\Ac$ if every nonzero object $E$ in $\Ac$ admits an $\Ac$-filtration
\begin{equation}\label{eq:HNfilt}
 0=E_0 \subsetneq E_1 \subsetneq \cdots \subsetneq  E_m = E
\end{equation}
for some $m \geq 1$, where each quotient $E_i /E_{i-1}$ ($1 \leq i \leq m$) is a $Z$-semistable object and
\[
  \phi_Z (E_1/E_0) \succ \phi_Z (E_2/E_1) \succ \cdots \succ \phi_Z (E_m/E_{m-1}).
\]

\subparagraph[(Weak) polynomial stability condition] Let \label{para:pgstab}$\Dc$ be a triangulated category and $\Ac$ the heart of a bounded t-structure on $\Dc$.  If  $Z$ is a (resp.\ weak) polynomial stability function on  $\Ac$ (with respect to $(\phi_0, \phi_0+1]$ for some $\phi_0$) that satisfies the HN property, then we say $(Z,\Ac)$ is a (resp.\ \emph{weak}) \emph{polynomial stability condition} on $\Dc$.

\subparagraph[Bayer's polynomial stability condition] If \label{para:wpsc-psc} we require the central charge $Z$ of a polynomial stability condition to take values in the polynomial ring $\mathbb{C}[v]$, then we obtain the original definition of a polynomial stability condition in the  sense of Bayer \cite{BayerPBSC}.

\subparagraph[Bridgeland stability condition] A stability \label{para:wpsc-bsc} condition as defined by  Bridgeland in \cite{StabTC} is a polynomial stability condition $(Z,\Ac)$ where $Z : K(\Dc) \to \mathbb{C}$  is a (constant) polynomial stability function on $\Ac$ with respect to $(0,1]$, where we regard $0$ and $1$ as constant polynomial phase functions (see Remark \ref{rem:constphase}).  We will write  $\mathrm{Stab}(X)$ to denote the space of   stability conditions as defined by Bridgeland that are both full and locally finite (these two properties together are equivalent to the support property - see \cite[Section 5.2]{MSlec}); $\mathrm{Stab}(X)$ has the structure of a complex manifold.

\subparagraph[Note on terminology] For convenience, we will say a pair $(Z, \Ac)$ is a Bridgeland stability condition on a triangulated category $\Dc$ with respect to $(a,a+1]$, or simply a Bridgeland stability condition, if $(Z,\Ac)$ is a polynomial stability condition on $\Dc$ with respect to $(a,a+1]$ for some constant $a \in \mathbb{R}$, and $Z : K(\Dc)\to \mathbb{C}$ takes values in $\mathbb{C}$.  In particular, we do not require that $a=0$ as in \cite{StabTC}.  This use of terminology gives more consistency  with the literature on polynomial stability conditions such as \cite{BayerPBSC, BMT1}, and gives us more flexibility in manipulating stability using both autoequivalences and "GL-actions" (see \ref{para:inducedGammaP}) simultaneously, especially when we are more interested in the semistability of objects themselves and their moduli spaces.

\paragraph[Slicing] Suppose \label{para:slicing} $(Z,\Ac)$ is a polynomial  stability condition on a triangulated category $\Dc$.  For any polynomial phase function $\phi$, let us write $\Pc (\phi)$  to denote the full subcategory of $Z$-semistable objects $E$ in $\Dc$ with phase function $\phi_Z (E)=\phi$, together with the zero object.  For any interval $I$ of  polynomial phase functions, we will define $\Pc (I)$ as a full subcategory of $\Dc$ via the extension closure
\[
  \Pc (I) = \langle \Pc (\phi) : \phi \in I \rangle.
\]
On the other hand, for any $r \in \mathbb{R}$, we will write $\wh{\Pc} (r)$ to denote the full subcategory  of $Z$-semistable objects $E$ in $\Dc$ with $\lim_{v \to \infty} \phi_Z (E)(v) = r$, together with the zero object.  For any interval $I$ in $\mathbb{R}$, we will define $\wh{\Pc} (I)$  as a full subcategory of $\Dc$ via the extension closure
\[
  \wh{\Pc} (I) = \langle \wh{\Pc} (r) : r\in I \rangle.
\]
If $(Z, \Ac)$ is a  polynomial stability condition with respect to $(a,a+1]$ on a triangulated category $\Dc$ and $\Pc$ denotes its slicing, then the data of $(Z,\Ac, (a,a+1])$  is equivalent to the data  $(Z,\Pc)$  \cite[Proposition 2.3.3]{BayerPBSC}.  As a result, we also refer to $(Z,\Pc)$ as a polynomial stability condition (which was how the notion was originally defined by Bayer \cite{BayerPBSC}).
  For a triangulated category $\Dc$, we formally define  $\Stabpol (\Dc)$ to be the set of all polynomial stability conditions  $(Z,\Pc)$   on $\Dc$.

\paragraph[Slope functions] Suppose $\Dc$ is a \label{para:mugeneraldef}  triangulated category,  $Z : K(\Dc) \to \CLoovc$ is a group homomorphism, and $\Ac$ is an abelian subcategory  of $\Dc$ such that
\[
  Z(E)(v) \in \mathbb{R}_{\geq 0} e^{i\pi(0,1]} \text{\qquad for $v \gg 0$}
\]
for all $E \in \Ac$.  Then we can define a  function $\mu_Z : \Ac \setminus \{0\} \to \mathbb{R}\Loovc \cup \{\infty\}$ by setting
\[
  \mu_Z (E) = \begin{cases}
  -\tfrac{\Re Z(E)}{\Im Z(E)} &\text{ if $\Im Z(E) \neq 0$}\\
  \infty &\text{ if $\Im Z(E)=0$}
  \end{cases}.
\]
The partial order $\preceq$ on $\mathbb{R}\Loovc$ (see \ref{para:R1vordering}) can be extended to $ \mathbb{R}\Loovc \cup \{\infty\}$ by declaring $f \prec \infty$ for any $f \in \mathbb{R}\Loovc$.  We then define a nonzero object $E\in \Ac$ to be $\mu_Z$-semistable (resp.\ $\mu_Z$-stable) if, for every $\Ac$-short exact sequence
\[
0 \to M \to E \to N \to 0
\]
where $M, N \neq 0$, we have $\mu_Z (M) \preceq \mu_Z (N)$ (resp.\ $\mu_Z (M) \prec \mu_Z (N)$).  We will refer to any function of the form $\mu_Z$ as a \emph{slope function} on $\Ac$.  In particular, we do not require $\Ac$ to be the heart of a t-structure on $\Dc$.

\subparagraph In the setting  \label{para:mugeneraldef1} of \ref{para:mugeneraldef}, even though $\Ac$ is not necessarily the heart of a t-structure on $\Dc$, we can still define polynomial phase functions $\phi_Z(-)$  associated to $Z$ and define the notion of $Z$-(semi)stability for objects in $\Ac$.  For any nonzero objects $E, E' \in \Ac$, it is easy to see that  $\mu_Z (E) \preceq \mu_Z (E')$ if and only if $\phi_Z (E) \preceq \phi_Z (E')$.  Therefore, a nonzero object $E$ of $\Ac$ is $\mu_Z$-semistable (resp.\ $\mu_Z$-stable) if and only if it is $Z$-semistable (resp.\ $Z$-stable).  When $Z$ has the HN property on $\Ac$, we will  say that $\mu_Z$ has the HN property.     In particular, when $\Ac$ is a noetherian abelian category, if we  assume that the image of $Z$ is contained in $\mathbb{Z}[i]$, then $Z$ always satisfies the HN property on $\Ac$ \cite[Proposition 3.4]{LZ2}. When $\mu_Z$ is a slope function with the HN property and $E$ is a nonzero object in $\Ac$ with $Z$-HN filtration as in \eqref{eq:HNfilt}, we  set $\mu_{Z,\text{max}}(E)=\mu_Z (E_1/E_0)$ and $\mu_{Z,\text{min}}(E)=\mu_Z (E_m/E_{m-1})$.

\begin{egsub}\label{eg:stabslopefunc}
Let $X$ be a smooth projective curve and $\Ac = \Coh (X)$ as in Example \ref{eg:smprojcurv}.  Consider the polynomial stability function $Z : K(D^b(\Ac))=K(X) \to \mathbb{C}$ on $\Ac$  given by
\[
  Z(E) = -\degree E + i \rank E.
\]
Then  the  slope function $\mu_Z (E)= (\degree E)/(\rank E)$ is the usual slope function for slope stability for coherent sheaves on $X$.
\end{egsub}

\paragraph[Different types of stability so far]  So far in this article, we have defined  two main types of stability:
\begin{itemize}
 \item[(i)] stability defined by  a weight function on a triangulated category (in \ref{para:weightfuncdef});
 \item[(ii)] stability defined by a weak polynomial stability function on a triangulated category (in \ref{para:weakpgstabss}).
\end{itemize}
Our definition of a weak polynomial stability condition specialises to a particular type of Bayer's polynomial stability condition (see \ref{para:wpsc-psc} and \cite[Definition 2.3.1]{BayerPBSC}), which  in turn specialises to stability in the sense of Bridgeland (see \ref{para:wpsc-bsc} and \cite{StabTC}).  Notions of stability that arise from  `slope functions' such as those in Example \ref{eg:stabslopefunc} or slope-like functions as in \cite[3.2]{LZ2} can also be interpreted as arising from weak polynomial stability conditions (see \ref{para:mugeneraldef}).

In Lemma \ref{lem:S-stabfuncdef}, we show how a weak polynomial stability function can be used to construct weight functions.  This lemma generalises \cite[Proposition 3.4]{rudakov1997stability} by Rudakov:

\begin{lem}\label{lem:S-stabfuncdef}
Suppose $\Dc$ is a triangulated category, and $\Ac$ is the heart of  a bounded t-structure  on $\Dc$.  Suppose    $Z : K(\Dc) \to \CLoovc$ is a weak polynomial stability function  on $\Ac$ with respect to $(0,1]$, and $M$ is a fixed  object in $\Ac$ such that $Z(M) \neq 0$.  Consider the group homomorphism $S : K(\Dc) \to \mathbb{R}\Loovc$ given by
\begin{equation}\label{eq:S-stabfuncdef}
    S(E):= S_{Z,M} (E):=
     \begin{vmatrix} \Re Z(M)  & \Im Z(M)  \\
    \Re Z(E) & \Im Z(E) \end{vmatrix}
     =-(\Im Z(M))(\Re Z(E)) + (\Re Z(M))(\Im Z(E)).
\end{equation}
Then $S(M)=0$, and for  any  object $E \in \Ac$ with $Z(E)\neq 0$, we have
\begin{itemize}
\item[(i)] $\phi_Z (E) \prec \phi_Z (M) \Leftrightarrow S(E) \prec 0$.
\item[(ii)] $\phi_Z (E) = \phi_Z (M) \Leftrightarrow S(E) = 0$.  When $S(E)=0$, we have $\Im Z(E)=0 \Leftrightarrow \Im Z(M)=0$.
\item[(iii)] $\phi_Z (E) \succ \phi_Z (M) \Leftrightarrow S(E) \succ 0$, provided  $\Im Z (M) \succ 0$.
\end{itemize}
For (iv) and (v), assume further that $\Hom (\Ac_{\mathrm{ker} Z}, E)=0$:
\begin{itemize}
\item[(iv)] $E$ is $S$-semistable in $\Ac$ if and only if $\begin{cases} \phi_Z (E) = \phi_Z (M)\\
  \text{$E$ is $Z$-semistable in $\Ac$}\end{cases}$.
\item[(v)] If $E$ is $S$-stable in $\Ac$, then $\begin{cases} \phi_Z (E)= \phi_Z (M) \\
    E \text{ is $Z$-stable in $\Ac$}\end{cases}$.  If $Z$ is a polynomial stability function on $\Ac$, then the converse also holds.
\end{itemize}
\end{lem}

\begin{proof}
Let $\mu_Z$ denote the slope function constructed from $Z$ as in \ref{para:mugeneraldef}.    Note that for any two  objects $C, D \in \Ac$ such that $Z(C), Z(D)\neq 0$, we have $\mu_Z (C) \prec \mu_Z (D) \Leftrightarrow\phi_Z (C) \prec \phi_Z (D)$.

(i)-(iii): Since $Z(M), Z(E)$ are both nonzero, for  $v \gg 0$, the two by two determinant in  \eqref{eq:S-stabfuncdef} equals  the area of the parallelogram formed by $Z(M)(v)$ and $Z(E)(v)$ on the complex plane, given by $|Z(M)(v)| |Z(E)(v)|\sin{\theta}$ where $\theta$ is the angle from $Z(M)$ to $Z(E)$ measured in the counterclockwise direction.  The  equivalences in (i), the first part of (ii), and (iii) then follow immediately.

Note that  for any $0 \neq C\in \Ac$ we have  $\Im Z (C)=0 \Leftrightarrow \phi_Z (C)=1 \Leftrightarrow\mu_Z (C)=\infty$, regardless of whether $Z(C)$ itself is zero.  Thus the second part of (ii) follows from the first part of (ii).

(iv): Suppose $E$ is $S$-semistable in $\Ac$.  Then $S(E)=0$ by definition and $\phi_Z(E)=\phi_Z (M)$ by (ii).  Take any short exact sequence $0\to E' \to E \to E'' \to 0$ in $\Ac$ where $E', E'' \neq 0$.  By the $S$-semistability of $E$, we have $S(E') \preceq 0$.  By our assumption on $E$, we also have $Z(E') \neq 0$, and so replacing $E$ by $E'$ in (i) and (ii) gives $\phi_Z (E') \preceq \phi_Z (M) = \phi_Z (E)$.  Since $Z(E')$ is nonzero, this implies $\phi_Z (E') \preceq \phi_Z (E'')$.

For the converse, suppose $\phi_Z (E)=\phi_Z (M)$ and $E$ is $Z$-semistable in $\Ac$.  Take any $0 \neq E' \subsetneq E$ in $\Ac$ and complete it to a short exact sequence $0 \to E' \to E \to E'' \to 0$ in $\Ac$.  Then $E'' \neq 0$, and by the $Z$-semistability of $E$ we have $\phi_Z (E') \preceq \phi_Z (E'')$.  We also have $Z(E') \neq 0$ by our assumption on $E$.  We now have two cases:
\begin{itemize}
\item[(a)] If $Z(E'') =0$, then $Z(E')=Z(E)$ and so $\phi_Z (E')=\phi_Z (E)=\phi_Z (M)$, implying $S(E')=0$ by (ii).
\item[(b)] If $Z(E'') \neq 0$, then $Z(E'), Z(E'')$ are both nonzero, and the assumption $\phi_Z (E') \preceq \phi_Z (E'')$ implies $\phi_Z (E') \preceq \phi_Z (E)=\phi_Z (M)$, which in turn implies $S(E') \preceq 0$ by (i) and (ii).
\end{itemize}
Hence $E$ is $S$-semistable in $\Ac$.

(v) The first part of (v) follows from  the same proof as the first part of (iv), with the argument adapted in the obvious manner for $E$ being $S$-stable.

For the second part of (v), we assume that $Z$ is a polynomial stability function on $\Ac$.  This means that every nonzero object $C$ of $\Ac$ satisfies $Z(C)\neq 0$.  The proof for the second part of (iv) then applies with the obvious adjustments, the main difference being that since $E'' \neq 0$, we must have $Z(E'') \neq 0$, and so case (a) in the proof of (iv) does not occur.
\end{proof}

Lemma \ref{lem:S-stabfuncdef}(v) suggests that when $Z$ is a weak polynomial stability function on a heart $\Ac$ that is not a polynomial stability function, i.e.\ $\Ac_{\mathrm{ker} Z} \neq \{0\}$, the notion of $Z$-stability is weaker than the notion of $S$-stability in general (with $M, S$ chosen as in the Lemma).  A concrete example is as follows:

\begin{egsub}
Suppose $X$ is a smooth projective surface and $\omega$ an ample divisor on $X$.  Let $\Ac = \Coh (X)$ and $Z (E) = -\omega \ch_1(E) + i \ch_0(E)$.  Then $(Z,\Ac)$ is a weak polynomial stability condition on $D^b(X)$ with respect to $(0,1]$, and $\OO_X$ is a $Z$-stable object in $\Ac$.  Nonetheless, $Z$ is not a polynomial stability function on $\Ac$ since $\Ac_{\mathrm{ker} Z}=\Coh^{\leq 0}(X)$.  If we choose $M = \OO_X$  and define $S$ as in Lemma \ref{lem:S-stabfuncdef}, then $\OO_X$ is not $S$-stable because, for any 0-dimensional subscheme $W$ of $X$, if we write $I_W$ to denote its ideal sheaf then $0 \neq I_W \subsetneq \OO_X$ but $S(I_W)=S(\OO_X)=0$.
\end{egsub}

\begin{remsub}
Previous constructions that pass between  King's stability and Bridgeland stability can be found in \cite[6.2]{bridgeland2017scattering} and \cite[Lemma 7.1.3]{Bayer2017}.
\end{remsub}

Given any polynomial stability condition on a triangulated category, the following lemma shows that the associated weight function constructed in Lemma \ref{lem:S-stabfuncdef} always satisfies a refinement property.

\begin{lem}\label{lem:HNgivesrefinement}
Suppose $(Z,\Ac)$ is a weak polynomial stability condition with respect to $(0,1]$.  Let $M$ be a fixed  object of $\Ac$ such that $Z(M) \neq 0$, and let $S$ be the weight function \eqref{eq:S-stabfuncdef}.  Then for any polynomial phase function $\phi_0$ satisfying $0 \prec \phi_0 \preceq 1$, the weight function $S$ satisfies the refinement property with respect to the torsion pair
\[
  \left( \Pc_Z (\phi_0,1], \Pc_Z (0,\phi_0] \right)
\]
in $\Ac$.
\end{lem}

\begin{proof}
That $\left( \Pc_Z (\phi_0,1], \Pc_Z (0,\phi_0] \right)$ is a torsion pair in $\Ac$ follows from the HN property of $Z$.
For any nonzero $S$-semistable object $E$ in $\Ac$, we have $S(E)=0$ and hence $\phi_Z(E)=\phi_Z(M)$ by Lemma \ref{lem:S-stabfuncdef}(ii).  Depending on whether $\phi_Z (E) \succ \phi_0$ or $\phi_Z (E) \preceq \phi_0$ (where $\phi_Z (E)$ and $\phi_0$ are comparable by  Lemma \ref{lem:seriescomparability}), the object $E$ lies in either $\Pc_Z (\phi_0,1]$ or $\Pc_Z (0,\phi_0]$, respectively.

Lastly, suppose every nonzero $S$-semistable object $E$ in $\Ac$ lies in $\Pc_Z (\phi_0,1]$ (resp.\ $\Pc_Z (0,\phi_0]$); then for every nonzero object $F$ in $\Pc_Z (0,\phi_0]$ (resp.\ $\Pc_Z (\phi_0,1]$), we have $\phi_Z(E) \succ \phi_Z (F)$ (resp.\ $\phi_Z (E) \prec \phi_Z (F)$), which is equivalent to $0 =S(E)\succ S(F)$ (resp.\ $0 =S(E)\prec S(F)$) by Lemma \ref{lem:S-stabfuncdef}(i) (resp.\ Lemma \ref{lem:S-stabfuncdef}(iii)).  Hence $S$ satisfies the refinement property with respect to $\left( \Pc_Z (\phi_0,1], \Pc_Z (0,\phi_0] \right)$.
\end{proof}

\section{Polynomial stability and Configuration III}\label{sec:configIII}

In this section, we give a configuration of weak polynomial stability functions, and prove a general `preservation of stability' result under this configuration without assuming the Harder-Narasimhan property.

\paragraph[Group actions on Bridgeland stability] For \label{para:gpactionBri} a triangulated category $\Dc$, recall that the space $\Stab (\Dc)$ of Bridgeland stability conditions on $\Dc$ possesses a right action by the group $\wt{\mathrm{GL}}^+\!(2,\mathbb{R})$, the universal covering space of $\mathrm{GL}^+\!(2,\mathbb{R})$, and a left action by $\Aut (\Dc)$, the group of exact autoequivalences of $\Dc$, and that these two actions commute (see \cite[Lemma 8.2]{StabTC} or \cite[Remark 5.14]{MSlec}).   Elements of $\wt{\mathrm{GL}}^+\!(2,\mathbb{R})$ can be considered as pairs $(T,g)$ where $T : \mathbb{R}^2 \to \mathbb{R}^2$ is an orientation-preserving $\mathbb{R}$-linear isomorphism, and $g: \mathbb{R} \to \mathbb{R}$ an increasing map such that $g(x+1)=g(x)+1$ for all $x \in \mathbb{R}$, with the compatibility that the induced maps  on $(\mathbb{R}^2\setminus \{0\})/\mathbb{R}_{>0} =  \mathbb{R}/2\mathbb{Z}$ are the same.

In particular, if $(Z, \Ac)$ is a  Bridgeland stability condition with respect to $(a,a+1]$ on a triangulated category $\Dc$ and $\Pc$ denotes its slicing, then the data of $(Z,\Ac)$  together with the data $(a,a+1]$ are  determined by the pair $(Z,\Pc)$, and vice versa \cite[Proposition 5.3]{StabTC}.  Given an element $(T,g)$ of $\wt{\mathrm{GL}}^+\!(2,\mathbb{R})$, we set $(Z,\Pc)\cdot (T,g)=(Z',\Pc')$ where   $Z'(-) = T^{-1}\cdot (Z(-))$ and $\Pc'(x)=\Pc(g(x))$ for all $x\in \mathbb{R}$.  On the other hand, given an element $\Phi \in \Aut (\Dc)$, we set $\Phi\cdot (Z,\Pc)=(Z'',\Pc'')$  where $Z'' = Z \circ (\Phi^K)^{-1}$ and $\Pc'' (x) = \Phi (\Pc (x))$ for all $x \in \mathbb{R}$.  Using the heart notation, we can also say $\Phi \cdot (Z,\Pc)$ is the Bridgeland stability condition $(Z'',\Phi (\Pc (a,a+1]))$ with respect to $(a,a+1]$.

We will now extend these group actions to polynomial stability conditions.

\paragraph[Induced map on polynomial phase functions] We \label{para:inducedGammaP} define  induced maps on the set of polynomial phase functions using certain elements of  $\mathrm{GL} (2,\RLoovc)$.    Given an element $T \in \mathrm{GL} (2,\RLoovc)$, there exists $v_0 \in \mathbb{R}_{>0}$ such that  $T_v := T(v)$ lies in  $\mathrm{GL} (2,\mathbb{R})$ for all $v \geq v_0$.  This defines a path
\[
  \mathrm{Pa}(T) : [v_0, \infty) \to \mathrm{GL}(2,\mathbb{R}) : v \mapsto T_v
\]
in $\mathrm{GL} (2,\mathbb{R})$.  We now set
\[
  \GLlp = \{ T \in \mathrm{GL}(2,\RLoovc) : T_v \in \mathrm{GL}^+(2,\mathbb{R}) \text{\quad for $v \gg 0$}\}.
\]

\subparagraph Suppose $T\in \GLlp$.  \label{para:liftedpath} There exists $v_1 \in \mathbb{R}_{>0}$ such that $T_v\in \mathrm{GL}^+\!(2,\mathbb{R})$ for all $v \geq v_1$, i.e.\ by restricting the domain if necessary, we can assume the path $\mathrm{Pa}(T)$ is a path in $\mathrm{GL}^+\!(2,\mathbb{R})$, and hence can be lifted to a path in the universal covering space $\wt{\mathrm{GL}}^+\!(2,\mathbb{R})$ of $\mathrm{GL}^+\!(2,\mathbb{R})$:
\begin{equation}\label{eq:liftedpath}
  \Patilde (T) : [v_1,\infty) \to \wt{\mathrm{GL}}^+\!(2,\mathbb{R}) : v \mapsto (T_v,g_v).
\end{equation}
Now, given a polynomial phase function $\phi$, take any $0 \neq f\in \CLoovc$ that gives rise to $\phi$, i.e.\ any $f$ satisfying
\[
  f(v) \in \mathbb{R}_{>0} e^{i\pi \phi (v)} \text{\quad for $v\gg 0$}.
\]
By regarding $\CLoovc$ as an $\RLoovc$-vector space in the obvious manner, elements of $\mathrm{GL} (2,\RLoovc)$ act on $\CLoovc$ from the left and  we have
\[
  (Tf)(v) = T_v  \cdot f(v) \in \mathbb{R}_{>0}e^{i\pi g_v (\phi (v))}\text{\quad for $v\gg 0$}.
\]
As a result, $v\mapsto g_v(\phi (v))$ defines a polynomial phase function.  Writing $P$ to denote the set of all polynomial phase functions,  we obtain an order-preserving bijection $\Gamma_T : P \to P : \phi \mapsto \phi'$ where
\[
  \phi' (v) = g_v (\phi (v)) \text{ for $v \gg 0$}.
\]
In this notation, we have
\[
  (Tf)(v) \in \mathbb{R}_{>0} \cdot e^{i\pi (\Gamma_T \phi)(v)} \text{\quad for $v \gg 0$}.
\]
Note that the definition of $\Gamma_T$ is independent of the choice of $f$ since $T$ is a $\RLoovc$-linear automorphism of $\CLoovc$.  On the other hand, the definition of $\Gamma_T$ is dependent on the choice of the lifted path $\Patilde (T)$ in $\wt{\mathrm{GL}}^+\!(2,\mathbb{R})$ (and in particular, dependent on the choice $\{g_v\}_v$).

It is clear, that for any $T \in \GLlp$, we have $\Gamma_T(\phi+1)=\Gamma_T(\phi)+1$.

\subparagraph Let  \label{para:Torderpres}  $\Dc$ be a triangulated category  and $Z : K(\Dc) \to \CLoovc$ a weak polynomial stability function  on a heart $\Ac$ with respect to $(\phi_0, \phi_0+1]$, for some polynomial phase function  $\phi_0$.  For any $T \in \GLlp$, set $(TZ)(E)=T\cdot Z(E)$.  Then  $TZ$ is a weak polynomial stability function  on $\Ac$   with respect to $(\Gamma_T(\phi_0), \Gamma_T(\phi_0)+1]$ for any choice of $\Gamma_T$ from \ref{para:liftedpath}.  Moreover, if $E_1, E_2$ are nonzero objects of $\Ac$, then  $\phi_Z (E_1) \prec \phi_Z (E_2)$ if and only if $\phi_{TZ}(E_1) \prec \phi_{TZ}(E_2)$.  Additionally,  a nonzero object $E\in \Dc$ is $Z$-semistable (resp.\ $Z$-stable) in $\Ac$ of phase $\phi$ if and only if it is $(TZ)$-semistable (resp.\ $(TZ)$-stable) in $\Ac$ of phase $\Gamma_T(\phi)$.

\subparagraph[Group of lifted paths] Let \label{para:gpactionliftedpaths} $\Patilde$ denote the set of germs of all possible lifted paths of the form $\Patilde (T)$ in \eqref{eq:liftedpath}, for some $T \in \GLlp$.  That is, elements of  $\Patilde$ are represented by paths in $\wt{\mathrm{GL}}^+\! (2,\mathbb{R})$ of the form \eqref{eq:liftedpath}, and two such paths
\begin{align*}
 \Patilde (T) &: [v_1,\infty) \to \wt{\mathrm{GL}}^+\!(2,\mathbb{R}) : v \mapsto (T_v,g_v), \\
 \Patilde (T') &: [v_2,\infty) \to \wt{\mathrm{GL}}^+\!(2,\mathbb{R}) : v \mapsto (T'_v,g'_v)
\end{align*}
are identified if, for some $v_3 \geq \max{\{v_1, v_2\}}$, we have $\Patilde(T)|_{[v_3,\infty)} = \Patilde(T')|_{[v_3,\infty)}$.  The set $\Patilde$ has a natural group structure with multiplication given by
\[
  \Patilde (T)\Patilde (T') : [\max{\{v_1,v_2\}}, \infty) \to \wt{\mathrm{GL}}^+\!(2,\RR) : v \mapsto (T_v,g_v)(T'_v,g'_v)
\]
which is clearly a lift of the path $[\max{\{v_1,v_2\}},\infty) \to \mathrm{GL}^+(2,\RR) : v \mapsto T_v T'_v$.

\subparagraph We \label{para:gpactionpolystabdefn} are now ready to extend  the group actions on Bridgeland stability conditions in \ref{para:gpactionBri} to the setting of polynomial stability conditions.  Recall from \ref{para:slicing} the definition of the set $\Stabpol (\Dc)$ of polynomial stability conditions on $\Dc$.  The group $\Patilde$ acts on $\Stabpol (\Dc)$ as follows: given an element
\[
\Patilde (T) : [v_1,\infty) \to \wt{\mathrm{GL}}^+\!(2,\RR) : v \mapsto (T_v,g_v)
\]
 of $\Patilde$ and an element $(Z,\Pc)$ of $\Stabpol (X)$ where $\Pc$ denotes the slicing, we set $(Z,\Pc)\cdot \Patilde (T) = (Z',\Pc')$ where $Z' = T^{-1} Z (-)$ and $\Pc' (\phi) = \Pc (\Gamma_T (\phi))$.  On the other hand, if $\Phi : \Dc \to \Uc$ is an exact equivalence of triangulated categories,  $\Phi \cdot (Z,\Pc) =(Z'',\Pc'')$ is the polynomial stability condition on $\Uc$ where $Z'' = Z \circ (\Phi^K)^{-1}$ and $\Pc'' (\phi ) = \Phi (\Pc(\phi))$.  When $\Dc=\Uc$, it is easy to check that the actions of $\Patilde$ and $\Aut (\Dc)$ on $\Stabpol (\Dc)$ commute.

\subparagraph For \label{para:marker1} any $r \in\mathbb{R}$, we will set $\HH[r] := \mathbb{R}_{>0}e^{i\pi(r, r+1]}$ and simply write $\mathbb{H}$ for  $\mathbb{H}[0]$, the  usual "upper half plane" in the literature on stability conditions.  For any $f\in \CLoovc$ and any polynomial phase function  $\psi$, we will  write $f \in \HH [\psi]$ to mean $f(v) \in \HH [\psi (v)]$ for $v \gg 0$.

\begin{egsub}[Dilation]
Suppose $d_1, d_2$ are elements of $\RLoovc$ such that $d_1, d_2 \succ 0$.  Then $T = \begin{pmatrix} d_1 & 0 \\ 0 & d_2 \end{pmatrix} \in \mathrm{GL}(2,\RLoovc)$ gives an $\RLoovc$-linear map (also denoted by $T$ by abuse of notation)
\begin{align*}
  T : \CLoovc &\to \CLoovc \\
    f&\mapsto( d_1 \Re f )+ i (d_2 \Im f)
\end{align*}
and $T$ is an element of $\GLlp$.  Let $\Gamma_T$ denote the choice of induced map $P \to P$ on the set of polynomial phase functions, such that $\Gamma_T$ fixes elements of $\tfrac{1}{2}\mathbb{Z}$.  Notice that for every  $v \gg 0$, the map $T(v) : \mathbb{R}^2\to \mathbb{R}^2$  satisfies the following properties:
\begin{enumerate}
\item  $T(v)$  preserves each half-axis.
\item $T(v)$ preserves each quadrant.
\end{enumerate}
Under our assumption that $\Gamma_T$ fixes elements of $\tfrac{1}{2}\mathbb{Z}$, property (2) implies  that if $\phi$ is a polynomial phase function such that $\phi \in \HH [k]$ for some $k \in \tfrac{1}{2}\mathbb{Z}$, then $\Gamma_T (\phi) \in \HH [k]$ as well.
\end{egsub}

\begin{egsub}[Rotation]\label{eg:mulbynegi}
Consider the $\RLoovc$-linear map $T : \CLoovc \to \CLoovc$ given by multiplication by $(-i)$, i.e.\
\[
  Tf := (-i)f = (\Im f) + i (-\Re f)  \text{\quad for $f \in \CLoovc$}.
\]
Then we can choose $\Gamma_T$ to be $\Gamma_T (\phi)=\phi-\tfrac{1}{2}$.  Also,  under the identification $\CLoovc \to (\RLoovc)^{\oplus 2} : f \mapsto \begin{pmatrix} \Re f \\ \Im f \end{pmatrix}$, we have
\[
  Tf = T \begin{pmatrix} \Re f \\ \Im f \end{pmatrix} = \begin{pmatrix} 0 & 1 \\ -1 & 0 \end{pmatrix} \begin{pmatrix} \Re f \\ \Im f \end{pmatrix}.
\]
Clearly, $T$ lies in $\mathrm{GL}^+\!(2,\mathbb{R})\subset \GLlp$.
\end{egsub}

\paragraph[Configurations III-w and III] We  \label{para:configIII} say we are in Configuration III-w when we have the following setting:
\begin{itemize}
\item[(a)] $\Phi : \Dc \to \Uc$ and $\Psi : \Uc \to \Dc$ are exact equivalences between triangulated categories $\Dc, \Uc$ satisfying $\Psi \Phi \cong \mathrm{id}_\Dc [-1]$ and $\Phi \Psi \cong \mathrm{id}_\Uc [-1]$.
\item[(b)] $\Ac, \Bc$ are hearts of bounded t-structures on $\Dc, \Uc$, respectively, such that $\Phi \Ac \subset D^{[0,1]}_\Bc$.
\item[(c)] There exist \emph{weak} polynomial stability functions $Z_\Ac : K(\Dc) \to \CLoovc$ and $Z_{\Bc} : K(\Uc) \to \CLoovc$ on $\Ac, \Bc$  with respect to $(a,a+1], (b,b+1]$, respectively, together with an element $T \in \GLlp$ such that
    \begin{equation}\label{eq:AG45-108-17}
 Z_{\Bc} (\Phi E) = T   Z_{\Ac}(E) \text{\quad for all $E\in \Dc$}
\end{equation}
    and a choice of an induced map $\Gamma_T$ such that
    \begin{equation}\label{eq:AG47-23-1}
    a \prec \Gamma_T^{-1}(b)  \prec a+1.
    \end{equation}
\end{itemize}
We   say we are in Configuration III if conditions (a), (b), and (c) all hold, with the additional condition that the stability functions $Z_\Ac, Z_\Bc$ in (c) are polynomial stability functions, i.e.\ $Z_\Ac(C)\neq 0$ (resp.\ $Z_\Bc (C) \neq 0$) for every nonzero object $C$ in $\Ac$ (resp.\  $\Bc$).  We do not impose the HN property in either Configuration III-w or Configuration III.

Note that  Conditions (a) and (b) in Configuration III-w  are exactly the same as Conditions (a) and (b) in Configuration II in \ref{para:config2}.  Also, imposing \eqref{eq:AG45-108-17} is equivalent to imposing commutativity in the lower square of the following  diagram:
\[
  \xymatrix{
  \Dc \ar[r]^\Phi \ar[d]_{[\, \, ]} & \Uc \ar[d]^{[\,\,]} \\
  K(\Dc) \ar[r]^{\Phi^K} \ar[d]_{Z_\Ac} & K(\Uc) \ar[d]^{Z_\Bc} \\
  \CLoovc \ar[r]^T & \CLoovc
  }.
\]

Although all the applications in Section \ref{sec:ellcur} and later in this article use Configuration III, we prove some results for Configuration III-w in this section as they may be useful in future work.

\subparagraph[Notation] Under the setting of Configuration III-w, for any object $0 \neq E \in \Bc$ and any $j \in \mathbb{Z}$, we will set $\phi_\Bc ( E[j]) =\phi_\Bc (E) + j$.  Also, for convenience, we will write $\phi_\Ac, \phi_\Bc$ for the phase functions $\phi_{Z_\Ac}, \phi_{Z_\Bc}$, respectively.

\begin{lem}\label{lem:phiGammaTeq}
 Assume Configuration III-w.  For any  object $E$ in $\Ac$ with $Z(E) \neq 0$ and any $j \in \mathbb{Z}$,  we have
\[
  Z_\Bc (E [j])\in \mathbb{R}_{>0}e^{i\pi \phi_\Bc (E[j])}.
\]
Moreover, if $E$  is $\Phi_\Bc$-WIT$_j$ where $j=0$ or $1$, then
\begin{equation}\label{eq:AG48-20-1}
\phi_\Bc (\Phi E)=\Gamma_T(\phi_\Ac (E)).
\end{equation}
\end{lem}

\begin{proof}
The first assertion is clear.  To see the second assertion, suppose $E$ is $\Phi_\Bc$-WIT$_j$.  Using the first assertion together with \eqref{eq:AG45-108-17}, we obtain
\[
  (-1)^j TZ_\Ac (E) = (-1)^j Z_\Bc (\Phi E)=Z_\Bc (\Phi E [j]) \in \mathbb{R}_{>0} \cdot e^{i\pi \phi_\Bc (\Phi E [j])}
\]
while we also have $TZ_\Ac (E) \in \mathbb{R}_{>0} \cdot e^{i \pi (\Gamma_T (\phi_\Ac (E)))}$.  Altogether, these give
\[
  \Gamma_T (\phi_\Ac (E)) + j + 2k = \phi_\Bc (\Phi E [j]) \text{\quad for some $k \in \mathbb{Z}$}.
\]
We would be done if we can show that $k=0$.

Since $\Phi E \in \Bc [-j]$, we have
\begin{equation}\label{eq:AG48-20-2}
\phi_\Bc (\Phi E) \in (b-j,b+1-j].
\end{equation}
On the other hand, since $E\in \Ac$, we have
\begin{equation}\label{eq:AG48-20-3}
\Gamma_T (\phi_\Ac (E))+2k \in (\Gamma_T(a)+2k,\Gamma_T (a+1)+2k].
\end{equation}
Hence the intersection of the intervals
\[
(b-j,b+1-j] \cap (\Gamma_T(a)+2k,\Gamma_T (a+1)+2k]
\]
must be nonempty, i.e.\
\[
(\Gamma_T^{-1}(b)-j,\Gamma_T^{-1}(b)+1-j] \cap (a+2k,a+1+2k]
\]
is nonempty.  The assumption \eqref{eq:AG47-23-1} now gives $(\Gamma_T^{-1}(b)-j,\Gamma_T^{-1}(b)+1-j] \subset (a-j,a+2-j)$, so in order for the above intersection to be nonempty, we are forced to have $k=0$, proving \eqref{eq:AG48-20-1} for $\Phi_\Bc$-WIT$_j$ objects $E$ in $\Ac$.
\end{proof}

\begin{remsub}
The proof of Lemma \ref{lem:phiGammaTeq} stills works when we have a non-strict inequality on the right of \eqref{eq:AG47-23-1}, i.e.\
\[
  a \prec \Gamma_T^{-1}(b) \preceq a+1.
\]
\end{remsub}

\paragraph[Symmetry in Configuration III-w] Assume \label{para:ABconfigsymm} Configuration III-w.  Note that conditions (a) and (b) in  \ref{para:configIII} are symmetric in $\Phi, \Psi$ and $\Ac, \Bc$ by Lemma \ref{para:AcBcPhitilt}.  To see the symmetry in \ref{para:configIII}(c), let us write $F = \Phi E$ so that  $E=\Psi F [1]$.  Then   \eqref{eq:AG45-108-17} yields $T^{-1} Z_\Bc (F) = Z_\Ac (\Psi F [1])$, i.e.\
\[
Z_\Ac (\Psi F) = (-T^{-1})Z_\Bc (F)
\]
where $\wh{T} := -T^{-1}=(-1)T^{-1}$ also lies in $\GLlp$.  If we choose $\Gamma_{\wh{T}} : P \to P$ to be the induced map on polynomial phase functions (see \ref{para:liftedpath}) such that $\Gamma_{\wh{T}} = (\Gamma_T)^{-1} -1$, then $(\Gamma_{\wh{T}})^{-1}=\Gamma_T + 1$ and  the relation \eqref{eq:AG47-23-1} implies
\[
  \Gamma_T (a) \prec b \prec \Gamma_T (a) + 1,
\]
which in turn implies
\[
  b \prec \Gamma_T (a)+1  \prec b+1
\]
where $\Gamma_T (a)+1= (\Gamma_{\wh{T}})^{-1} (a)$.  Thus overall,  Configuration III-w is symmetric in the following pairs: $\Phi$ and $\Psi$, $\Ac$ and $\Bc$, $Z_\Ac$ and $Z_\Bc$, $a$ and $b$, $T$ and $\wh{T}$ and also $\Gamma_T$ and $\Gamma_{\wh{T}}$.

\begin{thm}\label{thm:configIII-w}
Assume Configuration III-w.  Then
\begin{itemize}
\item[(i)] For any $\Phi_\Bc$-WIT$_0$ object $E$  and any $\Phi_\Bc$-WIT$_1$ object $F$ in $\Ac$ such that where $Z_\Ac (E), Z_\Ac (F) \neq 0$, we have
    \[
      \phi_{\Ac} (F) \preceq \Gamma_T^{-1}(b)  \prec \phi_\Ac (E).
    \]
\item[(ii)] Suppose $E$ is a   $Z_\Ac$-semistable object in $\Ac$ such that  $Z_\Ac (E)\neq 0$ and
    \begin{itemize}
    \item the only strict $\Ac$-injection from an object in $\Ac_{\mathrm{ker} Z_\Ac} \cap W_{0,\Phi,\Ac,\Bc}$ to $E$ is the zero morphism;
    \item the only strict $\Ac$-surjection from $E$ to an object in $\Ac_{\mathrm{ker} Z_\Ac} \cap W_{1,\Phi,\Ac,\Bc}$ is the zero morphism.
    \end{itemize}
    Then $E$ is either $\Phi_\Bc$-WIT$_0$ or $\Phi_\Bc$-WIT$_1$.
\item[(iii)]  Suppose  $E\in \Ac$ satisfies the same hypotheses as in (ii).  Then
\[
 \text{$E$ is $Z_\Ac$-semistable }\Rightarrow
 \text{ $\Phi E$ is $Z_\Bc$-semistable}.
\]
\end{itemize}
\end{thm}

\begin{proof}
(i) Take any $E, F$ as described.  We have  $\Phi E \in \Bc$ and so $b \prec \phi_\Bc (\Phi E)\preceq b+1$, implying $\Gamma_T^{-1} (b) \prec \phi_\Ac (E)$ by \eqref{eq:AG48-20-1}.  On the other hand, $\Phi F \in \Bc[-1]$ so $b-1\prec \phi_\Bc (\Phi F) \preceq b$, implying $\phi_\Ac (F)\preceq \Gamma_T^{-1} (b)$ by \eqref{eq:AG48-20-1}.  Part (i) thus follows.

(ii)  By Lemma \ref{lem:AG46-80-1}(ii), we have an  $\Ac$-short exact sequence
\[
0 \to E_0 \to E \to E_1 \to 0
\]
where $E_i$ is $\Phi_\Bc$-WIT$_i$ for $i=0,1$.  Suppose $E_0 \neq 0$.  Then we must have $Z_\Ac(E_0) \neq 0$, and hence  $\Gamma_T^{-1}(b) \prec \phi_\Ac (E_0)$ by part (i).  If we also have $E_1 \neq 0$, then $Z_\Ac(E_1)$ must be nonzero, in which case part (i) also gives $\phi_\Ac (E_1) \preceq \Gamma^{-1}_T(b)$.  Then $\phi_\Ac (E_0) \succ \phi_\Ac (E_1)$, contradicting the semistability of $E$.

(iii) Let $E$ be as described, and suppose $E$ is $Z_\Ac$-semistable in $\Ac$.  By part (ii), $E$ is either $\Phi_\Bc$-WIT$_0$ or $\Phi_\Bc$-WIT$_1$.  Although the proofs of the two  cases are similar, we include them both for the sake of completeness.

\noindent
\textbf{Case 1}: $E$ is $\Phi_\Bc$-WIT$_0$.  Take any $\Bc$-short exact sequence
\[
  0 \to M \to \Phi E \to N \to 0
\]
where $M, N \neq 0$.  That $E$ is $\Phi_\Bc$-WIT$_0$ implies $\Phi E$ is $\Psi_\Ac$-WIT$_1$ by Lemma \ref{lem:AG46-80-1}(i).  Hence $M$ is also $\Psi_\Ac$-WIT$_1$ since $W_{1,\Psi,\Bc,\Ac}$ is a torsion-free class in $\Bc$ by Lemma \ref{lem:AG46-80-1}(ii).  We then have the $\Ac$-long exact sequence
\[
0 \to \Psi^0_\Ac N \to \Psi M [1] \overset{\alpha}{\to} E \to \Psi^1_\Ac N \to 0.
\]
If $\alpha=0$, then $\Psi^0_\Ac N \cong \Psi M [1]$ where $\Psi^0_\Ac N$ is $\Phi_\Bc$-WIT$_1$ while $\Psi M [1]$ is $\Phi_\Bc$-WIT$_0$ by Lemma \ref{lem:AG46-80-1}(i) again; this forces $\Psi^0_\Ac N = \Psi M [1]=0$, i.e.\ $M=0$, a contradiction.  Hence $\image \alpha \neq 0$, and the $Z_\Ac$-semistability of $E$ implies $\phi_\Ac (\image \alpha) \preceq \phi_\Ac (E)$.  On the other hand, since $\Psi M [1]$ is $\Phi_\Bc$-WIT$_0$, the $\Ac$-quotient $\image \alpha$ is also $\Phi_\Bc$-WIT$_0$ (since $W_{0,\Phi,\Ac,\Bc}$ is a torsion class in $\Ac$).  Then the first vanishing condition in (ii) ensures $Z_\Ac(\image \alpha)\neq 0$.  Part (i) then gives $\Gamma^{-1}_T(b) \prec \phi_\Ac (\image \alpha)$.  Now we divide further into two cases:

If $Z_\Ac(\Psi^0_\Ac N)\neq 0$, then from part (i) we have $\phi_\Ac (\Psi^0_\Ac N) \preceq \Gamma^{-1}_T(b)$; together with $\Gamma^{-1}_T(b) \prec \phi_\Ac (\image \alpha)$ from above and noting $Z_\Ac(\Psi^0_\Ac N), Z_\Ac(\image \alpha)$ are both nonzero, we obtain $\phi_\Ac (\Psi M [1])\prec \phi_\Ac (\image \alpha)$ and hence $\phi_\Ac (\Psi M [1])\prec \phi_\Ac (E)$.  Then $\phi_\Bc (M) \prec \phi_\Bc (\Phi E)$ by \eqref{eq:AG48-20-1} and hence $\phi_\Bc (M)\prec \phi_\Bc (N)$.

If $Z_\Ac(\Psi^0_\Ac N)=0$, then $Z_\Ac(\Psi M[1])=Z_\Ac(\image \alpha)\neq 0$.  If we also have $Z_\Ac(\Psi^1_\Ac N)=0$, then $Z_\Bc (N)=0$ by \eqref{eq:AG45-108-17} and $\phi_\Bc (M) \preceq \phi_\Bc (N)$ follows.  On the other hand, if $Z_\Ac (\Psi^1_\Ac N)\neq 0$, then the $Z_\Ac$-semistability of $E$ gives $\phi_\Ac (\Psi M [1]) = \phi_\Ac (\image \alpha) \preceq \phi_\Ac (\Psi^1_\Ac N)= \phi_\Ac (\Psi N [1])$, which implies $\phi_\Bc (M) \preceq \phi_\Bc (N)$.

Hence $\Phi E$ is $Z_\Bc$-semistable in Case 1.

\noindent
\textbf{Case 2:} $E$ is $\Phi_\Bc$-WIT$_1$. Take any $\Bc$-short exact sequence
\[
  0 \to M \to \Phi E[1] \to N \to 0
\]
where $M, N \neq 0$. Since $\Phi E [1]$ is a $\Psi_\Ac$-WIT$_0$ object in $\Bc$, its $\Bc$-quotient $N$ is also $\Psi_\Ac$-WIT$_0$ by Lemma \ref{lem:AG46-80-1}(ii), giving us the $\Ac$-exact sequence
\[
  0 \to \Psi^0_\Ac M \to E \overset{\alpha}{\to} \Psi  N \to \Psi^1_\Ac M \to 0.
\]
As in Case 1, we must have $\image \alpha \neq 0$, and the $Z_\Ac$-semistability of $E$ implies $\phi_\Ac (E)\preceq \phi_\Ac (\image \alpha)$.  Since $\Psi N$ is $\Phi_\Bc$-WIT$_1$, so is $\image \alpha$, and by the second vanishing in (ii) we must have $Z_\Ac (\image \alpha)\neq 0$.  We now divide further into two cases:

If $Z_\Ac(\Psi^1_\Ac M) \neq 0$, then by  (i), we have $\phi_\Ac (\image \alpha) \prec \phi_\Ac (\Psi^1_\Ac M)$ and hence $\phi_\Ac (\image \alpha) \prec \phi_\Ac (\Psi N) \prec  \phi_\Ac (\Psi^1_\Ac M)$.  On the other hand, the $Z_\Ac$-semistability of $E$ gives $\phi_\Ac (E)\preceq \phi_\Ac (\image \alpha)$, giving us $\phi_\Ac (E) \prec \phi_\Ac (\Psi N)$, which in turn gives $\phi_\Bc (\Phi E [1]) \prec \phi_\Bc (N)$ and then $\phi_\Bc (M) \prec \phi_\Bc (N)$.

If $Z_\Ac (\Psi^1_\Ac M)=0$, then $Z_\Ac (\image \alpha)=Z_\Ac (\Psi N)$ while $Z_\Ac (\Psi^0_\Ac M)= Z_\Ac (\Psi M )$ and, as above, the $Z_\Ac$-semistability of $E$ gives $\phi_\Ac (\Psi M)=\phi_\Ac (\Psi^0_\Ac M) \preceq \phi_\Ac (\image \alpha)=\phi_\Ac (\Psi N)$, and so $\phi_\Bc (M) \preceq \phi_\Bc (N)$.

Hence $\Phi E [1]$ is $Z_\Bc$-semistable in Case 2.
\end{proof}

When we are in Configuration III, the  functions $Z_\Ac, Z_\Bc$ are polynomial stability functions, meaning $\Ac_{\mathrm{ker} Z_\Ac}, \Bc_{\mathrm{ker} Z_\Bc}$ are both trivial, giving us cleaner statements:

\begin{thm}\label{prop:paper30prop11-11ext}
Assume Configuration III.  Then
\begin{itemize}
\item[(i)] For any $\Phi_\Bc$-WIT$_0$ object $E$ and any $\Phi_\Bc$-WIT$_1$ object $F$ in $\Ac$ where $E,F \neq 0$, we have
    \[
      \phi_{\Ac} (F) \preceq \Gamma_T^{-1}(b)  \prec \phi_\Ac (E).
    \]
\item[(ii)] ($Z_\Ac$-semistable implies $\Phi_\Bc$-WIT$_i$) Any nonzero $Z_\Ac$-semistable object $E$ in $\Ac$ is either $\Phi_\Bc$-WIT$_0$ or $\Phi_\Bc$-WIT$_1$.
\item[(iii)] Let $M$ be any nonzero $\Phi_\Bc$-WIT$_i$ object in $\Ac$ where $i=0$ or $1$, and let $S_\Ac$ denote   be the weight function $S_{Z_\Ac',M} : K(\Dc) \to \CLoovc$  defined   as in \eqref{eq:S-stabfuncdef} using $Z_\Ac':=e^{i\pi (-a)}Z_\Ac$.  Then $S_\Ac$ satisfies  refinement-$i$  with respect to the torsion pair $(W_{0,\Phi,\Ac,\Bc}, W_{1,\Phi,\Ac,\Bc})$ in $\Ac$.
\item[(iv)] (correspondence of stability) For a nonzero object $E\in \Ac$,
\[
 \text{$E$ is $Z_\Ac$-semistable }\Leftrightarrow
 \text{ $\Phi E$ is $Z_\Bc$-semistable}.
\]
The same statement holds if we replace `semistable' by `stable.'
\item[(v)] Suppose $(Z_\Ac, \Ac)$ and $(Z_\Bc, \Bc)$ are polynomial stability conditions with respect to $(a,a+1], (b,b+1]$, respectively, and $\Pc_\Ac,  \Pc_\Bc$ denote their respective  slicings.  Then
    \begin{equation}\label{eq:Gepnereqpolystab}
    \Phi \cdot (Z_\Ac, \Pc_\Ac) = (Z_\Bc,\Pc_\Bc)\cdot \Patilde (T)
    \end{equation}
    where $\Patilde (T)$ is as defined in \eqref{eq:liftedpath}.
\end{itemize}
\end{thm}

Even though part (iv) of Theorem \ref{prop:paper30prop11-11ext} can be proved via symmetry in Configuration III and adapting the argument for Theorem \ref{thm:configIII-w}(iii) for the `stable' case, we give a proof using Configuration II.  This shows that all the applications of Configuration III in the later parts of this article can be traced back to Configuration II and hence Configuration I.

\begin{proof}
(i), (ii): These  follow from their counterparts in Theorem \ref{thm:configIII-w}.

(iii) Let $M, Z_\Ac', S_\Ac$ be as described.  Then $Z_\Ac'$ is a polynomial stability function on $\Ac$ with respect to $(0,1]$, and $Z_\Ac$-semistability is equivalent to $Z_\Ac'$-semistability (the phases are simply translated by $a$).  By Lemma \ref{lem:S-stabfuncdef}, every $S_\Ac$-semistable object $F$ in $\Ac$ is $Z_\Ac$-semistable with $\phi_\Ac (F)=\phi_\Ac (M)$; by parts (i) and (ii), $M$ and $F$ are both  $\Phi_\Bc$-WIT$_i$.  Then by part (i), all the nonzero $S_\Ac$-semistable objects in $\Ac$ lie in $W_{i,\Phi,\Ac,\Bc}$ for some fixed $i$.  If $i=0$ (resp.\ $i=1$), then for every nonzero object $G$ in $W_{1,\Phi,\Ac,\Bc}$ (resp.\ $W_{0,\Phi,\Ac,\Bc}$) we have $\phi_\Ac (G) \prec \phi_\Ac (M)$ (resp.\ $\phi_\Ac (G) \succ \phi_\Ac (M)$) by part (i), and so $S_\Ac (G)\prec 0$ (resp.\ $S_\Ac (G)  \succ 0$) by Lemma \ref{lem:S-stabfuncdef}.  Hence $S_\Ac$ satisfies the refinement property  with respect to $(W_{0,\Phi,\Ac,\Bc}, W_{1,\Phi,\Ac,\Bc})$.

(iv) By the symmetry in Configuration III (see \ref{para:ABconfigsymm}), we only need to prove the `only if' direction.  Let us fix a nonzero $Z_\Ac$-semistable object $E$ in $\Ac$.  Then   $E$ is $\Phi_\Bc$-WIT$_i$ for  $i=0$ or $1$ by part (ii).  Set $\wh{E}:= \Phi E [i]$, which lies in $\Bc$.  Let $S_\Ac$ be the weight function $S_{Z'_\Ac, E}$ as defined in \eqref{eq:S-stabfuncdef} where $Z'_\Ac=e^{i\pi (-a)}Z_\Ac$  as in part (iii), and let $S_\Bc$ be the weight function $S_{Z_\Bc',\wh{E}}: K(\Uc) \to \CLoovc$ defined as in \eqref{eq:S-stabfuncdef} where  $Z_\Bc' := e^{i\pi (-b)}Z_\Bc$.

We will now show that the weight functions  $S_\Ac$ and $(-1)^iS_\Bc$ satisfy condition \ref{para:config2}(c).  To see this, take any $F\in \Ac$ that is  $\Phi_\Bc$-WIT$_j$ where $j=0$ or $1$.  We need to show that $\sgn S_\Ac (F) = \sgn ((-1)^iS_\Bc) (\Phi F)$.  We divide into the following cases:
\begin{itemize}
\item $i=0, j=1$: By part (i), we have $\phi_\Ac (E) \succ \phi_\Ac (F)$ and hence $S_\Ac (F) \prec 0$ by Lemma \ref{lem:S-stabfuncdef}.  On the other hand, $\wh{E}$ is $\Psi_\Ac$-WIT$_1$ while $\wh{F}$ is $\Psi_\Ac$-WIT$_0$ by Lemma \ref{lem:AG46-80-1}(i), so by the symmetry in Configuration III and part (i), we have $\phi_\Bc (\wh{F}) \succ \phi_\Bc (\wh{E})$, implying $-S_\Bc (\Phi F)=S_\Bc (\Phi F [1])=S_\Bc (\wh{F})\succ 0$, i.e.\ $S_\Bc (\Phi F) \prec 0$. Hence  $\sgn S_\Ac (F)=\sgn ((-1)^iS_\Bc) (\Phi F)$ in this case.
\item $i=1, j=0$: The argument is similar to the previous case. By part (i), we have $\phi_\Ac (E) \prec \phi_\Ac (F)$ and hence $S_\Ac (F) \succ 0$ by Lemma \ref{lem:S-stabfuncdef}.  On the other hand, $\wh{E}$ is $\Psi_\Ac$-WIT$_0$ while $\wh{F}$ is $\Psi_\Ac$-WIT$_1$, so by the symmetry in Configuration III and part (i), we have $\phi_\Bc (\wh{F}) \prec \phi_\Bc (\wh{E})$, implying $S_\Bc (\Phi F)=S_\Bc (\wh{F})\prec 0$, i.e.\ $(-1)S_\Bc (\Phi F) \succ 0$. Hence  $\sgn S_\Ac (F)=\sgn ((-1)^iS_\Bc) (\Phi F)$ in this case.
\item $i=j$: Let us write $\phi_\Ac', \phi_\Bc'$ to denote the phase functions for the polynomial stability functions $Z_\Ac', Z_\Bc'$, respectively.  Then
\begin{align*}
S_\Ac (F)  \succ 0  &\Leftrightarrow \phi_\Ac' (F)  \succ \phi'_\Ac (E)  \text{\quad by Lemma \ref{lem:S-stabfuncdef}} \\
& \Leftrightarrow \phi_\Ac (F)  \succ \phi_\Ac (E)  \\
&\Leftrightarrow \phi_\Bc (\Phi F) \succ \phi_\Bc (\Phi E) \text{\quad by Lemma \ref{lem:phiGammaTeq} and \eqref{eq:AG48-20-1}}\\
&\Leftrightarrow \phi_\Bc' (\Phi F) \succ \phi_\Bc' (\Phi E) \\
&\Leftrightarrow \phi_\Bc' (\wh{F}) \succ \phi_\Bc' (\wh{E}) \text{\quad since $i=j$}\\
&\Leftrightarrow S_\Bc (\wh{F}) \succ 0,
\end{align*}
and all these equivalences still hold if we replace all instances of `$\succ$' by `$=$' or `$\prec$'.  Together with
\[
  S_\Bc (\wh{F}) = S_\Bc (\Phi F [j])=S_\Bc (\Phi F [i]) = ((-1)^iS_\Bc)(\Phi F),
\]
we obtain $\sgn S_\Ac (F)=\sgn ((-1)^iS_\Bc) (\Phi F)$ in this case.
\end{itemize}
Therefore, the equivalences $\Phi, \Psi$, the hearts $\Ac, \Bc$, and the weight functions $S_\Ac, (-1)^iS_\Bc$ are in Configuration II.  Moreover, $S_\Ac$ satisfies  refinement-$i$  with respect to $(W_{0,\Phi,\Ac,\Bc}, W_{1,\Phi,\Ac,\Bc})$ by part (iii).  Then by Theorem \ref{thm:main3}, $\wh{E}$ is a semistable object in $\Bc$ with respect to the weight function $(-1)^{2i}S_\Bc=S_\Bc$, meaning $\wh{E}$ is a $Z_\Bc$-semistable object in $\Bc$.  The `stable' case of part (iv) follows from the correspondence case of Theorem \ref{thm:main3}.

(v) Suppose $\Phi \cdot (Z_\Ac,\Pc_\Ac) = (Z'',\Pc'')$ and $(Z_\Bc,\Pc_\Bc)\cdot \Patilde (T) = (Z',\Pc')$.  Then $Z'' =  Z_\Ac \circ (\Phi^K)^{-1} = T^{-1}\cdot Z_\Bc=Z'$ where the second equality follows from \eqref{eq:AG45-108-17}. Also, for any polynomial phase function $\phi$, we have $\Pc'' (\phi) = \Phi (\Pc_\Ac (\phi)) = \Pc_\Bc (\Gamma_T(\phi))=\Pc'(\phi)$, where the second equality follows from  \eqref{eq:AG48-20-1} together with part (iv).
\end{proof}

\begin{prop}\label{prop:AHNgivesBHN}
Assume Configuration III.   If $Z_\Ac$ has the Harder-Narasimhan property on $\Ac$, then $Z_\Bc$  has the Harder-Narasimhan property on $\Bc$.
\end{prop}

\begin{proof}
Suppose $Z_\Ac$ has the HN property on $\Ac$.  Fix any $E \in \Bc$, and we will construct the $Z_\Bc$-HN filtration for $E$.

To begin with, we have $\Psi \Bc \subset D^{[0,1]}_{\Ac}$ by Lemma \ref{para:AcBcPhitilt}, giving us  an exact triangle
\begin{equation}\label{eq:AG45-132-2}
(\Psi^0_\Ac E)[1] \to \Psi E[1] \to \Psi^1_\Ac E  \to (\Psi^0_\Ac E)[2].
\end{equation}
Suppose
\begin{equation}\label{eq:AG45-133-1}
 0=M_0 \subsetneq M_1 \subsetneq M_2 \subsetneq \cdots \subsetneq M_k = \Psi^0_\Ac E
\end{equation}
is the $Z_\Ac$-HN filtration of $\Psi^0_\Ac E$ in $\Ac$, and
\begin{equation*}
0=N_0 \subsetneq  N_1 \subsetneq N_2 \subsetneq \cdots \subsetneq N_l = \Psi^1_\Ac E
\end{equation*}
the $Z_\Ac$-HN filtration of $\Psi^1_\Ac E$ in $\Ac$.  In particular, we have
\[
\phi_\Ac (M_1) \succ \phi_\Ac (M_2/M_1) \succ \cdots \succ \phi_\Ac (M_k/M_{k-1})
\]
and
\[
  \phi_\Ac (N_1) \succ \phi_\Ac (N_2/N_1) \succ \cdots \succ \phi_\Ac (N_l/N_{l-1}).
\]
The filtration \eqref{eq:AG45-133-1} gives the sequence of exact triangles
\begin{equation}\label{eq:AG45-133-2}
\scalebox{0.95}{
\xymatrix@=0.9em{
0=M_0[1] \ar[rr] & & M_1[1]  \ar[r] \ar[dl] & \cdots \ar[r] & M_{k-1}[1] \ar[rr] & & M_k[1]= (\Psi^0_\Ac  E)[1]  \ar[dl] \\
 & M_1[1] \ar[ul]^{[1]} & & & & (M_k/M_{k-1})[1] \ar[ul]^{[1]} &
}.
}
\end{equation}
Using the octahedral axiom, we can construct lifts
 $\wt{N_1}, \cdots, \wt{N_l}$ in $\Dc$ of $N_1, \cdots, N_l$ with respect to the morphism $\Psi E [1] \to \Psi^1_\Ac E$, respectively,   together with a sequence of exact triangles 
\begin{equation}\label{eq:AG45-133-3}
\scalebox{0.95}{
\xymatrix@=0.9em{
(\Psi^0_\Ac E) [1] \ar[rr] & & \wt{N_1}  \ar[r] \ar[dl] & \cdots \ar[r] & \wt{N_{l-1}} \ar[rr] & & \wt{N_l}= \Psi E [1] \ar[dl] \\
 & N_1 \ar[ul]^{[1]} & &  & & N_l/N_{l-1} \ar[ul]^{[1]} &
}.
}
\end{equation}
We will now show that $\Phi$ takes the concatenation of \eqref{eq:AG45-133-2} and \eqref{eq:AG45-133-3} to the $Z_\Bc$-HN filtration of $E$.

Since $\Psi^0_\Ac E$ is $\Phi_\Bc$-WIT$_1$, so is $M_1$, and so    $\Gamma_T^{-1}(b)  \succeq \phi_\Ac (M_1)$ by Theorem \ref{prop:paper30prop11-11ext}(i).  On the other hand, for each $1 \leq i \leq k$, the factor $M_i/M_{i-1}$ is $Z_\Ac$-semistable with $\Gamma_T^{-1}(b) \succeq \phi_\Ac (M_1) \succ \phi_{\Ac}(M_i/M_{i-1})$, and so $M_i/M_{i-1}$ is $\Phi_\Bc$-WIT$_1$ by parts (i) and (ii) of Theorem \ref{prop:paper30prop11-11ext}.  A similar argument shows that for each $1 \leq i \leq l$, the factor $N_i/N_{i-1}$ is $Z_\Ac$-semistable and $\Phi_\Bc$-WIT$_0$ with $\phi_\Ac (N_i/N_{i-1}) \succ \Gamma_T^{-1}(b)$.  Hence $\Phi$ takes each $(M_i/M_{i-1})[1]$ and each $N_j/N_{j-1}$ into $\Bc$, and all of $\Phi ( (M_i/M_{i-1})[1]), \Phi (N_j/N_{j-1})$ are $Z_\Bc$-semistable by Theorem  \ref{prop:paper30prop11-11ext}(iv).

From \eqref{eq:AG48-20-1}, we now have
\[
 \phi_\Bc (\Phi (M_1)[1]) \succ \phi_\Bc (\Phi (M_2/M_1)[1]) \succ \cdots \succ \phi_\Bc (\Phi (M_k/M_{k-1})[1])
\]
and
\[
  \phi_\Bc (\Phi N_1) \succ \phi_\Bc (\Phi (N_2/N_1)) \succ \cdots \succ \phi_\Bc (\Phi (N_l/N_{l-1})).
\]
Since  $\Phi (M_k/M_{k-1}) [1]$ is $\Psi_\Ac$-WIT$_0$ and $\Phi N_1$ is $\Psi_\Ac$-WIT$_1$, we have
\[
  \phi_\Bc (\Phi (M_k/M_{k-1})[1]) \succ \phi_\Bc (\Phi N_1)
\]
by  Theorem  \ref{prop:paper30prop11-11ext}(i).
Overall, the concatenation of \eqref{eq:AG45-133-2} with \eqref{eq:AG45-133-3} gives a filtration of $\Psi E [1]$, and this filtration is taken by $\Phi$ to a filtration of $E$ in $\Bc$ where each factor is $Z_\Bc$-semistable and the factors have strictly decreasing $\phi_\Bc$.  That is, we have constructed the $Z_\Bc$-HN filtration of $E$.
\end{proof}

\begin{rem}
Even though it is possible to write down a  generalisation of Proposition \ref{prop:AHNgivesBHN} to weak polynomial stability conditions, as is done in \cite[Section 19]{bayer2019stability} when the central charge takes values in $\mathbb{C}$ along with other assumptions, it is not clear whether there is a canonical way to do so.  Since we do not require such a generalisation in this article, we refrain from stating it here.
\end{rem}

\section{Preliminaries on  elliptic fibrations}\label{sec:prelim-ellfib}

In this section, we fix our  notation and terminology on elliptic fibrations.

\paragraph[$\mu_{\omega,B}$-stability and $\Bob$]  Suppose $X$ is  \label{para:slopestab} a smooth projective variety of dimension $n$,  and   $\omega, B$ are  $\mathbb{R}$-divisor classes on $X$ where $\omega$ is ample.  The usual slope function for coherent sheaves $E$ on $X$ is defined as
\[
\mu_{\omega, B}(E)= \begin{cases}
\frac{\omega^{n-1}\ch_1^B(E)}{\ch_0^B(E)}  &\text{if $\ch_0^B(E)\neq 0$}   \\
\infty &\text{ if $\ch_0^B(E)=0$}
\end{cases}.
\]
Note that $\mu_{\omega,B}(E)=\mu_{\omega,0}(E)-\omega^{n-1}B$ when $\ch_0(E)\neq 0$, while $\mu_{\omega,0}$ itself has the HN property even when $\omega$ is not over $\mathbb{Q}$ by \cite[Corollary 2.27]{greb2016movable}.  Hence $\mu_{\omega,B}$ itself has the HN property for any $\mathbb{R}$-divisors $\omega, B$ where $\omega$ is ample.

We will write $K(X)$ for $K(D^b(X))$.  In terms of the notation of \ref{para:mugeneraldef}, $\mu_{\omega, B}$  is precisely the slope function $\mu_Z$ for the weak polynomial stability function $Z : K(X) \to \mathbb{C}$ on $\Coh (X)$ with respect to $(0,1]$ given by
\[
  Z (E) = -\omega^{n-1}\ch_1^B(E) + i\ch_0^B (E).
\]
We define the following full subcategories of $\Coh (X)$ via extension closures
\begin{align*}
  \Tob &= \langle F \in \Coh (X) : F \text{ is $\muob$-semistable with $\muob (F)>0$} \rangle \\
  \Fob &= \langle F \in \Coh (X) : F \text{ is $\muob$-semistable with $\muob (F)\leq 0$} \rangle.
\end{align*}
 The HN property of $\mu_{\omega,B}$ implies that  $(\Tob, \Fob)$ is a torsion pair in $\Coh (X)$, and so   the extension closure in $D^b(X)$
\[
  \Bob = \langle \Fob [1], \Tob \rangle
\]
is the heart of a bounded t-structure on $D^b(X)$.

\paragraph[The central charge $Z_{\omega,B}$] Given any smooth \label{para:defBheart} projective variety $X$ and   $\mathbb{R}$-divisor classes $\omega, B$ on $X$ where $\omega$ is ample, we  define the group homomorphism $Z_{\omega,B} : K(X) \to \mathbb{C}$ by
\[
Z_{\omega,B}(E) = -\int_X e^{-(B+i\omega)}\ch(E) =  -\int_X e^{-i\omega}\ch^B(E).
\]
On a smooth projective surface $X$, the pair $(Z_{\omega, B},\Bc_{\omega, B})$ is a Bridgeland stability condition with respect to $(0,1]$.  (We do not need to assume $\omega, B$ are over $\mathbb{Q}$ here -  see \cite[Section 2]{ABL} or \cite[Theorem 6.10]{MSlec}.)

When $B=0$, we will drop the subscript $B$, and similarly for the notations in \ref{para:slopestab}.

\paragraph[The heart $\Coh^p$]  On a \label{para:heartCohp} smooth projective variety $X$, we will write $\Coh^p$ or $\Coh^p(X)$ to denote the heart of bounded t-structure on $D^b(X)$ associated to the  perversity function $p(d)=-\lfloor \tfrac{d}{2}\rfloor$, the general definition of which can be found in \cite[Section 3.1]{BayerPBSC}.  Specifically, we have
\begin{align*}
\Coh^p (X) &= \Coh (X) \text{\quad when $\dimension X=1$},\\
\Coh^p (X) &= \langle \Coh^{= 2}(X)[1], \Coh^{\leq 1}(X)\rangle  \text{\quad when $\dimension X=2$},\\
\Coh^p (X) &= \langle \Coh^{\geq 2}(X)[1], \Coh^{\leq 1}(X)\rangle  \text{\quad when $\dimension X=3$}.
\end{align*}

\paragraph[Elliptic fibrations $p : X \to B$] Unless otherwise stated, in the remainder of this article, we will  write $p : X \to B$ to denote a  Weierstra{\ss} elliptic fibration  in the sense of \cite[2.1]{Lo15} where $X, B$ are both smooth projective varieties.  In particular, by $p$ being an elliptic fibration, we mean that   $p$ is a flat morphism whose fibers are Gorenstein curves of arithmetic genus 1.  The Weierstra{\ss} condition means that all the fibers of $p$ are geometrically integral, and $p$ has a section $\sigma : B \to X$ such that its image $\Theta = \sigma (B)$ does not intersect any singular point of any singular fiber.  The smoothness of $X$ implies that the generic fiber of $p$ is a smooth elliptic curve, while the Weierstra{\ss} assumption means that  the singular fibers are at worst nodal or cuspidal.  When $X$ is a surface (resp.\ threefold), we will refer to $p : X \to B$ or simply $X$ as a Weierstra{\ss} elliptic surface (resp.\ Weierstra{\ss} elliptic threefold).

We will write $f$ to denote the class of a fiber of $p$.  Since $\Theta$ is a section of $p$, it is a $p$-ample divisor on $X$ by the Nakai-Moishezon criterion \cite[Theorem 1.42]{KM}.  As a result, given any ample divisor $H_B$ on $B$,  divisors on $X$ of the form $\Theta + vp^\ast H_B$ are ample for $v \gg 0$ \cite[Proposition 1.45]{KM}.  We also define the slope function $\mu_f$ on $\Coh (X)$ by
\[
\mu_f (E)=\begin{cases}
\frac{f\ch_1 (E)}{\ch_0(E)} &\text{\quad if $\ch_0(E)\neq 0$} \\
\infty &\text{\quad if $\ch_0(E)=0$}
\end{cases},
\]
which has the HN property.

\subparagraph[Another word on notation] We adopt a good deal of notation surrounding elliptic fibrations from the reference \cite[6.2.6]{FMNT}.  In continuing with their notation, we use the letter $B$ to denote the base of the elliptic fibration $p : X \to B$.  On the other hand, we also use $B$ to denote  the  $\mathbb{R}$-divisor on $X$ (sometimes called the `$B$-field')  used in `twisting' Chern classes in defining slope functions and  central charges.  Although there is a conflict of notation here, we believe that no confusion should arise as the meaning of $B$ will always be clear from the context.

\subparagraph[Self-intersection of $\Theta$]  We will \label{para:selfinters} write the self-intersection of $\Theta$ on $X$ as $\Theta^2=-e$.  Note that $e=\mathrm{deg}\, \mathbb{L}$ where $\mathbb{L}=(R^1p_\ast \OO_X)^\ast$ (e.g.\ see  \cite[2.3]{LLM}), while $\deg\, \mathbb{L} \geq 0$  by \cite[Lemma II.5.6, Definition II.4.1, (II.3.6)]{MirLec}.  Hence $\Theta^2 = -e \leq 0$ overall.

\paragraph[Fourier-Mukai transforms $\Phi,\whPhi$] For a \label{para:FMTs} Weierstra{\ss} elliptic fibration $p : X \to B$, we will write $\Phi$ to  denote the relative Fourier-Mukai transform $D^b(X) \to D^b(X)$ as defined in \cite[6.2.3]{FMNT}, the kernel $\mathcal{P}$ of which is a relative Poincar\'{e} sheaf. In particular, $\mathcal{P}$ is a universal sheaf for a moduli problem parametrising rank-one torsion-free sheaves on the fibers of $p$.  That is, for any closed point $x \in X$, if $\OO_x$ denotes the structure sheaf of $\{x\}$ then $\Phi \OO_x$ is a rank-one torsion-sheaf supported on the fiber of $p$ containing $x$.   There is another relative Fourier-Mukai transform $\whPhi : D^b(X) \to D^b(X)$ satisfying
\[
  \whPhi \Phi \cong \mathrm{id}_{D^b(X)}[-1] \cong \Phi \whPhi,
\]
and the kernel of $\whPhi$ is also a universal sheaf for a moduli problem of rank-one torsion-free sheaves on fibers of $p$ (see \cite[Theorem 6.18, Lemma 6.22]{FMNT} and also \cite[8.4]{BMef}).

We say an object $E \in D^b(X)$ is $\Phi$-WIT$_i$ if $\Phi E$ is isomorphic to an object in  $\Coh (X)[-i]$, in which case we write $\wh{E}$ for any coherent sheaf satisfying $\Phi E \cong \wh{E}[-i]$; we  similarly define the notion of $\whPhi$-WIT$_i$.  These definitions are in line  with those in \ref{para:def-WIT}.

\section{Elliptic curves}\label{sec:ellcur}

Let $X$ be a \label{eg:FMTonellipcurve} smooth elliptic curve, so that the base $B$ of our elliptic fibration $p : X \to B$ is a single point.  Set $\omega$ to be  a single closed point on $X$ and the $B$-field in \ref{para:defBheart} to be zero, so that the central charge $Z(E) := Z_{\omega,B} (E)$ takes the form
\[
Z(E)= -\int_X \ch_1(E) + i \ch_0(E) \int_X \omega = -(\degree E) + i(\rank E).
\]
The pair $(Z,\Coh (X))$ is a polynomial stability function $Z : K(X) \to \mathbb{C}$ with respect to $(0,1]$, and the corresponding slope function (see \ref{eg:stabslopefunc}), being the usual slope function for Mumford stability for sheaves, has the Harder-Narasimhan property.  Therefore, $(Z, \Coh (X))$ is a Bridgeland stability condition, and an object $E \in D^b(X)$ is a  $Z$-semistable (resp.\ $Z$-stable) object in $\Coh (X)$ if and only if it is a slope semistable  (resp.\ slope stable) coherent sheaf on $X$ (see \ref{para:mugeneraldef1}).

Recall the following classification result for slope semistable sheaves on a smooth elliptic curve, which can be found in the works of Bridgeland \cite[3.2]{FMTes}, Hein-Ploog \cite{hein2005fourier}, Polishchuk \cite[Lemma 14.6]{polishchuk2003abelian}, and also in \cite[Corollary 3.29]{FMNT}:

\begin{thm}\label{thm:main0}
Let $X$ be a smooth elliptic curve, and $\Phi : D^b(X) \to D^b(X)$ the Fourier-Mukai transform with Poincar\'{e} line bundle as the kernel.  Then for any $E \in D^b(X)$,
\[
 \text{$E$ is a slope semistable sheaf if and only if $\Phi E$ is a slope semistable sheaf (up to a shift)}.
\]
The same result holds when `semistable' is replaced with `stable' on both sides.
\end{thm}

Theorem \ref{thm:main0} can be recovered using Configuration III and Theorem \ref{prop:paper30prop11-11ext}(iv).  To begin with, for any $E \in D^b(X)$ we have
\begin{equation}\label{eq:ellipc-rkdeg}
 \begin{pmatrix} \rank \Phi E \\ \degree \Phi E \end{pmatrix} = \begin{pmatrix} 0 & 1 \\ -1 & 0 \end{pmatrix} \begin{pmatrix} \rank E \\ \degree E \end{pmatrix}
\end{equation}
(e.g.\ see  \cite[Proposition 3.25]{FMNT}, which uses a functor that is dual to $\Phi$ but has the same    cohomological Fourier-Mukai transform), which can be rewritten  as
\[
  Z(\Phi E) = (-i)Z(E).
\]
Since $p$ has relative dimension 1 and the kernel of $\Phi$ is a sheaf sitting at degree 0, we have $\Phi \Coh (X) \subset D^{[0,1]}_{\Coh (X)}$ \cite[p.186]{FMNT}.

\begin{proof}[Proof of Theorem \ref{thm:main0}]
We are  in Configuration III if we take  $\Dc = \Uc = D^b (X)$, $\Psi = \whPhi$, $\Ac =  \Bc =\Coh (X)$, $Z_\Ac = Z_\Bc = Z$, $a=b=0$, and $T, \Gamma_T$ to be as in Example \ref{eg:mulbynegi} where $T$ corresponds to multiplication by $(-i)$.  The theorem then follows from Theorem \ref{prop:paper30prop11-11ext}(iv).
\end{proof}

Let $\Pc$ denote the slicing of the Bridgeland stability condition $(Z,\Coh (X))$  with respect to $(0,1]$  and write $\sigma = (Z,\Pc)$.  Them  the proof of Theorem \ref{thm:main0} together with Theorem \ref{prop:paper30prop11-11ext}(v) give
\begin{equation}\label{eq:ellcurv2}
\Phi \cdot \sigma = \sigma \cdot (T,g)
\end{equation}
where $T$ corresponds to multiplication by $(-i)$ and $g(x)=x-\tfrac{1}{2}$ for all $x \in \RR$.

Theorem \ref{thm:main0} can be generalised to higher dimensional elliptic fibrations using Theorem  \ref{prop:paper30prop11-11ext}.  This requires choosing appropriate input for Configuration III.  A  choice of such input on elliptic surfaces will be constructed  in Sections \ref{sec:ellsurfGLact} and \ref{sec:ellsurfpolystab}.

\section{Elliptic surfaces: $\GLtr$-action}\label{sec:ellsurfGLact}

On a smooth projective surface $X$, we have
\[
Z_{\omega,B}(E)= -\int_X e^{-i\omega}\ch^B(E) = -\ch_2^B(E) + \tfrac{\omega^2}{2}\ch_0^B(E) + i\omega \ch_1^B(E).
\]
In this section, we will take  $X$ to be a Weierstra{\ss} elliptic surface and  study the following question:

\begin{que}\label{q:1}
Given $\mathbb{R}$-divisors $\olw, \olB$ on $X$, can we find $\mathbb{R}$-divisors $\omega, B$ and $g \in \GLtr$ such that
\[
  Z_{\omega, B}(\Phi E) = T Z_{\olw, \olB}(E) \text{\quad for all $E\in D^b(X)$?}
\]
\end{que}

That is, we will find solutions to   equation \eqref{eq:AG45-108-17} in Configuration III, where the central charges are those frequently used in the study of Bridgeland stability conditions, and where $T$ is a constant function.

\paragraph[Cohomological Fourier-Mukai transform]  Suppose $p : X \to B$ \label{para:surfcohomFMT} is a Weierstra{\ss} elliptic surface.  For an object $E \in D^b(X)$, let us write
\begin{equation}\label{eq:ellipsurfchE}
\ch_0(E)=n,\text{\quad}  f\ch_1(E)=d, \text{\quad} \Theta \ch_1(E)=c,\text{\quad}  \ch_2(E)=s.
\end{equation}
We have the following formulas from \cite[(6.21)]{FMNT}
\begin{align*}
\ch_0(\Phi E) &= d, \\
\ch_1 (\Phi E) &= -\ch_1(E) + def + (d-n)\Theta + (c-\tfrac{1}{2} ed+s)f, \\
\ch_2 (\Phi E) &= -c-de+\tfrac{1}{2}ne,
\end{align*}
and hence
\[
  f\ch_1(\Phi E) = -n, \text{\quad} \Theta \ch_1(\Phi E) = (s-\tfrac{e}{2}d)+ne
\]
where we recall $\Theta^2=-e \leq 0$ from \ref{para:selfinters}.  An analogous formula for the cohomological Fourier-Mukai transform associated to $\whPhi$ can be found in \cite[(6.22)]{FMNT}.

\paragraph  For \label{para:Q2} $a, b \in \mathbb{R}_{>0}$  and any $\mathbb{R}$-divisor $B$ on $X$, we will write
\[
  Z_{a,b,B}(E) = -\ch_2^B(E) + a \ch_0^B(E) + i (\Theta \ch_1^B(E) + bf\ch_1^B(E) )
\]
 and
\[
  Z_{a,b,B}'(E) = \begin{pmatrix} 1 & -fB \\ 0 & 1 \end{pmatrix} Z_{a,b,B}(E)
\]
for any $E \in D^b(X)$.    Note that when $\omega$ is  an ample divisor of the form $\omega = x\Theta + yf$ where $x,y\in \mathbb{R}_{>0}$, we can write
\begin{equation}\label{eq:Zchangeofcoord}
  Z_{\omega, B}(E) =  \begin{pmatrix} 1 & 0 \\ 0 & x \end{pmatrix} Z_{\tfrac{\omega^2}{2}, \tfrac{y}{x},B}(E).
\end{equation}
Hence every central charge of the form $Z_{\omega, B}$ is equal to $gZ_{a,b,B}'$ for some $g\in \GLtr$ and $a, b \in \mathbb{R}_{>0}$  depending only on $\omega$.  Therefore, if we only consider divisors $\olw, \omega$ lying in $\mathbb{R}_{>0}\Theta + \mathbb{R}_{>0}f$, then    Question \ref{q:1} is equivalent to the following:

\begin{que}\label{q:1b}
Given $\gamma, \delta \in \mathbb{R}_{>0}$ and an $\mathbb{R}$-divisor $\olB$, can we find $\epsilon, \zeta\in \mathbb{R}_{>0}$, an $\mathbb{R}$-divisor $B$ and $T\in \GLtr$ such that
\begin{equation}\label{eq:Q2}
  Z_{\epsilon, \zeta, B}'(\Phi E) = TZ_{\gamma, \delta, \olB}'(E) \text{\quad for all $E\in D^b(X)$?}
\end{equation}\
\end{que}

\begin{remsub}
When $b>0$ is chosen so that $\Theta + bf$ is an ample divisor on $X$ and $a \in \mathbb{R}_{>0}$, the same argument as in the proof of \cite[Corollary 2.1]{ABL} shows that $(Z_{a,b,B},\Bc_{\Theta +bf,B})$ is a Bridgeland stability condition with respect to $(0,1]$.
\end{remsub}

\paragraph Let \label{para:Q2b}  us consider the equation \eqref{eq:Q2} under the additional constraint that    $\olB, B$  lie in $\mathbb{R}\Theta + \mathbb{R}f$.  Suppose $\olB = k\Theta + lf$ and $B = p\Theta + qf$ where $k,l,p, q \in \mathbb{R}$.  Using the notation for  $\ch (E)$ from \ref{para:surfcohomFMT}, we have
\[
  Z_{\gamma, \delta, \olB}'(E) = -s + (l-k\delta)d + \left( \gamma + (\delta - \tfrac{e}{2})k^2 \right)n + i\left( c+ \delta d - (l-ek+\delta k)n\right)
\]
 and
\begin{multline*}
  Z_{\epsilon, \zeta, B}'(\Phi E) = c + (e-\tfrac{e}{2}p^2 + \epsilon + p^2 \zeta ) d + (-\tfrac{e}{2} -q + p\zeta ) n + i \left( s + (-\tfrac{e}{2} - (q-ep) - p\zeta) d + (e-\zeta)n\right).
\end{multline*}
Using the last two formulas and the cohomological Fourier-Mukai transform for $\whPhi$ in \cite[(6.22)]{FMNT}, we can compute
\[
Z_{\gamma, \delta, \olB}'(\OO_x) = -1,\,\, Z_{\epsilon, \zeta, B}'(\Phi \OO_x)=i, \text{\quad and \quad} Z_{\gamma, \delta, \olB}'(\whPhi \OO_x) = i,\,\, Z_{\epsilon, \zeta, B}'(\Phi \whPhi \OO_x)=1.
\]
Putting $E=\OO_x$ in \eqref{eq:Q2} now gives $i = T\cdot (-1)$, while putting $E=\whPhi \OO_x$ gives $1=T\cdot i$.  It follows that the only   element $T \in \GLtr$ that can satisfy \eqref{eq:Q2} for all $E \in D^b(X)$ is $T=\begin{pmatrix} 0 & 1 \\ -1 & 0 \end{pmatrix}$, i.e.\ $T$ is simply multiplication by $-i$.  In this case,   \eqref{eq:Q2} holds for all $E \in D^b(X)$  whenever the following relations are satisfied:
\begin{align}
  l - k\delta &= \tfrac{e}{2} + (q-ep) + p\zeta \notag\\
  \gamma + (\delta - \tfrac{e}{2})k^2  &= \zeta - e \notag\\
  \delta &= e - \tfrac{e}{2}p^2 + \epsilon + p^2 \zeta \notag\\
  l - ek + \delta k &= \tfrac{e}{2} + q - p\zeta \label{eq:c4}.
\end{align}

\subparagraph In the \label{para:Q2bs} special case $k=p, l=q, \gamma=\epsilon$ and $\delta = \zeta$, in which case $\ol{B}=B$ and $Z'_{\gamma, \delta, \ol{B}}=Z'_{\epsilon, \zeta,B}$, the four relations in \eqref{eq:c4} together simplify to  just three relations
\begin{equation*}
 e=0, \text{\quad} p\zeta =0, \text{\quad} \zeta  =\epsilon.
\end{equation*}
In particular, with $g=-i$ and for $e=0$, if $k=p=0$ then any $l=q$ and any $\gamma=\delta=\epsilon=\zeta>0$ give a solution to \eqref{eq:Q2}.

\section{A polynomial stability and twisted Gieseker stability on surfaces}\label{sec:polytwststab}

In this section, we  take $X$ to be an arbitrary smooth projective surface.  We describe the large volume limit, which can be defined as a polynomial stability condition denoted as $Z_l$-stability.  Then we prove that 1-dimensional $Z_l$-semistable objects coincide with 1-dimensional twisted Gieseker semistable sheaves.

\paragraph[$Z_l$-stability] On a surface $X$, \label{para:surfZlowerl}  we have $\Coh^p (X) = \langle \Coh^{= 2}(X)[1], \Coh^{\leq 1}(X)\rangle$ from \ref{para:heartCohp}.   Then for any fixed  $\mathbb{R}$-divisors $\wt{\omega}, \ol{B}$ on $X$ where $\wt{\omega}$ is ample, and any $\beta \in \mathbb{R}_{>0}$, we set $\olw = \beta \wt{\omega}$ and define
\[
Z_l (E) = Z_{\ol{\omega},\olB}(E) = Z_{\beta \wt{\omega},\ol{B}}(E)
  \]
for $E \in D^b(X)$.  By regarding $\beta$ as the parameter,  the pair $(Z_l,\Coh^p)$ is a polynomial stability condition on $D^b(X)$ with respect to $(\tfrac{1}{4}, \tfrac{5}{4}]$ \cite[Proposition 4.1(b)]{BayerPBSC}.  The wall and chamber  structure of $(Z_l,\Coh^p)$ as a polynomial stability condition was studied in \cite{LiQin1}, while the mini-wall and mini-chamber structure of the associated Bridgeland stability conditions $\{(Z_{\beta\wt{\omega},\ol{B}}, \Bc_{\beta\wt{\omega},\ol{B}})\}_\beta$ was  studied in \cite{LQ}.  For a fixed Chern character of objects in $D^b(X)$, there is an effective bound $\beta_0$ such that, for $\beta \geq \beta_0$, the moduli of $Z_l$-semistable objects (i.e.\ the moduli of polynomial semistable objects) with that Chern character can be identified with the moduli of $Z_{\beta\wt{\omega},\ol{B}}$-semistable objects (i.e.\ the moduli of Bridgeland semistable objects) with the same Chern character  \cite[Theorem 4.4]{LQ}.

\paragraph[Twisted Gieseker stability]  Let \label{para:twGiedef} $X$ be a smooth projective surface and   $L$  a line bundle on $X$. The $L$-twisted Euler characteristic of a coherent sheaf $E$ on $X$ is defined as
\begin{equation*}
\chi_L (E):=\chi (E\otimes L) =\int_X \ch (E \otimes L)\mathrm{td} (T_X)
 \end{equation*}
 where the second equality follows from  Hirzebruch-Riemann-Roch.  Given a fixed ample divisor $\wt{\omega}$ on $X$, we say a coherent sheaf $E$ in $\Coh^{\leq 1}(X)$  is  \emph{$L$-twisted $\wt{\omega}$-Gieseker semistable}, or simply \emph{twisted Gieseker semistable}, if for every short exact sequence of sheaves $0 \to M \to E \to N \to 0$ on $X$ where $M, N\neq 0$ we have
\[
\frac{\chi_L(M)}{\wt{\omega}\ch_1(M)} \leq \frac{\chi_L(N)}{\wt{\omega}\ch_1(N)}
\]
(e.g.\ see \cite[Definition 6.1]{LLM}).  We say $E$ is twisted Gieseker stable if we always have a strict inequality above.

\paragraph Suppose  \label{para:obs2} $L$ is  a line bundle on  $X$ such that $\ch_1(L)K_X=0$.  Then for any  $E \in \Coh^{\leq 1}(X)$ we have $\chi_L(E)=\ch_2^{\ol{B}}(E)$ if we put $\ol{B}=\ch_1(L^\ast)+\tfrac{1}{2}K_X$; furthermore, for  any ample $\mathbb{R}$-divisor $\wt{\omega}$ we have
\[
 Z_{\wt{\omega},\olB}(E)=-\ch_2^{\olB}(E)+i\wt{\omega} \ch_1^{\olB}(E)=-\chi_L(E)+i\wt{\omega} \ch_1(E).
\]

\begin{cor}\label{lem:AG48-25-1}
Let $X$ be a smooth projective  surface and $\wt{\omega}$ a fixed ample $\mathbb{R}$-divisor on $X$.  Suppose $L$ is a line bundle on $X$ such that $\ch_1(L)K_X=0$, and  $\olB = \ch_1(L^\ast)+\tfrac{1}{2}K_X$.  Then for any $E\in D^b(X)$,
\begin{align*}
  \text{$E$ is an $L$-twisted $\wt{\omega}$-Gieseker semistable sheaf in $\Coh^{\leq 1}(X)$ } &\Leftrightarrow \text{ $E$ is $Z_l$-semistable with $\ch_0(E)=0$ }
\end{align*}
where $Z_l$-stability is defined as in \ref{para:surfZlowerl}. The same statements holds if we replace `semistable' with `stable.'
\end{cor}

\begin{proof}
For any $0 \neq E \in \Coh^{\leq 1}(X)$, we have
\[
Z_l(E)=Z_{\beta \wt{\omega},\ol{B}}(E)=-\ch_L(E)+i\beta \wt{\omega}\ch_1(E)
\]
and $Z_l(E) \in \mathbb{R}_{>0}e^{i\pi (0,1]}$.  Considering $\Coh^{\leq 1}(X)$ as an abelian subcategory of $D^b(X)$, we have  the slope function
\[
\mu_{Z_l}(E)=\frac{\ch_L(E)}{\beta \wt{\omega}\ch_1(E)}
\]
for $0 \neq E \in \Coh^{\leq 1}(X)$ as defined in \ref{para:mugeneraldef}.

Now, given any $E \in D^b(X)$, if $E$ is an $L$-twisted $\wt{\omega}$-Gieseker semistable sheaf in  $\Coh^{\leq 1}(X)$, then $E \in \Coh^p(X)$, and any $\Coh^p(X)$-short exact sequence of the form $0 \to M \to E \to N \to 0$ (assuming $M, N \neq 0$) is a $\Coh^{\leq 1}(X)$-short exact sequence.  Since twisted Gieseker semistability clearly coincides with $\mu_{Z_l}$-semistability on $\Coh^{\leq 1}(X)$, we have $\mu_{Z_l} (M) \preceq \mu_{Z_l}(N)$ and hence $\phi_{Z_l}(M) \preceq \phi_{Z_l}(N)$  (see \ref{para:mugeneraldef1}), showing that $E$ is $Z_l$-semistable in $\Coh^p(X)$.

Conversely, if $E$ is a $Z_l$-semistable object in $\Coh^p(X)$ with $\ch_0(E)$, then $H^{-1}(E)$ must vanish and so $E \in\Coh^{\leq 1}(X)$.  Since $\Coh^{\leq 1}(X)$ is a subcategory of  $\Coh^p(X)$, every $\Coh^{\leq 1}(X)$-short exact sequence of the form $0 \to M \to E \to N \to 0$ (assuming $M, N \neq 0$) is also a $\Coh^p(X)$-short exact sequence.  Then the $Z_l$-semistability of $E$ in $\Coh^p(X)$ implies $\mu_{Z_l}(M) \preceq \mu_{Z_l}(N)$ by \ref{para:mugeneraldef1}, which in turn shows $E$ is twisted Gieseker semistable.  Hence the 'semistable' case is proved.

The above arguments can be easily modified to show the 'stable' case.
\end{proof}

\section{Elliptic surfaces: a second polynomial stability}\label{sec:ellsurfpolystab}

In this section, we continue to take $X$ to be a Weierstra{\ss} elliptic surface as we did in Section \ref{sec:ellsurfGLact} unless otherwise stated.  We describe another polynomial stability which we  denote as $Z^l$-stability.  As opposed to $Z_l$-stability, the definition of which involves deforming an ample class towards infinity along a ray in the ample cone of $X$, the definition of $Z^l$-stability involves deforming an ample class towards the fiber class  along a hyperbola in the ample cone.  We prove that $Z_l$-semistable objects correspond to $Z^l$-semistable objects under the autoequivalence $\Phi$, for all Chern classes.

\paragraph[Notation] Fix  \label{para:ellipsurfnotation} an $m \in  \mathbb{R}_{>0}$ such that $\Theta + kf$ is ample for all $k\geq m$.  For any $\alpha, \beta \in \mathbb{R}_{>0}$, we will write $\wt{\omega}=\tfrac{1}{\alpha}(\Theta+mf)+f$ and
\[
\ol{\omega}=\beta \wt{\omega}=\tfrac{\beta}{\alpha}(\Theta + mf) + \beta f = \tfrac{\beta}{\alpha}( \Theta+(m+\alpha)f).
\]
Then $\wt{\omega}$ and $\ol{\omega}$ are both ample by our choice of $m$.  In addition, for any $u, v \in \mathbb{R}_{>0}$ we write
 \[
 \omega = u(\Theta + mf) +vf= u(\Theta + (m+\tfrac{v}{u})f),
 \]
 which is also ample by our choice of $m$.  Since   $\Theta^2 = -e$ and  $\Theta +mf$ is ample by assumption, we have $\Theta(\Theta + mf)>0$, i.e.\  $m-e>0$.

\paragraph[$\olB, B$ in $\mathbb{R}f$] In the rest \label{para:L3p-relation} of this section, we will  focus on $B$-fields of the form  $\olB = lf$ and $B=qf$ where $l,q \in \mathbb{R}$.  We do not consider the case where $f\olB, fB$ are nonzero because, in extending  the  construction of $Z^l$-stability in \cite{LLM} to such nonzero $B$-fields,  essential ingredients such as \cite[3.2(iii), (v)]{LLM} will no longer hold (one might have to  replace the autoequivalence $\Phi$ by a Fourier-Mukai transform between $X$ and one of its Fourier-Mukai partners constructed in \cite{BMef}).  With $\olB = lf$ and $B=qf$, we have $f\ol{B}=fB=0$ and so from \ref{para:Q2} and \ref{para:Q2b} we have $Z'_{\gamma, \delta, \ol{B}}=Z_{\gamma, \delta, \ol{B}}, Z'_{\epsilon, \zeta, B}=Z_{\epsilon, \zeta, B}$ and
\begin{align*}
  Z_{\gamma, \delta, \olB}(E) &= ( -s + ld + \gamma n ) + i ( c + \delta d - ln), \\
  Z_{\epsilon, \zeta, B}(\Phi E) &= c + (e+\epsilon )d + (-\tfrac{e}{2} -q  )n + i \left( s + (-\tfrac{e}{2}  -q )d + (e-\zeta) n\right)
\end{align*}
which satisfy
\begin{equation}\label{eq:L1}
  Z_{\epsilon, \zeta, B}(\Phi E) = (-i) Z_{\gamma, \delta, \olB}(E)= \begin{pmatrix} 0 & 1 \\ -1 & 0 \end{pmatrix}Z_{\gamma, \delta, \olB}(E) \text{\quad for any $E \in D^b(X)$}
\end{equation}
provided we have the relations
\begin{equation}\label{eq:constraints-simp}
  l = \tfrac{e}{2} + q, \text{\quad}  \gamma  = \zeta - e, \text{\quad}   \delta = e + \epsilon.
\end{equation}

\paragraph We will now rewrite the computations in \ref{para:L3p-relation} in terms of $Z_{\ol{\omega}, \ol{B}}$ and $Z_{\omega, B}$.  From \eqref{eq:Zchangeofcoord}, we have
\begin{align*}
   Z_{\olw, \olB}(E) &= 
    \begin{pmatrix} 1 & 0 \\ 0 & \tfrac{\beta}{\alpha} \end{pmatrix} Z_{\tfrac{\olw^2}{2}, m+ \alpha, \olB}(E)\\
   Z_{\omega, B}(\Phi E) &= \begin{pmatrix} 1 & 0 \\ 0 & u \end{pmatrix} Z_{\tfrac{\omega^2}{2}, m+\tfrac{v}{u},B}(\Phi E)
\end{align*}
and so by choosing
\begin{equation*}
\gamma=\tfrac{\olw^2}{2}, \text{\quad} \delta = m+ \alpha, \text{\quad} \epsilon = \tfrac{\omega^2}{2}, \text{\quad} \zeta = m+ \tfrac{v}{u}
\end{equation*}
in \ref{para:L3p-relation}, the relations \eqref{eq:constraints-simp} can be rewritten as
\begin{align}
l &= \tfrac{e}{2} + q, \label{eq:lqlrelation}\\
\tfrac{\beta^2}{\alpha^2} (m+\alpha -\tfrac{e}{2}) &= m + \tfrac{v}{u}-e, \label{eq:surfcons1}\\
m+ \alpha -e &= (m-\tfrac{e}{2})u^2+uv. \label{eq:surfcons2}
\end{align}
Under these  relations,  \eqref{eq:L1} can be rewritten as
\begin{multline*}
Z_{\omega, B}(\Phi E) = \begin{pmatrix} 1 & 0 \\ 0 & u \end{pmatrix} Z_{\epsilon, \zeta, B}(\Phi E) =       \begin{pmatrix} 1 & 0 \\
  0 & u \end{pmatrix}\begin{pmatrix} 0 & 1 \\
  -1 & 0 \end{pmatrix}  Z_{\gamma, \delta, \ol{B}} (E)  \\ =\begin{pmatrix} 1 & 0 \\
  0 & u \end{pmatrix}\begin{pmatrix} 0 & 1 \\
  -1 & 0 \end{pmatrix}  \begin{pmatrix} 1 & 0 \\
  0 & \tfrac{\beta}{\alpha} \end{pmatrix}^{-1} Z_{\olw,\olB} (E)=\begin{pmatrix} \tfrac{\alpha}{\beta} & 0 \\ 0 & u \end{pmatrix} \begin{pmatrix} 0 & 1 \\
  -1 & 0 \end{pmatrix}  Z_{\ol{\omega}, \ol{B}}(E),
\end{multline*}
i.e.\
\begin{equation}\label{eq:L2}
Z_{\omega, B}(\Phi E) = \begin{pmatrix} \tfrac{\alpha}{\beta} & 0 \\ 0 & u \end{pmatrix} \begin{pmatrix} 0 & 1 \\
  -1 & 0 \end{pmatrix}  Z_{\ol{\omega}, \ol{B}}(E) = \begin{pmatrix} \tfrac{\alpha}{\beta} & 0 \\ 0 & u \end{pmatrix} (-i)  Z_{\ol{\omega}, \ol{B}}(E).
\end{equation}

\begin{remsub}\label{rem:solexistence}
Note   that equation \eqref{eq:lqlrelation} is independent of equations \eqref{eq:surfcons1} and \eqref{eq:surfcons2}.  Also, there is always a solution to \eqref{eq:surfcons1} and \eqref{eq:surfcons2}: We can rewrite \eqref{eq:surfcons2} as $$(m-\tfrac{e}{2})u^2+vu-(m+\alpha-e)=0$$ and think of it as a quadratic equation in $u$ with discriminant $\Delta = v^2 + 4(m-\tfrac{e}{2})(m+\alpha-e)$.  Since $m-e>0$ and $\alpha>0$, given any $v>0$, this quadratic always has a positive real solution in $u$.  We can then solve for $\beta$ using \eqref{eq:surfcons1}.
\end{remsub}

\paragraph[Solving equations] To \label{para:surfsolveeqs} relate Bridgeland stability to polynomial stability later on,  let us solve \eqref{eq:surfcons1} and \eqref{eq:surfcons2} in terms of Laurent series in $v$.  To this end, let us  make  the change of variable $w = \tfrac{1}{v}$ so that  \eqref{eq:surfcons2} reads
\begin{equation}\label{eq:AG48-26-1}
  u =   \left((m+\alpha-e) - (m-\tfrac{e}{2})u^2\right) w.
\end{equation}
 By \cite[Theorem 1.1, Remark 1.2(2)]{KollarLRS}, this power series is convergent on a neighbourhood of $0$, i.e.\   the relation \eqref{eq:surfcons2} defines $u$ as an implicit function in $v$ (hence $w$), and $u$ can be represented by an element of $(m+\alpha-e)w\RPwc$.  Since $m-e >0$ from \ref{para:ellipsurfnotation} and $\alpha>0$, it follows that     $u=\Theta (\tfrac{1}{v})$ for $v \gg 0$; in particular, as $v \to \infty$ we have $u \to 0^+$.

Since $m>e \geq \tfrac{e}{2} \geq 0$, for any $\alpha, u, v \in \mathbb{R}_{>0}$, the equation \eqref{eq:surfcons1} has a solution
\begin{equation}\label{eq:beta1}
  \beta = \alpha \sqrt{\frac{m+\tfrac{v}{u}-e}{m+\alpha -\tfrac{e}{2}}} >0.
\end{equation}
If we set $C= \alpha^2/(m+\alpha-\tfrac{e}{2})$, then   \eqref{eq:surfcons1} can  be rewritten as
\[
  \beta^2 = C \left( \tfrac{1}{uw} +  (m-e)\right).
\]
Since $u \in w\RPwc$  and $u$ has a zero of order 1 at $w=0$ from above, $\tfrac{1}{u^2}$ has a pole of order 2 at $w=0$ and $\tfrac{1}{u^2} \in \tfrac{1}{w^2} \RPwc$.  On the other hand, from \eqref{eq:AG48-26-1} we have
\begin{equation*}
  \tfrac{1}{uw} =   (m+\alpha-e)\tfrac{1}{u^2} - (m-\tfrac{e}{2}),
\end{equation*}
which also has a pole of order 2 at $w=0$.  Overall, we obtain
\[
  \beta^2=\tfrac{C}{m+\alpha-e}w^{-2}  + (\text{higher-degree terms in $w$}).
\]
In particular, $\beta^2$ is an implicit function in $v$ that is  represented by an element of $\tfrac{1}{w^2}\RPwc$.  Since the lowest-degree term of $\beta^2$ has positive coefficient, we can solve for  $\beta$  as a series in $w$ directly and see that $\beta$ is represented by an element of $\tfrac{1}{w}\RPwc$.  Note that $\beta=\Theta(v)$ as $v \to \infty$.  Regarding $u, \beta$ as elements of  $\RLoovc$ allows us to consider $Z_{\olw,\olB}, Z_{\omega, B}$ as group homomorphisms $K(X) \to \CLoovc$.

\subparagraph In a subsequent article \cite{LoWong}, the coefficients of $u$ as a power series in $w$ are worked out explicitly in terms of the constants $m, \alpha, e$ and the Catalan numbers, along with the radius of convergence of $u$.    The coefficients of $\beta$ as a power series in $w$, and hence the coefficients of $u$ and $\beta$ as Laurent series in $\tfrac{1}{v}$, can then be worked out completely following the recipe above.

We include the following lemma here, although it will not be needed until Theorem \ref{cor:AG48-58-1}:

\begin{lemsub}\label{lem:betaderpos}
Consider $\beta, u$ as function in $v$ as in \ref{para:surfsolveeqs}, and suppose $v_0\in \mathbb{R}_{>0}$ is such that, on $[v_0, \infty)$, both $\beta, u$ are positive and convergent as Laurent series in $v$.  Then there exists $v' > v_0$ such that $u$ is monotone decreasing while $\beta$ is monotone increasing for $v \in [v',\infty)$.
\end{lemsub}

\begin{proof}
From \eqref{eq:beta1}, it suffices to show that $\tfrac{v}{u}$ is monotone increasing for $v \gg 0$.  On the other hand, rewriting \eqref{eq:surfcons2} as
\[
  (m+\alpha -e)\tfrac{1}{u^2} = (m-\tfrac{e}{2})+ \tfrac{v}{u},
\]
it suffices to show that  $\tfrac{1}{u^2}$ is monotone increasing.  From \eqref{eq:AG48-26-1}, we have
\[
  u = (m+\alpha -e)w \big( 1 - (m-\tfrac{e}{2})w^2 + \text{ higher-order terms in $w$}\big)
\]
and hence
\begin{align*}
  \tfrac{1}{u^2} &= \tfrac{1}{(m+\alpha -e)^2} \big( \tfrac{1}{w^2} + 2(m-\tfrac{e}{2}) + \text{ higher-order terms in $w$}\big)\\
  &= \tfrac{1}{(m+\alpha -e)^2} \big( v^2 + 2(m-\tfrac{e}{2}) + \text{ lower-order terms in $v$}\big),
\end{align*}
 from which we see the derivative of $\tfrac{1}{u^2}$ with respect to $v$ is a  series where the highest-degree term is $\tfrac{2}{(m+\alpha-e)^2}v$, and so  $\tfrac{1}{u^2}$ is monotone increasing as  a function in $v$ on the interval $[v', \infty)$ for some $v' >v_0$.
\end{proof}

\paragraph[The heart $\Bl$] With $\omega=u(\Theta +mf)+vf$ and $B=qf$,  for any coherent sheaf $A$ of nonzero rank on $X$ we have
\[
\mu_{\omega, B}(A)=\frac{\omega\ch_1^B(A)}{\ch_0^B(A)} = u \mu_{\Theta + mf, B}(A) + v \mu_f (A).
\]
The proof of \cite[Lemma 3.1]{LLM} can then be easily modified to show:

\begin{lemsub}\label{lem:LLMlem3-1an}
Suppose $m \in \mathbb{R}_{>0}$ is as in \ref{para:ellipsurfnotation}, and $\omega = u(\Theta +mf)+vf$   where $u, v \in  \mathbb{R}_{>0}$,   $u_0 \in \mathbb{R}_{>0}$ is fixed, $B=qf$ for some $q \in \mathbb{R}$  and $F$ is a  coherent sheaf  on $X$.
\begin{itemize}
\item[(1)] The following are equivalent:
  \begin{itemize}
  \item[(a)] There exists $v_0 \in \mathbb{R}_{>0}$ such that $F \in \Fc_{\omega, B}$ for all $(v,u) \in (v_0,\infty) \times (0,u_0)$.
  \item[(b)] There exists $v_0 \in \mathbb{R}_{>0}$ such that, for every nonzero subsheaf $A \subseteq F$, we have $\mu_{\omega, B}(A) \leq 0$ for all $(v,u) \in (v_0,\infty) \times (0,u_0)$.
  \item[(c)] For every nonzero subsheaf $A \subseteq F$, either (i) $\mu_f (A) < 0$, or (ii) $\mu_f(A)=0$ and also $\mu_{\Theta + mf, B}(A) \leq 0$.
  \end{itemize}
\item[(2)] The following are equivalent:
  \begin{itemize}
  \item[(a)] There exists $v_0 \in \mathbb{R}_{>0}$ such that $F\in \Tc_{\omega, B}$ for all $(v,u) \in (v_0,\infty) \times (0,u_0)$.
  \item[(b)] There exists $v_0 \in \mathbb{R}_{>0}$ such that, for every nonzero quotient sheaf $A$ of $F$ we have $\mu_{\omega, B}(A) > 0$ for all $(v,u) \in (v_0,\infty) \times (0,u_0)$.
  \item[(c)] For every nonzero quotient sheaf $A$ of $F$, either (i) $\mu_f (A) > 0$, or (ii) $\mu_f(A)=0$ and $\mu_{\Theta + mf, B}(A) > 0$.
\end{itemize}
\end{itemize}
\end{lemsub}

We now define $\Fc^l$ (resp.\ $\Tc^l$) to be the extension closure in $\Coh (X)$ of all coherent sheaves $F$ satisfying condition (1)(c) (resp.\ (2)(c)) in Lemma \ref{lem:LLMlem3-1an}, and define the extension closure in $D^b(X)$
\[
  \Bl = \langle \Fc^l[1], \Tc^l \rangle.
\]

\subparagraph Now suppose $u : \mathbb{R}_{>0} \to \mathbb{R}_{>0}$ is any continuous function in $v$ such that $u \to 0$ as $v \to \infty$, such as when $u$ is defined as an implicit function in $v$ via  \eqref{eq:surfcons2} as in \ref{para:surfsolveeqs}.  Then we can equivalently define
\begin{align*}
\Tc^l &= \{ F \in \Coh (X) : F \in \Tc_{\omega, B} \text{ for all $v \gg 0$}\}, \\
  \Fc^l &= \{ F \in \Coh (X) : F \in \Fc_{\omega, B} \text{ for all $v \gg 0$}\}
\end{align*}
just as we did in \cite[Remark 4.4(vi)]{Lo14}.

\paragraph[Fixing a relation between $u,v$] From now on, we will assume that $u, v$ satisfy the relation \eqref{eq:surfcons2} and  regard $u$ as a function in $v$ as we did  in \ref{para:surfsolveeqs}.  Then many of the arguments and results in \cite{LLM} carry over directly, including: $\Coh^{\leq 1}(X) \subset \Tc^l$, $\Fc^l \subset \Coh^{=2}(X)$, $W_{0,\whPhi} \subset \Tc^l$, and that $f\ch_1 (E) \geq 0$ for all $E \in \Bl$ (see \cite[3.2]{LLM} and \cite[Remark 4.4]{Lo14}).  Replacing $\mu^\ast$ with $\mu_{\Theta + mf, B}$ in the proof of \cite[Lemma 4.6]{Lo14} (see also \cite[Lemma 3.3]{LLM}) gives that $(\Tc^l, \Fc^l)$ is a torsion pair in $\Coh (X)$, and so $\Bl$ is the heart of a t-structure on $D^b(X)$.  The inclusion $W_{0,\whPhi} \subset \Tc^l$ then implies $\Fc^l \subset W_{1,\whPhi}$.  We also have the analogue of \cite[Lemma 3.4]{LLM}:

\begin{lemsub}\label{lem:Zlstabfunc}
Fix $m, \alpha \in \mathbb{R}_{>0}$  as in \ref{para:ellipsurfnotation},  consider $u$ as a function in $v$ as in \ref{para:surfsolveeqs} and let $B \in \mathbb{R}f$.  Then for any nonzero $F \in \Bl$, we have $Z_{\omega, B} (F) \in \mathbb{H}$ for $v\gg 0$.
\end{lemsub}

\paragraph[$Z^l$-stability]  Suppose  \label{para:surfZupperl}  $X$  is a Weierstra{\ss} elliptic surface.  By  regarding $u$ as a function in $v$ as in \ref{para:surfsolveeqs}, for any $B \in \mathbb{R}f$ we now define a group homomorphism $Z^l : K(D^b(X))\to  \CLoovc$ via
\[
Z^l(E) = Z_{\omega,B} (E)
\]
for $E \in D^b(X)$. By Lemma \ref{lem:Zlstabfunc},   the pair  $(Z^l,\Bc^l)$ is a polynomial stability function with respect to $(0,1]$.  (The HN property of $Z^l$ on $\Bl$  will be established below  using Proposition \ref{prop:AHNgivesBHN}.)

The next main goal in this section is to establish a relation between $Z_l$-stability and $Z^l$-stability.  For this, we need a little more technical preparation.

\paragraph Under \label{para:obs1} the relations \eqref{eq:lqlrelation}, \eqref{eq:surfcons1} and \eqref{eq:surfcons2},   any $E \in D^b(X)$ with $\ch_0(E)= 0$ satisfies
\[
-\ch_2^\olB ( E) = \Re Z_{\olw, \olB}(E) = -\tfrac{1}{u} \Im Z_{\omega, B}(\Phi E) = - \tfrac{1}{u} \omega \ch_1^B (\Phi E)
 \]
 where the second equality follows from \eqref{eq:L2}.

\paragraph  Consider the abelian subcategory $\{\Coh^{\leq 0}\}^\uparrow$  of $\Coh^{\leq 1}(X)$.  Any nonzero $E \in \Coh^{\leq 1}(X)$ with $f\ch_1(E)=0$ must be a fiber sheaf; if $E$ also lies in $\{\Coh^{\leq 0}\}^\uparrow$, then $E$ must be a 0-dimensional sheaf and hence $\ch_2(E)>0$.  Hence for any $\mathbb{R}$-divisor $\ol{B}$ on $X$, we have the slope function on $\{\Coh^{\leq 0}\}^\uparrow$
\[
  \mu_{\ast,\olB} (E) = \begin{cases}
  \frac{\ch_2^\olB(E)}{f\ch_1(E)} &\text{ if $f\ch_1(E) \neq 0$} \\
  \infty &\text{ if $f\ch_1(E)=0$}
  \end{cases} \text{\quad for } E \in \{\Coh^{\leq 0}\}^\uparrow.
\]
Note that $\ch_1^\olB(E)=\ch_1(E)$ for any $E \in \Coh^{\leq 1}(X)$.    The HN property of $\mu_{\ast,\olB}$ on $\{\Coh^{\leq 0}\}^\uparrow$  follows from \cite[Proposition 3.4]{LZ2}.

\begin{lem}\label{lem:AG47-30-1}
Suppose $\olB, B \in \mathbb{R}f$ and $E \in \{\Coh^{\leq 0}\}^\uparrow$ is a $\mu_{\ast,\olB}$-semistable sheaf supported in dimension 1.  Then $\Phi E$ lies in $\Tc^l$ (resp.\ $\Fc^l$) if $\ch_2^\olB(E)>0$ (resp.\ $\ch_2^\olB(E)\leq 0$).
\end{lem}

\begin{proof}
Every coherent sheaf in $\{\Coh^{\leq 0}\}^\uparrow$ is  $\Phi$-WIT$_0$  \cite[Remark 3.14]{Lo11},  so $\Phi E \in \Coh (X)$.  Take any short exact sequence of sheaves
\[
0 \to M\to \Phi E \to N \to 0
\]
where $M, N \neq 0$.  Since $\Phi E$ is $\whPhi$-WIT$_1$, so is $M$.  Hence we have a long exact sequence of sheaves
\[
 0 \to \whPhi^0N \to \wh{M} \overset{g}{\to} E \to \whPhi^1N \to 0.
\]
Since $E$ lies in $\{\Coh^{\leq 0}\}^\uparrow$ and  is supported in dimension 1, we have $f\ch_1 (E)\neq 0$; we also know $\Phi E$ is a torsion-free sheaf  \cite[Lemma 4.3(ii)]{Lo11},  so $M$ is also torsion-free.

Since $\olB, B \in \mathbb{R}f$, we have $f\ch_1^\olB(F)=f\ch_1(F)$ for any coherent sheaf $F$ on $X$ and similarly for $\ch_1^B$.

\textbf{Case 1: $\ch_2^\olB(E)>0$.} To prove  $\Phi E \in \Tc^l$, we  use the characterisation of $\Tc^l$ in Lemma \ref{lem:LLMlem3-1an}(2)(c), by which it suffices to show either $\mu_f (N)>0$, or $\mu_f (N)=0$ together with $\mu_{\Theta + mf,B}(N)>0$, for any nonzero quotient sheaf $N$ of $\Phi E$.  In particular, we can assume $\ch_0(N)\neq 0$.

That $M$ is $\whPhi$-WIT$_1$ gives  $f\ch_1(M)\leq 0$  \cite[Lemma 5.4(a)]{Lo15}, while  $\ch_0(E)=0$ gives $f\ch_1(\Phi E)=0$.  Hence  $f\ch_1(N)\geq 0$.  Suppose $f\ch_1(N)=0$.  Then  $f\ch_1(M)=0$ and hence $\ch_0(\wh{M})=0$, i.e.\ $\wh{M}$ is a torsion sheaf. This implies  $\whPhi^0 N$ is a $\Phi$-WIT$_1$ torsion sheaf, and hence  a fiber sheaf \cite[Lemma 2.6]{Lo7}, and so  $\ch_2^\olB (\whPhi^0 N) \leq 0$ \cite[Lemma 5.4(a)]{Lo15}.

If $\whPhi^1N$ is also a fiber sheaf, then $N$ itself is a fiber sheaf, and hence a torsion sheaf,  and so  $\mu_{\Theta +mf,B}(N)=\infty$.

So let us assume $\whPhi^1 N$ is not a fiber sheaf, in which case  $\mu_{\ast,\olB}(\whPhi^1 N) < \infty$.    The $\mu_{\ast,\olB}$-semistability of $E$ then gives $\ch_2^\olB (\whPhi^1 N)>0$, and  we have $\ch_2^\olB (\whPhi N)<0$ overall.  We are also assuming $f\ch_1 (N)=0$ which implies $\ch_0 (\whPhi N)=0$.   Applying  \ref{para:obs1} to $\whPhi N$ and noting  $0>\ch_2^\olB(\whPhi N)$, we obtain $\omega \ch_1^B (N)>0$.  Since $f\ch_1^B (N)=0$, we have $\omega \ch_1^B (N)=u(\Theta + mf)\ch_1^B(N)>0$ and hence $\mu_{\Theta +mf,B}(N)>0$.

\textbf{Case 2: $\ch_2^\olB(E)\leq 0$.}  In this case, we want to prove that $\Phi E \in \Fc^l$.  By the characterisation of $\Fc^l$ in Lemma \ref{lem:LLMlem3-1an}(1)(c), it suffices to show either $\mu_f(M)<0$, or $\mu_f(M)=0$ together with $\mu_{\Theta + mf,B}(M)\leq 0$.  As in Case 1, we know $M$ is a torsion-free sheaf with $f\ch_1(M)\leq 0$. If $f\ch_1(M)<0$ then we are done, so let us assume $f\ch_1(M)=0$, which gives  $f\ch_1(N)=0$.  The same argument as in Case 1 then gives $\ch_2^\olB (\whPhi^0 N) \leq 0$.

If $\image g =0$, then $\whPhi^0 N \cong \wh{M}$ is both $\Phi$-WIT$_1$ and $\Phi$-WIT$_0$, forcing $M=0$, a contradiction.  Hence $\image g \neq 0$.  Since $\image g$ is a subsheaf of $E$, which is pure 1-dimensional and lies in $\{\Coh^{\leq 0}\}^\uparrow$, it follows that  $f\ch_1(\image g) \neq 0$.  Hence $\ch_2^\olB (\image g) \leq 0$ by the $\mu_{\ast,\olB}$-semistability of $E$.  Thus $\ch_2^\olB (\wh{M})=\ch_2^{\ol{B}}(\whPhi^0 N) + \ch_2^{\ol{B}}(\image g)\leq 0$.  Applying \ref{para:obs1} to $\wh{M}$ now gives $\omega \ch_1^B(M)\leq 0$.  Since $f\ch_1(M)=0$, it follows that $(\Theta + mf)\ch_1^B(M) \leq 0$ and hence  $\mu_{\Theta + mf,B}(M)\leq 0$.
\end{proof}

\begin{prop}\label{prop:AG47-30-1}
Fix an $\mathbb{R}$-divisor  $B \in \mathbb{R}f$ on $X$.  Then $\Phi \Coh^p \subset D^{[0,1]}_{\Bl}$.
\end{prop}

\begin{proof}[Proof of Proposition \ref{prop:AG47-30-1}]
For any $E \in \Coh^p$, consider  the canonical exact triangle
\[
H^{-1}(E)[1] \to E \to H^0(E) \to H^{-1}(E)[2]
\]
where $H^{-1}(E)\in \Coh^{=2}(X)$  and $H^0(E) \in \Coh^{\leq 1}(X)$.  Fix an $\mathbb{R}$-divisor $\ol{B} \in \mathbb{R}f$ on $X$. Since $\{\Coh^{\leq 0}\}^\uparrow$ is a torsion class in $\Coh^{\leq 1}(X)$, we have a filtration in $\Coh^{\leq 1}(X)$
\[
 A_0 \subseteq A_1 \subseteq A_2 \subseteq A_3 = H^0(E)
\]
where $A_2$ is the maximal subsheaf of $H^0(E)$ lying in $\{\Coh^{\leq 0}\}^\uparrow$, $A_0$ is the maximal subsheaf of $A_2$ lying in $\Coh^{\leq 0}(X)$, all the $\mu_{\ast,\olB}$-HN factors of $A_1/A_0$ have $\mu_{\ast,\olB}$ in the range $(0,\infty)$, and all the $\mu_{\ast,\olB}$-HN factors of $A_2/A_1$ have $\mu_{\ast,\olB}$ in the range $(-\infty,0]$.  Since $\Coh^{\leq 1}(X) \subset \Coh^p$, we can construct a  filtration of $E$ in $\Coh^p$ with factors $H^{-1}(E)[1], A_0, A_1/A_0, A_2/A_1$ and $A_3/A_2$.

Since $H^{-1}(E)$ is a torsion-free sheaf, we have $\Phi H^{-1}(E)[1] \in \Bl$ by \cite[3.9]{LLM} (the proof of which depends on \cite[3.7, 3.8]{LLM}, which still apply).  Next we have $\Phi A_0 \in \Coh (p)_0 \subset \Bl$.  Then, by Lemma \ref{lem:AG47-30-1}, we have $\Phi (A_1/A_0) \in \Tc^l \subset \Bl$ while $\Phi (A_2/A_1) \in \Fc^l \subseteq \Bl [-1]$.  Lastly, since $A_3/A_2$ is necessarily a fiber sheaf, both its $\Phi$-WIT$_0$ and $\Phi$-WIT$_1$ components are fiber sheaves, implying $\Phi (A_3/A_2) \in \langle \Coh (p)_0, \Coh (p)_0 [-1] \rangle \subseteq D^{[0,1]}_{\Bl}$.  Overall, we have $\Phi E \in D^{[0,1]}_{\Bl}$ and the proposition follows.
\end{proof}

We now have all the input for Configuration III to prove a correspondence theorem between $Z_l$-stability and $Z^l$-stability via the autoequivalence $\Phi$:

\begin{thm}\label{thm:main5}
Let $p : X \to B$ be a Weierstra{\ss} elliptic surface and $\Phi : D^b(X) \to D^b(X)$ the Fourier-Mukai transform given by a relative Poincar\'{e} sheaf as in \ref{para:FMTs}.  Let $\ol{B}, B$ be divisors in $\mathbb{R}f$ satisfying \eqref{eq:lqlrelation}, and let $\ol{\omega}, \omega$ be ample $\mathbb{R}$-divisors written in the form of \ref{para:ellipsurfnotation}, where we regard $u, \beta$ as elements of $\RLoovc$ as in \ref{para:surfsolveeqs} (so that  \eqref{eq:surfcons1} and \eqref{eq:surfcons2} holds).  Then for any $E\in D^b(X)$,
\[
\text{ $E$ is $Z_l$-semistable if and only if $\Phi E$ is $Z^l$-semistable.}
\]
The same  result  holds when `semistable' is replaced with `stable' on both sides.  In fact, we have
\[
  \Phi \cdot (Z_l,\Pc_{Z_l}) = (Z^l,\Pc_{Z^l}) \cdot \Patilde (T)
\]
for some element $\Patilde (T)$ of $\Patilde$, where $\Pc_{Z_l}, \Pc_{Z^l}$ denote the slicings for the polynomial stability conditions $(Z_l,\Coh^p), (Z^l,\Bc^l)$, respectively.
\end{thm}

\begin{proof}
We will check that we are in Configuration III; the theorem will then follow from Theorem \ref{prop:paper30prop11-11ext}(iv).

Let $\Dc=\Uc=D^b(\Coh (X))$, $\Phi$ as described and $\Psi=\whPhi$ where $\whPhi$ is as in \ref{para:FMTs}.  Let $\Ac = \Coh^p, \Bc=\Bl$, $Z_\Ac = Z_l$ and $Z_\Bc = Z^l$, while  $a=\tfrac{1}{4}, b=0$ and $T = \begin{pmatrix} \tfrac{\alpha}{\beta} & 0 \\
  0 & u \end{pmatrix}\begin{pmatrix} 0 & 1 \\
  -1 & 0 \end{pmatrix}\in \GLlp$.

 Condition \ref{para:configIII}(a) now follows from the identities in \ref{para:FMTs}, while condition \ref{para:configIII}(b) follows from Proposition \ref{prop:AG47-30-1}.  From \ref{para:surfZlowerl}, we know $Z_\Ac$ is a polynomial stability function on $\Ac$ with respect to $(a,a+1]$; from \ref{para:surfZupperl}, we know $Z_\Bc$ is a polynomial stability function on $\Bc$ with respect to $(b,b+1]$.  The identity \eqref{eq:AG45-108-17} then follows from \eqref{eq:L2}.  Choosing $\Gamma_T$ to be the branch satisfying  $\Gamma_T(\phi')= \phi' -\tfrac{1}{2}$ for all $\phi' \in \tfrac{1}{2}\mathbb{Z}$, we obtain $\Gamma_T^{-1}(0)=\tfrac{1}{2}$ and so \eqref{eq:AG47-23-1} holds.  Therefore, we are in Configuration III and the first claim   follows from Theorem \ref{prop:paper30prop11-11ext}(iv).  The second claim then follows from Theorem \ref{prop:paper30prop11-11ext} together with Lemma \ref{cor:ZupplstabHN} below.
\end{proof}

\begin{lem}\label{cor:ZupplstabHN}
Let $p : X \to B$ be a Weierstra{\ss} elliptic surface, $B$ a divisor in $\mathbb{R}f$, $\omega$ an ample $\mathbb{R}$-divisor on $X$ written as in \ref{para:ellipsurfnotation}, where we regard $u$ as an element of $\RLoovc$ as in \ref{para:surfsolveeqs} (so that \eqref{eq:surfcons2} holds).  Then $(Z^l,\Bl)$ is a polynomial stability condition on $D^b(X)$ with respect to $(0,1]$.
\end{lem}

\begin{proof}
From the proof of Theorem \ref{thm:main5}, we know that if we choose $\ol{B}, \ol{\omega}$ in a suitable manner then  $(Z_l,\Coh^p)$ and $(Z^l,\Bl)$ are in Configuration III.  Since $(Z_l,\Coh^p)$ has the HN property  \cite[Theorem 3.2.2]{BayerPBSC}, the HN property of $(Z^l,\Bl)$ follows from Proposition \ref{prop:AHNgivesBHN}.
\end{proof}

\paragraph[Bogomolov inequality] On a smooth projective surface $X$,  we  define the discriminant of any object $E\in D^b(X)$ as
\[
  \Delta (E) = \ch_1(E)^2 - 2\ch_0(E)\ch_2(E).
\]
The usual Bogomolov inequality on a smooth projective surface $X$ says, that when $E$ is a torsion-free slope semistable sheaf on $X$, we have $\Delta (E) \geq 0$.  It is easy to check that $\Delta$ is invariant under the shift functor in $D^b(X)$ and twisting the Chern character by a $B$-field.

\subparagraph Let  $p : X \to B$ be a Weierstra{\ss} elliptic surface.  Given any $E \in D^b(X)$, a straightforward computation using the formulas in \ref{para:surfcohomFMT}  shows
\begin{equation}\label{eq:L19p-1}
  \Delta (\Phi E) + e (f\ch_1(\Phi E))^2= \Delta (E)+ e(f\ch_1(E))^2,
\end{equation}
or equivalently
\begin{equation}\label{eq:L19p-1b}
  \Delta (\Phi E) - e (\ch_0(\Phi E))^2= \Delta (E)- e(\ch_0(E))^2.
\end{equation}
If $E\in \Coh^p$ is a $Z_l$-semistable object of nonzero rank, then  $H^{-1}(E)$ is a torsion-free slope semistable sheaf while $H^0(E)$ is a 0-dimensional sheaf \cite[Lemma 4.2]{BayerPBSC}, and so
\[
\Delta (E) = \Delta (H^{-1}(E)) -2\ch_0(H^{-1}(E)[1])\ch_2(H^0(E)) \geq \Delta (H^{-1}(E))\geq 0.
\]
That is, a $Z_l$-semistable object $E$ satisfies $\Delta (E)\geq 0$.  Therefore, by Theorem \ref{thm:main5} and Lemma \ref{cor:ZupplstabHN}, any $Z^l$-semistable object $F$ (i.e.\ a shift of some $Z^l$-semistable object in $\Bl$) satisfies
\[
  \Delta (F)  \geq e (\ch_0(F)^2-(f\ch_1(F))^2).
\]
In particular, when $e=0$ (such as when $X$ is a product and $p$ is the trivial elliptic fibration), we have $\Delta (F)\geq 0$ for any $Z^l$-semistable object $F$.

\section{Elliptic surfaces: Bridgeland stability}\label{sec:ellsurfBristab}

In this section, $p : X \to B$ will be a Weierstra{\ss} elliptic fibration.  We will fix the parameter $v$ (and subsequently $u, \beta$) and consider the resulting Bridgeland stability conditions.  We give a countable family of rays in the space of Bridgeland stability conditions on $D^b(X)$ such that, for every stability condition $\sigma$ on such a ray sufficiently away from the origin, we are able to describe precisely the image of $\sigma$ under the action of the autoequivalence $\Phi$.

\paragraph[Notation] Let us continue  \label{para:vchoice}  with the notation from \ref{para:ellipsurfnotation}, where $m, \alpha \in \mathbb{R}_{>0}$ are fixed such that $\Theta + kf$ is ample on $X$ for all $k \geq m$ (which implies $m >e$).  Let us now fix parameters $l,q,\beta, u, v$ so that  \eqref{eq:lqlrelation},  \eqref{eq:surfcons1} and \eqref{eq:surfcons2} all hold, in which case $\olw, \olB, \omega, B$ are fixed divisors such that   the relation \eqref{eq:L2}
\begin{equation*}
Z_{\omega, B}(\Phi E) =  \begin{pmatrix} \tfrac{\alpha}{\beta} & 0 \\ 0 & u \end{pmatrix} (-i)  Z_{\ol{\omega}, \ol{B}}(E)
\end{equation*}
holds for any $E \in D^b(X)$.  Here,  $\left(\begin{smallmatrix} \tfrac{\alpha}{\beta} & 0 \\ 0 & u \end{smallmatrix}\right)$ is an element of $\mathrm{GL}^+(2,\mathbb{R})$.

\paragraph Let  \label{para:Zwpreimag-1} $\whPhi$ be a quasi-inverse of $\Phi [1]$ as in \ref{para:FMTs} so that $\whPhi \Phi \cong \mathrm{id}[-1]$, and let us      write $\whPhi^K$ to denote the induced map on $K(X)$.  Since $(Z_{\olw,\olB}, \Bc_{\olw,\olB})$ is a Bridgeland stability condition with respect to $(0,1]$, the pair $(Z_{\olw,\olB} \circ \whPhi^K, \Phi \Bc_{\olw,\olB}[1])$ is also a Bridgeland stability condition with respect to $(0,1]$.  For any $F \in \Phi \Bc_{\olw,\olB}[1]$, we have $\whPhi F \in \Bc_{\olw,\olB}$ and
\[
\begin{pmatrix} \tfrac{\alpha}{\beta} & 0 \\
  0 & u \end{pmatrix}\begin{pmatrix} 0 & 1 \\
  -1 & 0 \end{pmatrix}  \left( (Z_{\olw,\olB}\circ \whPhi^K) (F)\right) = Z_{\omega,B}(\Phi \whPhi F)= Z_{\omega,B} (F[-1])
\]
from \eqref{eq:L2}.  Hence $(Z_{\omega,B}, \Phi \Bc_{\olw,\olB}[1])$ is a Bridgeland stability condition with respect to $(\tfrac{1}{2},\tfrac{3}{2}]$.  In particular, for any $0 \neq F \in \Phi \Bc_{\olw,\olB}[1]$ we have  $\Re Z_{\omega,B} (F)\in\mathbb{R}_{\leq 0}$, and if $\Re Z_{\omega,B} (F)=0$ then $\Im Z_{\omega,B} (F)<0$ must hold.

\begin{lem}\label{lem:AG48-53-1}
$\Coh^{\leq 0}(X)$ is a Serre subcategory of $\Bl$.
\end{lem}

\begin{proof}
Take any $T \in \Coh^{\leq 0}(X)$ and any $\Bl$-short exact sequence $0 \to M \to T \to N \to 0$.  Since $\Coh^{\leq 0}(X) \subset \Bl$, it follows that $H^{-1}(M)=0$ and we have an exact sequence of sheaves $0 \to H^{-1}(N) {\to} H^0(M) \overset{\alpha}{\to} T \to H^0(N) \to 0$ in which $\image \alpha \in \Coh^{\leq 0}(X)$.  Then $f\ch_1^B(H^{-1}(N))=f\ch_1^B(H^0(M))=0$ from Lemma \ref{lem:LLMlem3-1an}, which implies $H^{-1}(N) \in \Fc^{l,0}$ and $H^0(M) \in \langle \Coh^{\leq 1}(X),\Tc^{l,0}\rangle$ using the notation in and given the torsion quintuple in $\Bl$ from \cite[(3.7.2)]{LLM}.

If $\ch_0(H^0(M))>0$, then from the description of the categories $\Tc^{l,0}$ and $\Fc^{l,0}$ in \cite[3.7]{LLM} and Lemma \ref{lem:LLMlem3-1an}, we have $(\Theta+mf) \ch_1^B(H^0(M))>0$ while $(\Theta+mf) \ch_1^B(H^{-1}(N))\leq 0$, a contradiction.  Hence $\ch_0(H^0(M))=0$, forcing $H^{-1}(N)=0$, and the lemma follows.
\end{proof}

\begin{lem}\label{lem:AG47-123-1}
Suppose $X$ is a smooth projective surface and $\omega, B$ are fixed $\mathbb{R}$-divisors on $X$ where $\omega$ is ample.  Suppose $\ol{m} = \min \{ \omega \ch_1^B(F) > 0 : F \in \Coh (X)\}$ exists and $A$ is  a pure 1-dimensional sheaf on $X$ such that $\omega  \ch_1^B(A) = \ol{m}$.  Then  $A$ is a stable object with respect to the Bridgeland stability condition $(Z_{\omega, B}, \Bc_{\omega,B})$.
\end{lem}

\begin{proof}
Let us write $\phi$ to denote the phase function associated to the Bridgeland stability $(Z_{\omega,B}, \Bc_{\omega,B})$.  Take any $\mathcal{B}_{\omega,B}$-short exact sequence $0 \to M \to A \to N \to 0$ where $M,N \neq 0$, and consider the corresponding long exact sequence of sheaves
\[
  0 \to H^{-1}(N) \to H^0(M) \overset{\alpha}{\to} A \to H^0(N) \to 0.
\]
If $\dimension (\image  \alpha)=0$, then by the purity of $A$ we  have $\image \alpha =0$, forcing  $H^{-1}(N)=0=H^0(M)$, i.e.\ $M=0$, a contradiction.  Hence  $\dimension (\image \alpha)=1$ and $\omega \ch_1^B(\image \alpha)>0$.  The minimality assumption on $\omega \ch_1^B(A)$  then implies $\omega \ch_1^B(\image \alpha)=\ol{m}=\omega\ch_1^B(A)$, and so $H^0(N)$ is supported in dimension 0.  We now divide into two cases:

\textbf{Case 1: $\ch_0 (H^{-1}(N))=0$.}  Then  $H^{-1}(N)=0$ and $N=H^0(N)$ is a 0-dimensional sheaf from above, giving us $\phi (N)=1 > \phi (A)$.

\textbf{Case 2: $\ch_0(H^{-1}(N)>0$.}  Then $\ch_0(H^{-1}(N))=\ch_0(H^0(M))>0$ and so $\omega \ch_1^B (H^0(M))>0$.  Since $-\omega \ch_1^B (H^{-1}(N))\geq 0$ and
\[
\ol{m}=\omega \ch_1^B (\image \alpha) =  \omega \ch_1^B (H^0(M)) - \omega \ch_1^B (H^{-1}(N))\notag,
\]
the minimality of $\ol{m}$ implies  $\omega \ch_1^B (H^0(M))=\ol{m}$ and  $\omega \ch_1^B (H^{-1}(N))=0$.  Since $H^0(N)$ is supported in dimension 0, it follows that $\omega \ch_1^B (N)=0$ whereas $\omega \ch_1^B (M) > 0$, giving us $\phi (M) < \phi (N)$.

Hence $A$ is a $Z_{\omega,B}$-stable object in $\Bc_{\omega,B}$.
\end{proof}

In part (a) of the next proposition, we describe the transform of the heart $\Bc_{\ol{\omega},\ol{B}}$ under $\Phi$ in terms of the slicing for the polynomial stability condition $(Z^l,\Bl)$ in \ref
{para:surfZupperl} (see also Lemma \ref{cor:ZupplstabHN}); for part (b) of the proposition, recall that $m, \alpha$ are defined in \ref{para:ellipsurfnotation} and $l$ defined in \ref{para:L3p-relation}, and recall from \ref{para:Zwpreimag-1} that $(Z_{\omega,B}, \Phi \Bc_{\olw,\olB}[1])$ is a Bridgeland stability condition with respect to $(\tfrac{1}{2},\tfrac{3}{2}]$.

\begin{prop}\label{prop:AG48-52-1}
Recall the notation in \ref{para:vchoice}.
\begin{itemize}
\item[(a)] $\Phi \Bc_{\olw,\olB}[1] = \Pc_{Z^l}(\tfrac{1}{2},\tfrac{3}{2}]$ where $\Pc_{Z^l}$ is the $\RLoovc$-valued slicing of the polynomial stability condition $(Z^l,\Bl)$.
\item[(b)] Suppose $m+\alpha$ and $l$ are both integers.  Then for any closed point $x \in X$, its structure sheaf $\OO_x$ is stable with respect to the Bridgeland stability condition $(Z_{\omega,B}, \Phi \Bc_{\olw,\olB}[1])$.
\end{itemize}
\end{prop}

\begin{proof}
(a) Let $\Pc_{Z_l}$ denote the polynomial slicing of the polynomial stability condition $(Z_l,\Coh^p)$.  From the proof of \cite[Proposition 4.1]{BayerPBSC}, we have $\Bc_{\olw,\olB}=\Pc_{Z_l}(0,1]$. By Theorem \ref{thm:main5} and  \eqref{eq:L2}, we have $\Phi \Bc_{\olw,\olB}[1]=\Pc_{Z^l}(\tfrac{1}{2},\tfrac{3}{2}]$.  (Note that $\OO_x$, being an object in $\Bl$ of maximal phase 1 with respect to $Z^l$, lies in $\Pc_{Z^l}(1)$, while $\whPhi \OO_x$, being a fiber sheaf, lies in $\Bc_{\olw,\ol{B}}$.)

(b) Let $x \in X$ be a closed point.  By the discussion in \ref{para:Zwpreimag-1}, the skyscraper sheaf $\OO_x$ is stable with respect to $(Z_{\omega,B}, \Phi \Bc_{\olw,\olB}[1])$ if and only if $\whPhi \OO_x$ is stable with respect to the Bridgeland stability condition $(Z_{\olw,\olB},\Bc_{\olw,\olB})$.  For any $A\in \Coh (X)$, we have
\[
  \olw \ch_1^\olB (A)=\olw \ch_1(A) - \olw \olB \ch_0(A) = \tfrac{\beta}{\alpha} \Big( (\Theta + (m+\alpha)f)\ch_1(A) - l \ch_0(A)\Big),
\]
which lies in $\tfrac{\beta}{\alpha}\mathbb{Z}$ under our assumption that both $m+\alpha$ and $l$ are integers.  In particular, the minimum of $\{ \olw \ch_1^\olB (F) >0 : F \in \Coh (X)\}$ exists, is equal to $\tfrac{\beta}{\alpha}$ and is attained when $F=\whPhi \OO_x$.  By Lemma \ref{lem:AG47-123-1}, $\whPhi \OO_x$ is stable with respect to  $(Z_{\olw,\olB},\Bc_{\olw,\olB})$ as wanted.
\end{proof}

\begin{prop}\label{prop:AG48-48-1}
With notation as in \ref{para:vchoice}, suppose    $m+\alpha$ and $l$ are both integers.   Let $\Pc$ denote the $\mathbb{R}$-valued slicing of the Bridgeland stability condition $(Z_{\omega,B},\Phi \Bc_{\olw,\olB} [1])$.  Then $\Pc (0,1]=\Bc_{\omega,B}$.
\end{prop}

\begin{proof}
By Proposition \ref{prop:AG48-52-1}(b), every skyscraper sheaf $\OO_x$ on $X$ is stable of phase 1 with respect to the Bridgeland stability condition $(Z_{\omega,B},\Phi \Bc_{\olw,\olB} [1])$.  On the other hand, $(Z_{\omega,B},\Phi \Bc_{\olw,\olB} [1])$ satisfies the support property since $(Z_{\olw,\olB}, \Bc_{\olw,\olB})$ does and by  the argument in  \ref{para:Zwpreimag-1}.  Then by a standard argument (see \cite[Lemma 6.20]{MSlec}, \cite[Theorem 3.2]{huybrechts2014introduction} or \cite[Proposition 10.3]{SCK3}), we must have $\Pc (0,1]=\Bc_{\omega,B}$.
\end{proof}

\begin{thm}\label{thm:Bristabcorr}
With notation as in \ref{para:vchoice}, suppose    $m+\alpha$ and $l$ are both integers.    Consider the Bridgeland stability conditions  $\sigma_{\olw,\olB}=(Z_{\olw,\olB},\Bc_{\olw,\olB})$ and $\sigma_{\omega,B}=(Z_{\omega,B}, \Bc_{\omega,B})$ on $D^b(X)$, both of which are with respect to $(0,1]$.  Then
\begin{equation}\label{eq:M30-2}
  ([1]\circ \Phi)\cdot \sigma_{\olw,\olB}= \sigma_{\omega,B} \cdot (T,g)
\end{equation}
where $(T,g)$ is an element of $\wt{\mathrm{GL}}^+\!(2,\RR)$ with $T=\begin{pmatrix} \tfrac{\alpha}{\beta} & 0 \\
  0 & u \end{pmatrix}\begin{pmatrix} 0 & 1 \\
  -1 & 0 \end{pmatrix}$ and $g(x)=x-\tfrac{1}{2}$ for $x \in \tfrac{1}{2}\ZZ$.
 In particular, for any $E\in D^b(X)$,
\[
  \text{$E$ is $(Z_{\olw,\olB},\Bc_{\olw,\olB})$-semistable if and only if $\Phi E$ is $(Z_{\omega,B}, \Bc_{\omega,B})$-semistable.}
\]
The statement also holds if we replace `semistable' with `stable.'  \end{thm}

\begin{proof}
Let us use Configuration III.  Using the notation of Proposition \ref{prop:AG48-48-1}, we have $\Phi \Bc_{\olw,\olB}[1] = \Pc (\tfrac{1}{2},\tfrac{3}{2}]$, i.e.\ $\Phi \Bc_{\olw,\olB} \subset D^{[0,1]}_{\Bc_{\omega, B}}$.  If we choose $\Dc=\Uc=D^b(X)$, $\Psi = \whPhi$, $\Ac = \Bc_{\olw,\olB}, \Bc=\Bc_{\omega, B}$, $Z_\Ac=Z_{\olw,\olB}, Z_{\Bc}=Z_{\omega,B}$, $a=b=0$, $T=\begin{pmatrix} \tfrac{\alpha}{\beta} & 0 \\
  0 & u \end{pmatrix}\begin{pmatrix} 0 & 1 \\
  -1 & 0 \end{pmatrix}$ and $\Gamma_T=g : \RR \to \RR$ such that $g(x)= x-\tfrac{1}{2}$ for $x \in \tfrac{1}{2}\ZZ$, then \eqref{eq:AG45-108-17} follows from \eqref{eq:L2}, and we are in Configuration III.  The second  half of  the theorem  then follows from  Theorem \ref{prop:paper30prop11-11ext}, part (iv), while \eqref{eq:M30-2} follows from part (v) of the same theorem.
  \end{proof}

In the following two sections, we present several applications of Theorem \ref{thm:Bristabcorr}.

\section{Applications: wall-crossing for Bridgeland stability}\label{sec:wcBristab}

In this  section, $p : X \to B$ will be a Weierstra{\ss} elliptic surface unless otherwise stated.  We give the first series of applications of Theorem \ref{thm:Bristabcorr}, including wall-crossing for Bridgeland stability conditions, and relations among Bridgeland stability, polynomial stability and twisted Gieseker stability for sheaves.

\paragraph[Notation] We will continue \label{para:vchoice2}   the notation from \ref{para:ellipsurfnotation} where $m, \alpha \in \mathbb{R}_{>0}$ are fixed such that $\Theta + kf$ is ample on $X$ for all $k \geq m$ (which implies $m >e$).  Regarding $u, \beta$ as functions in $v$ as in \ref{para:surfsolveeqs}, let us fix a $v_0 \in \mathbb{R}_{>0}$ such that, on the region $v>v_0$, the functions $u,\beta$ are both positive and their Laurent series representations in $v$ are convergent.  Now fix any $v>v_0$; then $u, \beta \in \mathbb{R}_{>0}$ are also fixed, and $\wt{\omega}, \olw, \omega$ are all fixed ample divisors on $X$.  Let us also fix $\mathbb{R}$-divisors $\ol{B}, B \in \mathbb{R}f$ on $X$ such that \eqref{eq:lqlrelation} holds.  Then for any $E \in D^b(X)$, we have the relation \eqref{eq:L2}
\begin{equation*}
Z_{\omega, B}(\Phi E) =  \begin{pmatrix} \tfrac{\alpha}{\beta} & 0 \\ 0 & u \end{pmatrix} (-i)  Z_{\ol{\omega}, \ol{B}}(E).
\end{equation*}

\paragraph[Mini-walls and Bridgeland stability] Using the author's joint  work with Qin on mini-walls for Bridgeland stability \cite{LQ}, we can now complete the following diagram relating Bridgeland stability conditions and polynomial stability conditions on a Weierstra{\ss} elliptic surface:

\begin{thm}\label{thm:stabBriNlim}
Suppose   $m, \alpha\in\mathbb{Q}_{>0}$ and ample $\mathbb{R}$-divisors $\olw, \omega$ are as in \ref{para:ellipsurfnotation}, and   $m+\alpha, l$ are integers.    Suppose $\olB, B \in \mathbb{R}f$ are divisors satisfying \eqref{eq:lqlrelation}.  Let $(Z_l,\Coh^p)$, $(Z^l,\Bl)$ be polynomial stability conditions defined in \ref{para:surfZlowerl} and \ref{para:surfZupperl}, respectively.  For Bridgeland stability conditions $(Z_{\olw,\olB},\Bc_{\olw,\olB})$ and $(Z_{\omega,B},\Bc_{\omega,B})$, use the notation in \ref{para:vchoice2}.

Then for any fixed Chern character $\ch$, there exists $v_1 \in \mathbb{R}_{>0}$ such that, for every $v \geq v_1$ and $E \in D^b(X)$ of Chern character $\ch$, we have the equivalences
\begin{equation}\label{eq:stabequivs}
\xymatrix{
  E \text{ is $(Z_l,\Coh^p)$-semistable} \ar@{<=>}[r] \ar@{<=>}[d] & \Phi E \text{ is $(Z^l,\Bl)$-semistable} \ar@{<=>}[d]  \\
  E \text{ is $(Z_{\olw,\olB},\Bc_{\olw,\olB})$-semistable} \ar@{<=>}[r] & \Phi E \text{ is $(Z_{\omega,B},\Bc_{\omega,B})$-semistable}
}.
\end{equation}
\end{thm}

\begin{proof}
The upper horizontal equivalence follows immediately from Theorem \ref{thm:main5}.  For the other equivalences, let $v_0>0$  be fixed as in \ref{para:vchoice2}.  Then the  lower horizontal equivalence holds for any $v \geq v_0$ by Theorem \ref{thm:Bristabcorr}, which uses the assumption $m+\alpha, l \in \mathbb{Z}$.  With $m, \alpha \in \mathbb{Q}_{>0}$, we have $\omega \in \mathrm{NS}_{\mathbb{Q}}(X)$ and by \cite[Theorem 4.4]{LQ}, there exists $v_1 \geq v_0$  such that, for any fixed $v$ satisfying $v\geq v_1$, the left vertical equivalence holds.  The right vertical equivalence then follows.
\end{proof}

\paragraph[Mini-walls along \eqref{eq:surfcons2}] In our \label{para:def-miniwallsv} notation in \ref{para:ellipsurfnotation}, we write $\olw = \beta \wt{\omega}$ where $\wt{\omega} = \tfrac{1}{\alpha}(\Theta + mf)$ is a fixed ample class and $\beta \in \mathbb{R}_{>0}$.  As $\beta$ varies, we obtain a  1-parameter family of Bridgeland stability conditions  $\{ (Z_{\olw,\olB}, \Bc_{\olw,\olB}) : \beta >0\}$.  For any fixed Chern character $\ch$, it was proved in \cite[Theorem 1.1]{LQ} that the mini-walls for this 1-parameter family of Bridgeland stability conditions are locally finite.

On the other hand, we also write $\omega = u(\Theta + (m+\tfrac{v}{u})f)$ and, from \ref{para:vchoice2}, we know there exists $v_0>0$ such that, on the region $v>v_0$, we can regard $u$ as a positive analytic function in $v$ subject to the relation \eqref{eq:surfcons2}.  This gives us another 1-parameter family of  Bridgeland stability conditions $\{ (Z_{\omega,B}, \Bc_{\omega,B}) : v\in (v_0, \infty) \}$.  For any fixed Chern character $\ch$, we can then define mini-walls  in the interval $(v_0, \infty)$ in the same way as in \cite[Definition 3.2]{LQ}.   At this point, it is  natural to ask whether an analogue of \cite[Theorem 1.1]{LQ} holds for this family of Bridgeland stability conditions, i.e.\ whether  boundedness and local finiteness  of mini-walls hold.  We give a positive answer below using Theorem  \ref{thm:Bristabcorr}:

\begin{thm}\label{cor:AG48-58-1}
Suppose   $m \in \mathbb{Q}_{>0}$  is as in \ref{para:ellipsurfnotation}, and $\omega = u(\Theta + (m+ \tfrac{v}{u})f)$ where we regard $u$ as a  function in $v$ as in \ref{para:surfsolveeqs}.  Let  $B=qf$ where $q \in \mathbb{Z} + \tfrac{e}{2}$.  Fix a Chern character $\ch$. Let $v_1$ be as in Theorem \ref{thm:stabBriNlim}. Then there exists $v_2 \in \mathbb{R}_{>0}$ such that the following statements hold for mini-walls for $\{ (Z_{\omega,B}, \Bc_{\omega,B}) : v\in (v_0, \infty) \}$:
\begin{itemize}
\item[(i)] (local finiteness) The set of mini-walls for $\ch$ is locally finite in $(v_2, \infty)$.
\item[(ii)] (boundedness) There are no mini-walls for $\ch$ in $[v_1,\infty)$.
\end{itemize}
\end{thm}

\begin{proof}
In terms of our notation $\olB = lf$, the condition $B \in (\mathbb{Z} + \tfrac{e}{2})f$ means that if $\olB, B$ satisfy \eqref{eq:lqlrelation}, then $l$ is forced to be an integer. Now, choosing $\alpha \in \mathbb{Q}_{>0}$ so that $m+\alpha \in \mathbb{Z}$ and choosing $\omega$ as in \ref{para:vchoice2}, Theorem \ref{thm:stabBriNlim} applies (for suitable $\olw$ and $\olB$) and the right vertical equivalence in \eqref{eq:stabequivs} gives part (ii).

 For part (i), let $v_0$ be as in \ref{para:vchoice2} so that on the region  $v\geq v_0$, the functions  $u, \beta$ in $v$ both take on positive values, and the lower horizontal equivalence in \eqref{eq:stabequivs} holds.  By Lemma \ref{lem:betaderpos}, there exists $v' \in \mathbb{R}_{>0}$ such that $\beta$ is monotone increasing for $v \in [v',\infty)$.  Let $v_2 = \max{\{v_0, v'\}}$ and take any closed interval $I \subset (v_2,\infty)$.  Let $W$ be the set of mini-walls in $(v_2,\infty)$ for the family of Bridgeland stability conditions $\{ (Z_{\olw,\olB},\Bc_{\olw,\olB}) : v \in (v_2,\infty)\}$.  Then $W \cap I$ is finite by \cite[Theorem 4.5(i)]{LQ}. Hence for any $v'' \in I\setminus W$, the set of $Z_{\olw,\olB}$-semistable objects of Chern character $\whPhi^{\ch}(\ch)$ is  constant as $v$ varies in the connected component of $I\setminus W$ containing $v''$.  By the lower horizontal equivalence in  \eqref{eq:stabequivs}, the set of $Z_{\omega,B}$-semistable objects of Chern character $\ch$ is also constant as $v$ varies within the same connected component.  This means that the set of mini-walls for the family $\{ (Z_{\omega,B}, \Bc_{\omega,B}) : v \in (v_2,\infty)\}$ is contained in $W$, and hence is  locally finite.
\end{proof}

In the author's joint work with Liu and Martinez \cite[p.30]{LLM}, mini-walls along the curve \eqref{eq:surfcons2} were studied (including using computational techniques such as from \cite{liu2018bayer}).  The boundedness of mini-walls   was shown directly for the Chern characters of Fourier-Mukai transforms of 1-dimensional sheaves.  Theorem \ref{cor:AG48-58-1} generalises the result to all Chern characters in the form of local finiteness and boundedness.

\paragraph[Twisted Gieseker semistable sheaves]  The boundedness of mini-walls in \cite{LLM} allowed  Liu, Martinez and the author  to show that twisted Gieseker semistable 1-dimensional sheaves are sent by $\Phi$ to Bridgeland semistable objects:

\begin{prop} \cite[Corollary 6.14]{LLM}\label{prop:LLM1}
Let $L$ be a $\mathbb{Q}$-line bundle with $\ch_1(L)=\tfrac{1}{2}p^\ast K_B$.  Suppose     $E$ is a 1-dimensional twisted Gieseker semistable sheaf with $\chi_L (E)\geq 0$ and $f\ch_1 (E)>0$ and $m+\alpha \gg 0$.  Then for  $v \gg 0$ under the relation \eqref{eq:surfcons2}, the transform $\Phi E$ is $Z_\omega$-semistable.
\end{prop}

We  can now show  not only an embedding of moduli spaces, but an isomorphism of moduli spaces:

\begin{cor}\label{cor:AG48-55-1}
Suppose   $m, \alpha\in\mathbb{Q}_{>0}$ and $\wt{\omega}, \omega$ are as in \ref{para:vchoice2} and $m+\alpha, l\in \mathbb{Z}$.  Suppose $\olB, B\in \mathbb{R}f$ also satisfy \eqref{eq:lqlrelation}, and $L$ is a line bundle  on $X$ such that $\olB = \ch_1(L^\ast) + \tfrac{1}{2}K_X$.   Then for any fixed Chern character $\ch$ of a 1-dimensional sheaf, if  $v_1 \in \mathbb{R}_{>0}$ is as in Theorem \ref{thm:stabBriNlim} then  for any $v\geq v_1$ and $E\in D^b(X)$ of Chern character $\ch$ we have
\begin{equation*}
  E \text{ is an $L$-twisted $\wt{\omega}$-Gieseker semistable sheaf}
  \Leftrightarrow \Phi E \text{ is $(Z_{\omega,B}, \Bc_{\omega,B})$-semistable}.
\end{equation*}
\end{cor}

\begin{proof}
Since $K_X \in \mathbb{R}f$ \cite[Proposition III.1.1]{MirLec}, for any $\olB$ satisfying the hypotheses, we can find a line bundle $L$ satisfying $\olB = \ch_1(L^\ast) + \tfrac{1}{2}K_X$ and $\ch_1(L)$ would be a multiple of the fiber class and hence $\ch_1(L)K_X=0$, and so   Corollary \ref{lem:AG48-25-1} applies.  On the other hand, all the hypotheses of Theorem \ref{thm:stabBriNlim} are also satisfied, and the equivalence between the upper left and the lower right corners in \eqref{eq:stabequivs}, together with the equivalence in Corollary \ref{lem:AG48-25-1}, implies this corollary.
\end{proof}

\paragraph When $e$ is even, we can put $l=\tfrac{e}{2}\in \mathbb{Z}$ and $q=0$, in which case the condition $\olB = \ch_1(L^\ast)+\tfrac{1}{2}K_X$  in Corollary \ref{cor:AG48-55-1} reduces to the condition $\ch_1(L) = \tfrac{1}{2}p^\ast K_B$ in Proposition \ref{prop:LLM1} via the formula $K_X = p^\ast K_B + ef$ \cite[2.3]{LLM}.  For these particular parameters, Corollary \ref{cor:AG48-55-1} generalises Proposition \ref{prop:LLM1}.

\section{Applications: action on stability manifold}\label{sec:actionStab}

In this  section, we give the second series of applications of Theorem \ref{thm:Bristabcorr}, covering the preservation of a connected component of $\mathrm{Stab}(X)$ under the autoequivalence group, and solutions to Gepner equations.

\paragraph[Preservation of connected component]  For a smooth projective surface $X$, let $\Stabd (X)$ denote the connected component of $\Stab (X)$, the stability manifold of $X$, containing the geometric stability conditions, i.e.\ the Bridgeland stability conditions  for which skyscraper sheaves of points are stable of the same phase.  When $X$ is a K3 surface,    Bridgeland conjectured that the autoequivalence group of a K3 surface preserves the connected component $\Stabd(X)$ \cite[Conjecture 1.2]{SCK3}. For K3 surfaces of Picard rank 1, this conjecture was proved by Bayer-Bridgeland \cite[Theorem 1.3]{bayer2017derived}. We now prove an analogue of this conjecture when $X$ is an elliptic surface of nonzero Kodaira dimension:

\begin{thm}\label{thm:AG48-67-1}
Let $p : X \to B$ be a Weierstra{\ss} elliptic surface.
\begin{itemize}
\item[(i)]  The autoequivalence $\Phi$ preserves the connected component $\Stabd(X)$.
\item[(ii)]  If the Kodaira dimension of $X$ is nonzero, then  $\mathrm{Aut}(D^b(X))$ preserves $\Stabd(X)$.
\end{itemize}
\end{thm}

\begin{proof}
(i) This part follows immediately from  Theorem \ref{thm:Bristabcorr} and the fact that any Bridgeland stability condition of the form $(Z_{\omega',B'},\Bc_{\omega',B'})$, where $\omega', B'$ are $\mathbb{R}$-divisors with $\omega'$ ample, lies in $\Stabd (X)$.

(ii) Since our Weierstra{\ss} elliptic surface $p : X \to B$ has a section, the greatest common divisor $\lambda_X$ of all fiber degrees of objects in $D^b(X)$ is 1.  Together with our assumption on the Kodaira dimension of $X$, we have a short exact sequence of groups
\[
  1 \to \langle \OO_X (D) : fD=0\rangle \rtimes \mathrm{Aut}(X) \times \mathbb{Z}[2] \to \mathrm{Aut}(D^b(X)) \overset{s}{\to} \mathrm{SL}(2,\mathbb{Z}) \to 1
\]
by \cite[Theorem 4.1, Remark 1.2(iii)]{uehara2015autoequivalences} (see also \cite{martin2013relative}), where we write $s$ for the induced map on $H^{\text{even}}(C,\mathbb{Z})$ of a smooth fiber $C$ of $p$.  We can write $H^{\text{even}}(C,\mathbb{Z})=\{\begin{pmatrix}\ch_0(E)\\ \ch_1(E)\end{pmatrix} : E \in D^b(F) \}$, so that  $s$ takes an autoequivalence to its induced  map on Chern characters.

Clearly, tensoring by line bundles on $X$, automorphisms of $X$ and the shift functor preserve the connected component $\Stabd(X)$.  If we write $T$ to denote tensoring by the line bundle $\OO_X(\Theta)$, then $s(T)=\begin{pmatrix} 1 & 0 \\ 1 & 1 \end{pmatrix}$ while $s(\Phi)= \begin{pmatrix} 0 & 1 \\ -1 & 0 \end{pmatrix}$.  Since the images of $T$ and $\Phi$ under $s$ generate $\mathrm{SL}(2,\mathbb{Z})$, and $\Phi$ preserve the connected component $\Stabd(X)$ by part (i), part (ii) follows.
\end{proof}

\paragraph[Gepner equations]  Let $\Dc$ be a  triangulated category and $\mathrm{Stab}(\Dc)$ the Bridgeland stability manifold on $\Dc$. A Gepner equation on $\mathrm{Stab}(\Dc)$ is an equation of the form
\[
  \Phi \cdot\sigma =  \sigma \cdot (s)
\]
where $(\Phi,s) \in  \Aut (\Dc)\times \mathbb{C}$ and  $\sigma \in \mathrm{Stab}(\Dc)$.  If we write $\sigma = (Z,\Pc)$ where $\Pc$ is the slicing, then  by $\sigma\cdot (s)$ we mean the Bridgeland stability condition $(Z',\Pc')$ where $Z'= e^{-i\pi s}Z$ and $\Pc'(x) = \Pc (x + \Re (s))$ for all $x \in \mathbb{R}$.  Following Toda \cite[Definition 2.3]{toda2013gepner}, a Bridgeland stability condition $\sigma$ satisfying a Gepner equation as above is said to be of Gepner type with respect to $(\Phi, s)$.

\paragraph For our next theorem, we need the following facts: Given a smooth projective surface $X$ and $\mathbb{R}$-divisors $\omega, B$ on $X$ where $\omega$ is ample, $\mu_{\omega, B}$-stability for coherent sheaves on $X$ is independent of the divisor $B$ and is invariant under tensoring with any line bundle $L$ on $X$.  In addition, for any coherent sheaf $E$  on $X$ we have
\[
  \mu_{\omega, B'}(E\otimes L) = \mu_{\omega, B}(E) + \omega (B-B'+\ch_1(L))
\]
for any $\mathbb{R}$-divisor $B'$ on $X$.  Putting all these together, if $L$ is a line bundle satisfying $\omega B'=\omega (B+\ch_1(L))$, then $\{ E \otimes L : E \in \Bc_{\omega, B}\} = \Bc_{\omega, B+c_1(L)}$.  On the other hand, for any $E \in D^b(X)$, we have
\[
  Z_{\omega, B}(E \otimes L^\ast) =-\int_X e^{-(B+c_1(L))}\ch (E) = Z_{\omega, B+c_1(L)}(E).
\]
In summary, if we write $\otimes L$ to denote the functor of (derived) tensoring with $L$ and let $\sigma_{\omega, B}$ denote the Bridgeland stability condition $(Z_{\omega, B}, \Bc_{\omega, B})$ with respect to $(0,1]$, then
\begin{equation}\label{eq:M38-1}
(\otimes L) \cdot \sigma_{\omega, B} = \sigma_{\omega,B+\ch_1(L)}.
\end{equation}

\paragraph[Gepner-type stability on  elliptic surfaces] Let \label{para:fixedpoints} $p : X \to B$ be a Weierstra{\ss} elliptic surface.  Fix $m,\alpha \in \mathbb{R}_{>0}$ as in  \ref{para:ellipsurfnotation}. Then  $m>e$ and so $m+\alpha-\tfrac{e}{2}>0$.  In this case, if we choose
\begin{equation}\label{eq:M28-2}
 u = \sqrt{\frac{m+\alpha - e}{m+\alpha-\tfrac{e}{2}}}, \text{\quad}\beta = \alpha u, \text{\quad} v=\beta
\end{equation}
then \eqref{eq:surfcons1} and \eqref{eq:surfcons2} hold and, in particular, we have $\tfrac{v}{u}=\alpha, \tfrac{\beta}{\alpha}=u$ and $\olw=\omega$.  If we further choose $l, q \in \mathbb{R}$ so that they satisfy \eqref{eq:lqlrelation}, then we have a solution to \eqref{eq:L2} with $\olw=\omega$.  This leads  us to the following result, which in turn gives us  solutions to  Gepner equations  on  elliptic surfaces:

\begin{thm}\label{thm:Gepnersol}
Suppose $p : X \to B$ is a  Weierstra{\ss} elliptic surface, and $L$ is a line bundle on $X$ such that $\ch_1(L)=\tfrac{e}{2}f$.  Let $m \in \mathbb{R}_{>0}$ be as in  \ref{para:ellipsurfnotation} and $\alpha \in \mathbb{R}_{>0}$ be such that $m+\alpha \in \mathbb{Z}^+$, and let $q \in \mathbb{Z}$.  Set  $u = \sqrt{\frac{m+\alpha-e}{m+\alpha-(e/2)}}$,  $\omega = u(\Theta + (m+\alpha)f)$ and $B=qf$.  Then  the Bridgeland stability condition $\sigma = (Z_{\omega, B}, \Bc_{\omega, B})$ is a solution to the  equation
\begin{equation}\label{eq:M38-2}
  (\otimes L \circ [1] \circ \Phi )\cdot \sigma=  \sigma \cdot (T,g)
\end{equation}
where  $\otimes L$  denotes  tensoring with $L$, and $(T,g)$ is the element  in $\wt{\mathrm{GL}}^+\!(2,\mathbb{R})$ where $T=\begin{pmatrix} \tfrac{1}{u} & 0 \\ 0 &  u \end{pmatrix} \begin{pmatrix} 0 & 1 \\  -1 & 0 \end{pmatrix}$ and $g(x)=x-\tfrac{1}{2}$ for $x \in \tfrac{1}{2}\mathbb{Z}$.
\end{thm}

\begin{proof}
Let use the notation in \ref{para:ellipsurfnotation} and \ref{para:L3p-relation}, and further assume  $m, \alpha \in \mathbb{R}_{>0}$ satisfy $m+\alpha\in \mathbb{Z}^+$.  Choose $u, \beta, v$ as in \eqref{eq:M28-2}, which ensures $\olw=\omega$.  Then  choose $\olB, B \in \mathbb{R}f$ to satisfy \eqref{eq:lqlrelation}, i.e.\ $\olB - B = \tfrac{e}{2}f=\ch_1(L)$.  With $(T,g)$ as described, we  have
\begin{align*}
  (\otimes L \circ [1] \circ \Phi ) \cdot \sigma_{\olw,\olB} &= (\otimes L) \cdot \sigma_{\omega, B} \cdot (T,g)  \text{\quad by \eqref{eq:M30-2}} \\
  &= \sigma_{\olw, \olB}\cdot (T,g) \text{\quad by \eqref{eq:M38-1} and $\olw=\omega$}.
\end{align*}
\end{proof}

\begin{eg}[Gepner points]\label{eg:Gepnerpoints}  Assume the setting of Theorem \ref{thm:Gepnersol}.   Note that $T^2=-1$, and so \eqref{eq:M38-2} gives
\[
(\otimes L \circ [1] \circ \Phi)^2\cdot  \sigma = \sigma \cdot (1).
\]
If we specialise to the case where $e=0$ (i.e.\ when $X$ is a product elliptic surface), then $u=1$ and $T$ corresponds to multiplication by $-i$.  In this case,  \eqref{eq:M38-2} can be rewritten  as
\[
 ([1] \circ \Phi )\cdot\sigma= \sigma\cdot \left( -\tfrac{1}{2}\right).
\]
\end{eg}

\bibliography{refs}{}
\bibliographystyle{plain}

\end{document}